\documentclass[12pt, english]{article} % !TEX TS-program = LaTeX

\usepackage{color}
\usepackage[margin=0.5in]{geometry} % set the margins to 1in on all sides
\usepackage{graphicx} % to include figures
\usepackage{amsmath} % great math stuff
\usepackage{amsfonts} % for blackboard bold, etc
\usepackage{amsthm} % better theorem environments
\usepackage{amssymb}
\usepackage{mathrsfs}
\usepackage{mathtools}
\usepackage{enumitem}
\usepackage{babel}
\usepackage{hyperref}
\newtheorem{thm}{Theorem}[section]
\newtheorem{lem}[thm]{Lemma}
\newtheorem{hyp}[thm]{Hypothesis}
\newtheorem{prop}[thm]{Proposition}

\newtheorem{remark}[thm]{Remark}
\newtheorem{example}[thm]{Example}
\newcommand{\de}{\mathrm{d}} 
\DeclareMathAlphabet{\mathpzc}{OT1}{pzc}{m}{it}
\newcommand{\avint}[1]{\mathop{\rlap{\tiny  ${\hspace{0.15em}{\diagup}}$}\!{\int_{#1}}}\nolimits}
\providecommand{\keywords}[1]{\textbf{\textit{Keywords: }}#1}

\title{Propagation of Chaos for Stochastic Spatially Structured Neuronal Networks with Delay driven by Jump Diffusions}
\author{Sima Mehri\footnotemark[1]\ \footnotemark[2]\ \footnotemark[3]\and Michael Scheutzow\footnotemark[1] \and Wilhelm Stannat\footnotemark[1]\and Bijan Z. Zangeneh\footnotemark[2]}

\begin{document}

\maketitle
\renewcommand{\thefootnote}{\fnsymbol{footnote}}
\footnotetext[1]{Institut f\"ur Mathematik, Technische Universit\"at Berlin, D-10623 Berlin, Germany}
\footnotetext[2]{Department of Mathematical Sciences, Sharif University of Technology, Tehran, Iran}
\footnotetext[3]{The work of this author was supported by the Hilda Geiringer Scholarship awarded by the Berlin Mathematical School}
\renewcommand{\thefootnote}{\arabic{footnote}}

\begin{abstract}
Spatially structured neural networks driven by jump diffusion noise with monotone coefficients, 
fully path dependent delay and with a disorder parameter are considered. Well-posedness for the 
associated McKean-Vlasov equation and a corresponding propagation of chaos result in the infinite 
population limit are proven. Our existence result for the McKean-Vlasov equation is based on the Euler approximation, that is applied to this type of equation for the first time.  
\end{abstract}


%\tableofcontents

\bigskip

\section{Introduction and main results} 

The purpose of this paper is to prove a unified existence and uniqueness result for the 
McKean-Vlasov equation and a propagation of chaos result for a spatially structured coupled neural 
network of neural oscillators in the large population limit. Our mathematical framework covers all 
relevant modeling issues of networks of point neurons. In particular, we incorporate noise 
terms, both in the local dynamics of the neurons as well as in the synaptic transmission, in order to 
account for channel noise and synaptic noise in the neural dynamics. We also consider general 
delay terms modeling finite and variable propagation speed of neural signals. In order to cover 
all conductance-based neural oscillators and all types of delays being widely accepted in 
computational neuroscience, we consider stochastic delay differential 
equations with merely monotone coefficients. As stochastic forcing terms for 
the local dynamics and the synaptic transmission between individual neurons, we allow jump 
diffusions given as the independent sum of Brownian motions and Poisson processes.

We also incorporate spatial structure into our networks to take into account morphological 
properties of brain tissues. With a view towards spatial continuum limits, we consider 
the positions of the neurons as discrete subsets in a bounded subset $\Gamma\subset\mathbb R^k$, 
and introduce spatial dependence in the network dynamics in terms of an additional space 
parameter. We will then be in particular interested in the dynamical properties of the network in 
the infinite population limit, where the spatial distribution of the neurons is given in terms of 
a general Borel-measure on $\Gamma$.

With a view towards modeling brain networks, consisting of subpopulations of neurons, we also 
introduce a measurable partition of $\Gamma := \bigcup_{1\le \alpha\le P} \Gamma_\alpha $ and 
consider $\Gamma_\alpha$ as a given subpopulation. 

In order to incorporate variability in the neurons, and henceforth the associated neural dynamics, 
we finally introduce disorder in terms of a random parameter ${\omega^\prime}$. 

The above mentioned modeling issues lead to the following system of coupled delay-differential equations:  
\begin{equation}
\label{equ1}
\begin{split}
\de X^{r,\mathcal{A}_N}_t 
& = f(t,r,X_{t^-}^{r,\mathcal{A}_N},{\omega^\prime})\de t +g(t,r,X^{r,\mathcal{A}_N}_{t^-},{\omega^\prime})\de W^r_t 
+\int_{U}h(t,r,X^{r,\mathcal{A}_N}_{t^-},{\omega^\prime},\xi)\tilde{N}^r\left(\de t,\de \xi\right)\\ 
& \quad + \sum_{\alpha=1}^P\frac{1}{\mathcal{S}_{\mathcal{A}_N,\alpha}}\sum_{\tilde{r}\in \mathcal{A}_N\cap 
\Gamma_\alpha}\theta\left(t,r,\tilde{r},X^{r,\mathcal{A}_N}_{t^-},X^{\tilde{r},\mathcal{A}_N}_{(t-\tau)^-:{t^-}},
{\omega^\prime}\right)\de t\\ 
& \quad +\sum_{\alpha=1}^P\frac{1}{\mathcal{S}_{\mathcal{A}_N,\alpha}}\sum_{\tilde{r}\in \mathcal{A}_N\cap 
\Gamma_\alpha }\beta\left(t,r,\tilde{r},X^{r,\mathcal{A}_N}_{t^-}, X^{\tilde{r},\mathcal{A}_N}_{(t-\tau)^-:{t^-}},
{\omega^\prime}\right)\de B^{r,\alpha}_t\\ 
& \quad + \sum_{\alpha=1}^P\frac{1}{\mathcal{S}_{\mathcal{A}_N,\alpha}}\int_{U}\sum_{\tilde{r}\in \mathcal{A}_N
\cap \Gamma_\alpha }\eta\left(t,r,\tilde{r},X^{r,\mathcal{A}_N}_{t^-},X^{\tilde{r},\mathcal{A}_N}_{(t-\tau)^-:{t^-}},
{\omega^\prime},\xi\right)\tilde{N}^{r,\alpha}\left(\de t,\de \xi\right), \quad r\in \mathcal{A}_N
\\ 
X^{r,\mathcal{A}_N}_t & = z_t^r, \quad t\in [-\tau,0],\  r\in \mathcal{A}_N
\end{split}
\end{equation} 
for the time-evolution of the network. 

Here, $P$ is the number of different subpopulations, placed in space at disjoint measurable 
regions $\Gamma_\alpha$, $1\le \alpha\le P$, such that $\Gamma := \bigcup_{1\le\alpha\le P} 
\Gamma_\alpha $ is a bounded subset in $\mathbb R^k$. We suppose that the total number of neurons 
is $N$ and $\mathcal{A}_N\subset \Gamma$ denotes the position of neurons. In the equation \eqref{equ1}, the weight of neurons in subpopulation $\alpha$ is $1/\mathcal{S}_{\mathcal{A}_N,\alpha}$ with $\mathcal{S}_{\mathcal{A}_N,\alpha}\neq 0$. Finally, $X^{r,\mathcal{A}_N}_t\in \mathbb{R}^d$ 
denotes the state of neuron at position $r\in \mathcal{A}_N$ and at time $t\ge 0$.

The disorder ${\omega^\prime}$ is an element of a second probability space $\left(\Omega^\prime,\mathcal{F}^\prime, \mathbb{P}^\prime\right)$. We will denote with $\mathcal{E}$ the expectation with respect to $\mathbb{P}^\prime$.  

$\tau > 0$ is a fixed deterministic delay in the synaptic transmission between different neurons. 
For any path $y_t$ defined on some subinterval $I\subset \mathbb R$ we also use the following notation $y_{(t-\tau):t}$ for the path segment $y_s$, $s\in [t-\tau , t]$, for any $t$ with $[t-\tau , t]\subset I$. $W^r$ and $B^{r,\alpha}$, $r\in \mathcal{A}_N$, $1\le \alpha\le P$, are independent Brownian motions respectively in $\mathbb{R}^m$ and $\mathbb{R}^n$. 
$N^r$ and $N^{r,\alpha}$, $r\in \mathcal{A}_N$, $1\le \alpha\le P$, are independent time homogeneous Poisson measures on $[0,\infty)\times U$ with intensity measure 
$\de t \otimes \nu$, where $(U,\mathcal{U}, \nu)$ is an arbitrary $\sigma$-finite measure space.
$\tilde{N}^r = N^r - \de t \otimes \nu$ and $\tilde{N}^{r,\alpha}={N}^{r,\alpha}- \de t \otimes \nu$, 
$r\in \mathcal{A}_N$, $1\le \alpha\le P$, denote the associated compensated Poisson martingale 
measures.

The coefficients 
$$ 
\begin{aligned} 
f,g & : [0, \infty [\times \Gamma\times \mathbb{R}^d \times \Omega^\prime\to \mathbb{R}^d, \mathbb{R}^{d\times m} \\ 
h & : [0, \infty [\times \Gamma\times \mathbb{R}^d \times \Omega^\prime\times U\to \mathbb{R}^d \\ 
\theta,\beta & :  [0, \infty [\times \Gamma\times\Gamma\times\mathbb{R}^d \times \text{C\`agl\`ad}\,([-\tau,0];\mathbb{R}^d)\times \Omega^\prime \to \mathbb{R}^d, \mathbb R^{d\times n}\\
\eta & :  [0, \infty [\times \Gamma\times\Gamma\times\mathbb{R}^d \times \text{C\`agl\`ad}\,([-\tau,0];\mathbb{R}^d)\times \Omega^\prime\times U\to \mathbb{R}^d
\end{aligned} 
$$ 
are jointly measurable w.r.t. all variables and continuous w.r.t. $x\in\mathbb{R}^d$. Here, 
the space $\text{C\`adl\`ag}\,([-\tau,0];\mathbb{R}^d)$ is endowed with the supremum norm. 
We will specify appropriate monotonicity and growth conditions on the coefficients below (see 
Hypothesis \ref{hyp1}). 

The initial conditions $z^r$ for $r\in \Gamma_\alpha$
are independent and identically distributed copies of $\hat{z}^\alpha$ with $\hat{z}^\alpha\in 
L^2(\Omega,\mathbb{P};\text{C\`adl\`ag}\,([-\tau,0];\mathbb{R}^d))$. The space 
$\text{C\`agl\`ad}\,([-\tau,0];\mathbb{R}^d)$ being again endowed with the supremum norm.   
We assume that $\left\lbrace W^r\right\rbrace_{r\in \mathcal{A}_N}$, 
$\left\lbrace B^{r,\alpha}\right\rbrace_{\substack{r\in \mathcal{A}_N \\ 1\le \alpha\le P}}$, 
$\left\lbrace N^r\right\rbrace_{r\in \mathcal{A}_N}$, 
$\left\lbrace N^{r,\alpha}\right\rbrace_{\substack{r\in \mathcal{A}_N\\ 1\le \alpha\le P}}$ 
and the initial conditions $\left\lbrace z^r\right\rbrace_{r\in \mathcal{A}_N}$ 
are independent.

The reason for assuming the coefficients $\theta$, $\beta$ and $\eta$ as being defined on  
c\'agl\'ad-processes stems from our choice of an Euler approximation method \eqref{df Xn-} 
for the solution of \eqref{equ1}. We would like to mention that it is also possible to 
consider a c\`adl\`ag-version based on a different Euler approximation scheme.

We will prove in Theorem \ref{existunique} below for this general class of neural networks  
existence and uniqueness of a strong solution of the associated McKean-Vlasov equation, as well 
as a propagation of chaos result in Section \ref{SecPropagationChaos} under suitable assumptions on 
the spatial distribution of neurons.     

% technical details and comnparison with existing literature 

So far, the literature already contains a considerable amount of results concerning mean-field 
limits of interacting stochastic differential equations and also extensions to the delay case. 
However, a unified theory including all the above mentioned features of our model, monotonicity of 
the coefficients, general delay, jump diffusion forcing terms and spatial structure, is not 
available. 

Indeed, results under global Lipschitz assumptions have been obtained in 
\cite{quininao2015limits,touboul2014spatially,touboul2014propagation}, extensions to locally  
Lipschitz, resp. merely one-sided Lipschitz coefficients, have been obtained in 
\cite{baladron2012mean} (with clarification note \cite{bossy2015clarification}) and  
\cite{luccon2014mean,touboul2012limits}. The clarification \cite{bossy2015clarification} refers to an 
erroneous management of hitting times 
used in \cite{baladron2012mean} to localize the problem of existence of the McKean-Vlasov 
equations under local Lipschitz assumptions. Unfortunately, the same problem arises in the paper  
\cite{touboul2012limits}. Our approach in the present manuscript is different from 
the approach in all these references, since we apply the Euler approximation to the construction 
of a solution. To the best of our knowledge, this is the first time in the existing literature 
this technique has been applied to McKean-Vlasov equations. 

As already mentioned, delay terms form an important modelling issue in the neuroscience 
applications, and therefore also have been considered in the literature, e.g., point delays in  
\cite{quininao2015limits,touboul2012limits,touboul2014spatially,touboul2014propagation} and 
other references. Continuum limits of spatially structured biological neural networks and 
various types of random delay and/or random locations of the single neurons have been considered 
in \cite{luccon2014mean} and also in 
\cite{touboul2014propagation,touboul2014spatially}.  

Note that in all the previous mentioned results, the noise is assumed to be diffusive only. 
The only existing result on McKean-Vlasov limits of biological neural networks driven by L\'evy 
noise is the recent paper \cite{andreis2018mckean} for local dynamics under global Lipschitz assumptions or local dynamics of gradient type. 

Let us next specify the precise assumptions on the coefficients of \eqref{equ1} that we assume for the well-posedness of associated McKean-Vlasov equations.
 
\begin{hyp}
\label{hyp1} 
There exist a probability measure $\lambda$ on $[-\tau , 0]$ and nonnegative measurable functions 
$K_t ({\omega^\prime} )$, $L_t ({\omega^\prime} )$, $\bar{K}_t({\omega^\prime} )$, 
$\bar{L}_t({\omega^\prime} )$ and $\tilde{K}_t (R, {\omega^\prime} )$ in 
$L^1_{loc} ([0,\infty[,dt)$, for all $R > 0$ and all $\omega^\prime\in \Omega^\prime$, such that the 
following conditions concerning local dynamics, synaptic transmissions and disorder hold: 
\begin{itemize}
\item {Assumptions concerning local dynamics}
\begin{enumerate}[label=(H\theenumi)]
\item \label{H1} 
$2\left\langle x-\tilde{x},f(t,r,x,{\omega^\prime})-f(t,r,\tilde{x},{\omega^\prime})\right\rangle +\left\lvert g(t,r,x,{\omega^\prime}) 
 -g(t,r,\tilde{x},{\omega^\prime})\right\rvert^2$
 \[\qquad\qquad +\int_{U}\left\lvert h(t,r,x,{\omega^\prime},\xi)-h(t,r,\tilde{x},{\omega^\prime},\xi)\right\rvert^2\nu 
(\de \xi)
\le L_t({\omega^\prime})\left\lvert x-\tilde{x}\right\rvert^2 , 
\]
\item \label{H2}
$2\left\langle  x,f(t,r,x,{\omega^\prime})\right\rangle +\left\lvert g(t,r,x,{\omega^\prime})\right\rvert^2+\int_{U}\left\lvert h(t,r,x,{\omega^\prime},\xi)\right\rvert^2\nu(\de \xi)\le K_t({\omega^\prime})(1+\left\lvert x\right\rvert^2)$, 
%unnecessary...
%\item \label{H5}
%$x\mapsto f(t,r,x,{\omega^\prime})$ is continuous.
\item \label{H6}
$\sup_{\left\lvert x\right\rvert\le R}\left[\left\lvert f(t,r,x,{\omega^\prime})\right\rvert  
+\left\lvert g(t,r,x,{\omega^\prime})\right\rvert^2 
+\int_U\left\lvert h(t,r,x,{\omega^\prime} ,\xi )
\right\rvert^2\nu(\de \xi)\right]
\le \tilde{K}_t(R,{\omega^\prime})$. 
\end{enumerate}
\item {Assumptions concerning synaptic transmissions}
\begin{enumerate}[label=(H\theenumi)]
\setcounter{enumi}{3}
\item \label{H3}
$ 
\sum_{\Theta\in \left\lbrace \theta,\beta\right\rbrace}\left\lvert \Theta(t,r,r^\prime,x,y_{-\tau:0},{\omega^\prime})-
\Theta(t,r,r^\prime,\tilde{x},\tilde{y}_{-\tau:0},{\omega^\prime})\right\rvert^2 
$
\begin{flalign*}
&\qquad +\int_U\left\lvert \eta(t,r,r^\prime,x,y_{-\tau:0}, 
{\omega^\prime},\xi)-\eta(t,r,r^\prime,\tilde{x},\tilde{y}_{-\tau:0},{\omega^\prime},\xi)\right\rvert^2\nu 
(\de \xi)\\&\le \bar{L}_t({\omega^\prime})\Big[\left\lvert x-\tilde{x}\right\rvert^2+\int_{-\tau}^0\left(\left\lvert y_s-\tilde{y}_s\right\rvert^2+\mathbf{1}_{\left\lbrace s<0\right\rbrace}\left\lvert y_{s^+}-\tilde{y}_{s^+}\right\rvert^2\right)\lambda(\de s)\Big],
\end{flalign*}
\item \label{H4''}
$\sum_{\Theta\in \left\lbrace \theta,\beta\right\rbrace}\left\lvert\Theta(t,r,r^\prime,x,y_{-\tau:0},{\omega^\prime})\right\rvert^2+\int_U\left\lvert \eta(t,r,r^\prime,x,y_{-\tau:0}, 
{\omega^\prime},\xi)\right\rvert^2\nu 
(\de \xi)$\[ \le \bar{K}_t({\omega^\prime})\left[1+\left\lvert x\right\rvert^2+\int_{-\tau}^0 \left(\left\lvert y_s\right\rvert^{2}+\mathbf{1}_{\left\lbrace s<0\right\rbrace}\left\lvert y_{s^+}\right\rvert^{2}\right)\lambda(\de s)\right],\]
\end{enumerate}
\item{Assumption referring to the disorder}
\begin{enumerate}[label=(H\theenumi)]
\setcounter{enumi}{5}
\item \label{H9}  for all $T>0$ the following expectation value is finite:
$$
\mathcal{E}\left\lbrace\exp\int_0^T\left[2L_s({\omega^\prime}) 
+ \bar{L}_s({\omega^\prime})\left(P+6\sup_{N\in\mathbb{N}}
\sum_{\alpha=1}^P\frac{\left(\#\mathcal{A}_N\cap\Gamma_\alpha\right)^2}{\mathcal{S}
_{\mathcal{A}_N,\alpha}^2}\right)+K_s({\omega^\prime})+3P\bar{K}_s({\omega^\prime})\right]\de s\right\rbrace < \infty\, .
$$ 
\end{enumerate} 
\end{itemize}
\end{hyp}

\begin{prop}
Under Hypothesis \ref{hyp1}, equation \eqref{equ1} has a unique strong solution.
\end{prop}

The proof of this proposition is rather standard. However, we could not find a reference in the 
literature that covers our setting completely. Therefore, we have incorporated a general existence 
and uniqueness result for stochastic delay differential equations with monotone  
coefficients driven by jump diffusions in the Appendix \ref{Appendix}. 

%\begin{remark} 
%For uniform integrability of 
%It will be part of our future work to replace in condition \ref{H3} (resp. \ref{H4}) the 
%integral w.r.t. $\lambda$ by the sup norm in order to weaken the continuity assumptions on the 
%interaction terms w.r.t. the delay. This will require an extension of the stochastic Gronwall 
%lemma proven in \cite{Scheutzow2013} for the case of continuous adapted processes to the present 
%case of adapted processes with jumps.  
%\end{remark} 
 
\begin{example} 
[FitzHugh Nagumo model with electrical synapses and simple maximum conductance variation]

Let us briefly discuss as an important example for networks of conductance-based point neuron 
models a network of FitzHugh-Nagumo neurons. In this model, two variables, the voltage variable 
$V$ having a cubic nonlinearity and a slower recovery variable $w$ describe the state of each 
neuron, that is reduced to one single point. We consider external current acting on the neuron 
placed at $r\in\Gamma$ with $\de I^r_{ext}=\lambda_1^r\de t+\lambda_2^r\de W^r_t$. To account 
for the time for the signal of the presynaptic neuron to travel down the axon, we incorporate 
delay in the presynaptic voltage, so that the current $I_t^{r,\tilde{r}}$ of the presynaptic 
neuron at position $\tilde{r}$ acting on the neuron at position $r$ at time $t$ is given by the 
following differential equation $\de I_t^{r,\tilde{r}}= \left(V^{r,\mathcal{A}_N}_{t^-} 
-V^{\tilde{r},\mathcal{A}_N}_{t-\tau}\right)\de J^{r,\tilde{r}}_t$ where $J^{r,\tilde{r}}_t$ 
denotes the maximum conductance which we assume to be given as the following jump diffusion 
equation: 
\[
\de J^{r,\tilde{r}}_t=\frac{1}{\mathcal{S}_{\mathcal{A}_N,\alpha}}\left[A_1^{r,\tilde{r}} 
\de t + A_2^{r,\tilde{r}}\de B^{r,\alpha}_t+\int_U\bar{\eta}^{r,\tilde{r}}(\xi)\tilde{N}^{r,\alpha} (\de t,\de \xi)\right] . 
\]
The network equation in this example is then given as  
\begin{equation} 
\label{FHNNetwork1} 
\begin{split}
\begin{dcases}
\de V^{r,\mathcal{A}_N}_t 
& =\left(-\frac{1}{3}\left(V^{r,\mathcal{A}_N}_{t^-}\right)^3+V^{r,\mathcal{A}_N}_{t^-} 
- w^{r,\mathcal{A}_N}_{t^-}+\lambda_1^r\right)\de t
+ \lambda_2^r \de W^r_t \\ 
& \quad-\sum_{\alpha=1}^P\frac{1}{\mathcal{S}_{\mathcal{A}_N,\alpha}}\sum_{\tilde{r}\in \mathcal{A}_N\cap \Gamma_\alpha}\left( V^{r,\mathcal{A}_N}_{t^-}-V^{\tilde{r},\mathcal{A}_N}_{t-\tau}\right)A_1^{r,\tilde{r}}\de t
\\ 
& \quad-\sum_{\alpha=1}^P\frac{1}{\mathcal{S}_{\mathcal{A}_N,\alpha}}\sum_{\tilde{r}\in \mathcal{A}_N\cap \Gamma_\alpha}\left( V^{r,\mathcal{A}_N}_{t^-}-V^{\tilde{r},\mathcal{A}_N}_{t-\tau}\right)A_2^{r,\tilde{r}}\de B^{r,\alpha}_t
\\ 
& \quad -\sum_{\alpha=1}^P\frac{1}{\mathcal{S}_{\mathcal{A}_N,\alpha}}\sum_{\tilde{r}\in \mathcal{A}_N\cap \Gamma_\alpha}\int_U\left( V^{r,\mathcal{A}_N}_{t^-}-V^{\tilde{r},\mathcal{A}_N}_{t-\tau}\right)\bar{\eta}^{r,\tilde{r}}(\xi)\tilde{N}^{r,\alpha}(\de t,\de \xi),\\
\de w^{r,\mathcal{A}_N}_t&=\lambda_3^r(V^{r,\mathcal{A}_N}_t+\lambda_4^r-\lambda_5^r w^{r,\mathcal{A}_N}_{t^-})\de t
\end{dcases}\end{split}
\end{equation} 
(see e.g. \cite{baladron2012mean}). Here, the measurable functions $A_i^{r,\tilde{r}}$,  
$1\leq i\leq 2$ and $\lambda_i^r, 1\leq i \leq 2$ are real valued and $\lambda_i^r, 3\leq i 
\leq 5$ are positive. The measurable function $(r,\tilde{r})\mapsto \bar{\eta}^{r,\tilde{r}}$ 
takes values in $L^2(U,\mathcal{U},\nu)$.

To obtain a continuum limit for \eqref{FHNNetwork1} we now assume the existence of a finite 
Borel measure $\mathcal{R}$ on $\Gamma$ with the following property: For every  
$\varepsilon > 0$ there exists a (finite) partition 
$\left\lbrace\Gamma_\alpha^{m,\varepsilon}, 1\le m\le M_\alpha^{(\varepsilon)}\right\rbrace$ 
of $\Gamma_\alpha$ such that 
\begin{equation}
\lim_{N\to\infty}
\frac{\#\mathcal{A}_N\cap \Gamma_\alpha^{m,\varepsilon}}{\mathcal{S}_{\mathcal{A}_N,\alpha}} 
= \mathcal{R}(\Gamma_\alpha^{m,\varepsilon}) , \quad 1\le m\le M_\alpha^{(\varepsilon)}    
\end{equation} 
and $\lim_{N\to\infty} \mathcal{S}_{\mathcal{A}_N,\alpha} =\infty$, $1\le \alpha\le P$ and for 
every $r,\tilde{r}\in \Gamma_\alpha^{m,\varepsilon}$ and $r^\prime,\tilde{r}^\prime\in 
\Gamma_{\alpha^\prime}^{m^\prime,\varepsilon}$
\[
\left\lvert A_i^{r,r^\prime}-A_i^{\tilde{r}, \tilde{r}^\prime}\right\rvert<\varepsilon,  
\quad 1\leq i\leq 2;\qquad
\left\lvert \lambda_i^{r}-\lambda_i^{\tilde{r}}\right\rvert<\varepsilon,\quad 1\leq i\leq 5;
\qquad
\int_U\left\lvert \bar{\eta}^{r,r^\prime}(\xi)-\bar{\eta}^{\tilde{r}, \tilde{r}^\prime}
(\xi)\right\rvert^2 \nu(\de \xi)<\varepsilon. 
\]
Under these assumptions the solution of \eqref{FHNNetwork1} converges to the solution of the 
following McKean-Vlasov equation:
\begin{equation}
\begin{split}
\begin{dcases}
\de V^{r}_t&=\left(-\frac{1}{3}\left(V^{r}_{t^-}\right)^3+V^{r}_{t^-}-w^{r}_{t^-}+\lambda_1^r\right)\de t
+\lambda_2^r\de W^r_t\\&\quad-\sum_{\alpha=1}^P\int_{\Gamma_\alpha}\left( V^{r}_{t^-}-
V^{\tilde{r}}_{t-\tau}\right)A_1^{r,\tilde{r}}\mathcal{R}(\de \tilde{r})\,\de t
\\
& \quad 
-\sum_{\alpha=1}^P\int_{\Gamma_\alpha}\left( V^{r}_{t^-}-V^{\tilde{r}}_{t-\tau}\right)A_2^{r,
\tilde{r}}\mathcal{R}(\de \tilde{r})\,\de B^{r,\alpha}_t
\\
& \quad -\sum_{\alpha=1}^P\int_U\int_{\Gamma_\alpha}\left( V^{r}_{t^-} 
-V^{\tilde{r}}_{t-\tau}\right)\bar{\eta}^{r,\tilde{r}}(\xi)\mathcal{R}(\de \tilde{r})\,
\tilde{N}^{r,\alpha}(\de t,\de \xi), 
\\
\de w^{r}_t&=\lambda_3^r(V^{r}_t+\lambda_4^r-\lambda_5^r w^{r}_{t^-})\de t \, . 
\end{dcases} 
\end{split}
\end{equation}

\end{example}

\subsection{Well-posedness of the McKean-Vlasov equation}

For a given finite Borel measure $\mathcal{R}$ on $\Gamma$, specifying the spatial (unnormalized) 
distribution of neurons, the  McKean-Vlasov equation  for  the infinite population limit of the 
network \eqref{equ1} is given by the following equation:
\begin{equation} 
\label{equ3}
\begin{split}
\de \bar{X}^{r}_t&= f\left(t,r,\bar{X}^{r}_{t^-},{\omega^\prime}\right)\de t+g\left(t,r,\bar{X}^{r}_{t^-},{\omega^\prime}\right)\de W^{r}_t+\int_U h\left(t,r,\bar{X}^{r}_{t^-},{\omega^\prime},\xi\right)\tilde{N}^r(\de t,\de \xi)\\&\quad+\sum_{\alpha=1}^P\int_{\Gamma_\alpha}\tilde{\mathbb{E}}\left[\theta\left(t,r,r^\prime,\bar{X}^{r}_{t^-},\hat{X}^{r^\prime}_{(t-\tau)^-:t^-},{\omega^\prime}\right)\right]\mathcal{R}(\de r^\prime) \de t\\&\quad+\sum_{\alpha=1}^P\int_{\Gamma_\alpha}\tilde{\mathbb{E}}\left[\beta\left(t,r,r^\prime,\bar{X}^{r}_{t^-},\hat{X}^{r^\prime}_{(t-\tau)^-:t^-},{\omega^\prime}\right)\right]\mathcal{R}(\de r^\prime) \de B^{r,\alpha}_t\\&\quad+\sum_{\alpha=1}^P\int_U\int_{\Gamma_\alpha}\tilde{\mathbb{E}}\left[\eta\left(t,r,r^\prime,\bar{X}^{r}_{t^-},\hat{X}^{r^\prime}_{(t-\tau)^-:t^-},{\omega^\prime},\xi\right)\right]\mathcal{R}(\de r^\prime) \tilde{N}^{r,\alpha}(\de t, \de \xi),\\
\bar{X}^{r}_t&=z^r_t, \quad  t\in [-\tau,0]
\end{split}
\end{equation}
for independent Brownian motions $W^{r}$, $B^{r,\alpha}$ and independent compensated Poisson 
measures $\tilde{N}^{r}$ and $\tilde{N}^{r,\alpha}$.  Here $\hat{X}$ is an independent copy of $X$, the solution of 
\begin{equation}
\label{equ2}
\begin{split}
\de {X}^{r}_t & = f\left(t,r,{X}^{r}_{t^-},{\omega^\prime}\right)\de t+g\left(t,r,{X}^{r}_{t^-},{\omega^\prime}\right) 
\de W_t+\int_U h\left(t,r,{X}^{r}_{t^-},{\omega^\prime},\xi\right)\tilde{N}(\de t,\de \xi)\\
& \quad+\sum_{\alpha=1}^P\int_{\Gamma_\alpha}\tilde{\mathbb{E}} 
\left[\theta\left(t,r,r^\prime,{X}^{r}_{t^-},\tilde{X}^{r^\prime}_{(t-\tau)^-:{t^-}},{\omega^\prime}\right)\right]\mathcal{R}(\de r^\prime) 
 \de t \\
& \quad+\sum_{\alpha=1}^P\int_{\Gamma_\alpha}\tilde{\mathbb{E}} 
\left[\beta\left(t,r,r^\prime,{X}^{r}_{t^-},\tilde{X}^{r^\prime}_{(t-\tau)^-:{t^-}},{\omega^\prime}\right)\right] 
\mathcal{R}(\de r^\prime) \de B_t^\alpha\\ 
& \quad + \sum_{\alpha=1}^P\int_{\Gamma_\alpha}\int_U\tilde{\mathbb{E}} 
\left[\eta\left(t,r,r^\prime,{X}^{r}_{t^-},\tilde{X}^{r^\prime}_{(t-\tau)^-:{t^-}},{\omega^\prime},
\xi\right)\right]\mathcal{R}(\de r^\prime) \tilde{N}^\alpha(\de t, \de \xi), 
\\
{X}^{r}_t & =\hat{z}^{\zeta}_t\, ,  \quad r\in \Gamma_\zeta, \zeta = 1, \ldots , P\, , 
t\in [-\tau,0]\, , 
\end{split}
\end{equation}
where $\tilde{X}=\hat{X}$ is a copy of $X$, defined on another probability space, say 
$\left(\tilde{\Omega},\tilde{\mathcal{F}}, \tilde{\mathbb{P}}\right)$ and $\tilde{\mathbb{E}}$ 
denotes  expectation with respect to $\tilde{\mathbb{P}}$. To avoid unnecessary notations, we 
assume $\mathcal{R}(\Gamma_\alpha)=1$, for all $1\le \alpha\le P$. 

Note that a solution to \eqref{equ2} requires in particular the measurability of 
$\tilde{X}$ w.r.t. $r'$, since otherwise the integrals of the expectation values w.r.t.  
$\tilde{\mathbb{P}}$ are not well-defined. In Theorem \ref{existunique} below we will therefore prove 
existence and uniqueness of a strong solution of equation \eqref{equ2} that is also measurable 
w.r.t. $r$. Note that this implies in particular the existence of an independent measurable copy 
$\tilde{X}$ or $\hat{X}$ of $X$. Therefore  the integrals w.r.t. $\mathcal{R}(\de r^\prime)$ in \eqref{equ3} and \eqref{equ2} are well-defined. Since for each $r\in\Gamma$, $\bar{X}^{r}$ and $X^{r}$ have the same law, so that  
$\hat{X}^r$ can be assumed to be a copy of $\bar{X}^r$ too. 

In the reference \cite{touboul2014propagation} 
a representation of the McKean-Vlasov equation was given in terms of a continuum family of independent Brownian motions, so called spatially 
chaotic. In the statement of Theorem \ref{PropagationChaos}, we only use the solution of equation \eqref{equ3} for $r$ in the countable set  $ \bigcup_{N\in\mathbb{N}}\mathcal{A}_N$ and for $\hat{X}$ being an independent copy of $X$, the solution of equation \eqref{equ2}.   In this sense our solution coincides 
with the solution constructed in \cite{touboul2014propagation}, but we had circumvented  
measurability issues of $\bar{X}^{r}$ w.r.t. $r$. 

\begin{lem}
\label{boundedness}
Assume Hypothesis \ref{hyp1}. For every measurable strong solution $X$ of equation \eqref{equ2} which belongs to $L^\infty([-\tau,T],\de t; L^2(\Omega\times \Gamma, \mathbb{P}\otimes \mathcal{R}; \mathbb{R}^d))$ for $\mathbb{P}^\prime$-almost all ${\omega^\prime}\in\Omega^\prime$, we have
\[
\sup_{s\in [-\tau,t]}\mathbb{E}\left[\left\lvert X^{r}_s \right\rvert ^2 \right] \le C_1(t,{\omega^\prime}),\quad r\in\Gamma\, , t\le T\, , 
\]
where 
\begin{equation}
\label{C_1 bound of slv2}
C_1(t,{\omega^\prime}):= \left(\sup_{\substack{{u\in[-\tau,0]}\\{1\leq \zeta\leq P}}}\mathbb{E}\left\lvert \hat{z}^\zeta(u)\right\rvert^2+1\right)\exp\left(\int_0^t\left(K_s({\omega^\prime})+3P\bar{K}_s({\omega^\prime})+P\right)\de s\right).
\end{equation}
\end{lem}

\begin{proof}
Let $\tau_{n,r}:=\inf\left\lbrace t \ge 0; \left\lvert X^{r}_t\right\rvert>n\right\rbrace$ and 
\[ 
\psi_t({\omega^\prime}):=\exp\left(-\int_0^t\left(K_s({\omega^\prime}) 
+ P\bar{K}_s({\omega^\prime})+P\right)\de s\right)  \, . 
\]
Using It\^o's formula, we get for $r\in \Gamma_\zeta$ that 
\begin{align*}
\psi_{t\wedge \tau_{n,r}}({\omega^\prime})\left\lvert X^{r}_{t\wedge \tau_{n,r}}\right\rvert^2 
& = \left\lvert \hat{z}^\zeta(0)\right\rvert^2 
+\int_0^{t\wedge \tau_{n,r}}\psi_{s}({\omega^\prime})\Big[ 2\left\langle X^{r}_{s^-},f(s,r,X^{r}_{s^-},{\omega^\prime})\right\rangle+\left\lvert g(s,r,X^{r}_{s^-},{\omega^\prime})\right\rvert^2 \Big] \de s  \\
& \quad
+ \int_0^{t\wedge \tau_{n,r}}\psi_{s}({\omega^\prime})\int_U\left\lvert h(s,r,X^{r}_{s^-},{\omega^\prime},\xi) 
   \right\rvert^2{N}(\de s,\de \xi)  \\ 
& \quad +\int_0^{t\wedge\tau_{n,r}}2\psi_{s}({\omega^\prime})\sum_{\alpha=1}^P\left\langle  X^{r}_{s^-},\tilde{\mathbb{E}}\int_{\Gamma_\alpha}\theta\left(s,r,r^\prime,X^{r}_{s^-},\tilde{X}^{r^\prime}_{(s-\tau)^-:s^-},{\omega^\prime}\right)\mathcal{R}(\de r^\prime)\right\rangle \de s  \\ 
& \quad  +\int_0^{t\wedge\tau_{n,r}}\psi_{s}({\omega^\prime})\sum_{\alpha=1}^P\left\lvert\tilde{\mathbb{E}}\int_{\Gamma_\alpha}\beta\left(s,r,r^\prime,X^{r}_{s^-},\tilde{X}^{r^\prime}_{(s-\tau)^-:s^-},{\omega^\prime}\right)\mathcal{R}(\de r^\prime)\right\rvert^2\de s  \\  
& \quad  +\int_0^{t\wedge\tau_{n,r}}\psi_{s}({\omega^\prime})\sum_{\alpha=1}^P\int_U\left\lvert\tilde{\mathbb{E}}\int_{\Gamma_\alpha}\eta\left(s,r,r^\prime,X^{r}_{s^-},\tilde{X}^{r^\prime}_{(s-\tau)^-:s^-},{\omega^\prime}\right)\mathcal{R}(\de r^\prime)\right\rvert^2{N}^\alpha(\de s, \de \xi)  \\  
& \quad -\int_0^{t\wedge\tau_{n,r}}\psi_{s}({\omega^\prime})\left(K_s({\omega^\prime})+P\bar{K}_s({\omega^\prime})+P\right)\left\lvert X^{r}_{s^-}\right\rvert^2\de s+M_{t\wedge \tau_{n,r}}
\end{align*}
where $M_{t\wedge\tau_{n,r}}$ is a martingale starting from zero. Taking expectation and 
using \ref{H2} and \ref{H4''} we get
\begin{align*}
\mathbb{E} \left( \psi_{t\wedge \tau_{n,r}}({\omega^\prime}) 
\left\lvert X^{r}_{t\wedge \tau_{n,r}}\right\rvert^2 \right) 
& \le \mathbb{E}\left\lvert \hat{z}^\zeta(0)\right\rvert^2+\mathbb{E}\int_0^{t\wedge\tau_{n,r}}\psi_s({\omega^\prime})\Big[
K_s({\omega^\prime})+P\bar{K}_s({\omega^\prime}) \\ 
& \qquad\qquad  
+ \bar{K}_s({\omega^\prime})\tilde{\mathbb{E}}\int_\Gamma\int_{-\tau}^0 \left(\left\lvert \tilde{X}^{r^\prime}_{(s+u)^-}\right\rvert^2+\mathbf{1}_{\left\lbrace u<0\right\rbrace}\left\lvert \tilde{X}^{r^\prime}_{s+u}\right\rvert^2\right)\lambda(\de u)\mathcal{R}(\de r^\prime)\Big]\de s \\ 
& \le \mathbb{E}\left\lvert 
\hat{z}^\zeta(0)\right\rvert^2+1+\mathbb{E}\int_0^{t\wedge\tau_{n,r}} 2
\psi_s({\omega^\prime})\bar{K}_s({\omega^\prime})\sup_{u\in[-\tau,s]} 
\tilde{\mathbb{E}}\int_\Gamma\left\lvert \tilde{X}^{r^\prime}_{u}\right\rvert^2\mathcal{R} 
(\de r^\prime)\de s
\end{align*}
and with Fatou's lemma it follows that
\begin{align}\label{ineq1-proofbddlem}
\psi_t({\omega^\prime})\mathbb{E}\left\lvert X^{r}_t\right\rvert^2 
& \le \liminf_{n\to \infty} \mathbb{E}\left( \psi_{t\wedge \tau_{n,r}} 
({\omega^\prime}) \left\lvert X^{r}_{t\wedge \tau_n}\right\rvert^2 \right) \nonumber \\ 
& \le \mathbb{E}\left\lvert \hat{z}^\zeta(0)\right\rvert^2 + 1
+ \mathbb{E}\int_0^{t}2 \psi_s({\omega^\prime})\bar{K}_s({\omega^\prime})
\sup_{u\in[-\tau,s]}\tilde{\mathbb{E}}\int_\Gamma\left\lvert \tilde{X}^{r^\prime}_{u}\right\rvert^2\mathcal{R}(\de r^\prime)\de s.
\end{align} 
Integrating w.r.t. $\mathcal{R}$ and using Gronwall's lemma we obtain that 
\begin{align}\label{ineq2-proofbddlem}
\MoveEqLeft[2]\psi_t({\omega^\prime})\sup_{u\in [-\tau,t]}\mathbb{E}\int_\Gamma\left\lvert X^{r}_u \right\rvert ^2 \mathcal{R}(\de r)\nonumber \\  
& \le \left(\sup_{u\in[-\tau,0]} \mathbb{E} \sum_{\zeta=1}^P\left\lvert \hat{z}^\zeta(u)\right\rvert^2 + P\right)  
\exp{\int_0^t\left(2P \bar{K}_s({\omega^\prime})\right)\de s}\, . 
\end{align}
By substituting \eqref{ineq2-proofbddlem} in \eqref{ineq1-proofbddlem}, we get
\begin{align*}
\psi_t({\omega^\prime})\mathbb{E}\left\lvert X^{r}_t\right\rvert^2 
\le \left(\sup_{\substack{{u\in[-\tau,0]}\\{1\leq \zeta\leq P}}}\mathbb{E}\left\lvert  
\hat{z}^\zeta(u)\right\rvert^2+1\right) 
\exp{\int_0^t\left(2P \bar{K}_s({\omega^\prime})\right)\de s}\, .
\end{align*}
Therefore
\[
\sup_{s\in[-\tau,t]}\mathbb{E}\left\lvert X^{r}_s\right\rvert^2 
\le \left(\sup_{\substack{{u\in[-\tau,0]}\\{1\leq \zeta\leq P}}}\mathbb{E}\left\lvert 
\hat{z}^\zeta(u)\right\rvert^2 + 1\right)\exp{\int_0^t\left(K_s({\omega^\prime}) + 3P\bar{K}_s({\omega^\prime})+P\right)\de s}=C_1(t,{\omega^\prime})\, . 
\]
\end{proof}

\begin{thm}
\label{existunique}
Equation \eqref{equ2} has a unique strong solution $X({\omega^\prime})\in  
L^\infty([-\tau,T], \de t; L^2(\Omega\times \Gamma, \mathbb{P}\otimes \mathcal{R}; 
\mathbb{R}^d))$ on $[-\tau,T]$ for any $T>0$ and $\mathbb{P}^\prime$-almost every 
${\omega^\prime}\in\Omega^\prime$ which is in particular measurable w.r.t. $(r,\omega^\prime)$.
\end{thm}

The proof of the theorem is postponed to Section \ref{ProofOfThmExistunique}.

\subsection{Propagation of Chaos}
\label{SecPropagationChaos} 

We are now going to state a convergence result for the solution of the network equations 
\eqref{equ1} to the solution of the McKean-Vlasov equation \eqref{equ3} in the infinite 
population limit. To this end we first have to specify a condition on the spatial density of the 
approximating network populations and a statement concerning the dependence of $X_t^{r}$ w.r.t. 
the spatial parameter $r$. To this end consider the following hypothesis:

\begin{hyp}
\label{hyp2} 
The coefficients of the network equation \eqref{equ1} satisfy Hypothesis \ref{hyp1}. In addition 
we assume the existence of a finite Borel measure $\mathcal{R}$ on $\Gamma$ with the following property:  
For every $\varepsilon > 0$ there exists a (finite) partition 
$\left\lbrace\Gamma_\alpha^{m,\varepsilon}, 1\le m\le M_\alpha^{(\varepsilon)}\right\rbrace$ 
of $\Gamma_\alpha$ such that 
\begin{equation}\label{assumption}
\lim_{N\to\infty}
\frac{\#\left( \mathcal{A}_N\cap \Gamma_\alpha^{m,\varepsilon}\right)}{\mathcal{S}_{\mathcal{A}_N,\alpha}} 
= \mathcal{R}(\Gamma_\alpha^{m,\varepsilon}) , \quad 1\le m\le M_\alpha^{(\varepsilon)}    
\end{equation} 
and also $\lim_{N\to\infty} \mathcal{S}_{\mathcal{A}_N,\alpha} =\infty$, $1\le \alpha\le P$.  

W.r.t. this partition \ref{H1} and \ref{H3} of Hypothesis \ref{hyp1} are then replaced 
by the following stronger assumptions: 
\begin{itemize}
\item[(H1')] $\forall\ r,\tilde{r}\in \Gamma_\alpha^{m,\varepsilon}$, 
$$ 
\begin{aligned} 
& 2\left\langle x-y,f(t,r,x,{\omega^\prime})-f(t,\tilde{r},\tilde{x},{\omega^\prime})\right\rangle +\left\lvert g(t,r,x,
{\omega^\prime})-g(t,\tilde{r},\tilde{x},{\omega^\prime})\right\rvert^2 \\
& \qquad\qquad 
 + \int_{U}\left\lvert h(t,r,x,{\omega^\prime},\xi)-h(t,\tilde{r},\tilde{x},{\omega^\prime}, 
  \xi)\right\rvert^2\nu(\de \xi) 
\le L_t({\omega^\prime})\left[\left\lvert x-\tilde{x}\right\rvert^2 
+ \varepsilon \left(1+\left\lvert x\right\rvert^2\right)\right]\, , 
\end{aligned} 
$$ 
\item[(H4')] 
$\forall\ r,\tilde{r}\in \Gamma_\alpha^{m,\varepsilon}$, $\forall\ r^\prime ,
\tilde{r}^\prime\in \Gamma_{\alpha^\prime}^{m^\prime ,\varepsilon}$, 
$$ 
\begin{aligned} 
& \sum_{\Theta\in \left\lbrace \theta,\beta\right\rbrace} 
 \left\lvert \Theta(t,r,r^\prime,x,y_{-\tau:0},
 {\omega^\prime})-\Theta(t,\tilde{r},{\tilde{r}^\prime},\tilde{x},\tilde{y}_{-\tau:0}, 
  {\omega^\prime})\right\rvert^2 \\
& \qquad\qquad\qquad  
+\int_U \left\lvert \eta(t,r,r^\prime,x,y_{-\tau:0}, {\omega^\prime},\xi)
-\eta(t,\tilde{r},\tilde{r}^\prime, \tilde{x},\tilde{y}_{-\tau:0},{\omega^\prime},  
 \xi)\right\rvert^2\nu (\de \xi) \\
& 
\le \bar{L}_t({\omega^\prime})\Bigg[\left\lvert x-\tilde{x}\right\rvert^2+\int_{-\tau}^0\left[\left 
\lvert y_s-\tilde{y}_s\right\rvert^2+\mathbf{1}_{\left\lbrace s<0\right\rbrace}\left 
\lvert y_{s^+}-\tilde{y}_{s^+}\right\rvert^2\right]\lambda(\de s) 
\\ 
& \qquad\qquad\qquad  
+ \varepsilon \left(1+\left\lvert 
 x\right\rvert^2+\int_{-\tau}^0\left[\left\lvert y_s\right\rvert^2+\mathbf{1}_{\left\lbrace s<0\right\rbrace}\left\lvert y_{s^+}\right\rvert^2\right]\lambda(\de s)\right) \Bigg] \, . 
\end{aligned} 
$$ 
\end{itemize} 
\end{hyp}

\begin{lem}\label{continuity w.r.t r}
Under hypothesis \ref{hyp2} the solution $X({\omega^\prime})\in L^\infty([-\tau,T],\de t; 
L^2(\Omega\times \Gamma, \mathbb{P}\otimes \mathcal{R}; \mathbb{R}^d))$, ${\omega^\prime}\in 
\Omega^\prime$, of equation \eqref{equ2} satisfies
\[
\mathbb{E}\left\lvert X^{r}_t-X^{\tilde{r}}_t \right\rvert^2\le C_2(t,{\omega^\prime}) 
\varepsilon 
\qquad\forall\ r,\tilde{r}\in\Gamma_\alpha^{m,\varepsilon}\, , t \le T\, , 
\]
where
\[
C_2(t,{\omega^\prime}):=\exp\left[\int_0^t\left(L_s({\omega^\prime})+P\bar{L}_s({\omega^\prime})+P\right)\de s\right]\left(1+3C_1(t,{\omega^\prime})\right)\, . 
\]
\end{lem} 

\begin{proof}
Let $\psi_t({\omega^\prime}):=\exp\left[-\int_0^t\left(L_s({\omega^\prime}) 
+ P\bar{L}_s({\omega^\prime})+P\right)\de s\right]$. To simplify notations, let 
$u := (s,r,X_{s^-}^r, \omega')$, $\tilde{u} := (s,\tilde{r},X_{s^-}^{\tilde{r}}, \omega')$,  
$v := (s,r,r', X_{s^-}^r, \tilde{X}_{(s-\tau )^- :s^-}^{r'}, \omega')$ and 
$\tilde{v} := (s,\tilde{r},r', X_{s^-}^{\tilde{r}}, 
\tilde{X}_{(s-\tau )^- :s^-}^{r'}, \omega')$.  
We then have

\begin{align*}
\psi_t({\omega^\prime})\mathbb{E}\left\lvert X^{r}_t-X^{\tilde{r}}_t\right\rvert^2  
& \le \mathbb{E}\int_0^t \psi_s({\omega^\prime})\Big[2\left\langle X^{r}_{s^-} 
- X^{\tilde{r}}_{s^-}, f\left(u\right) - f\left(\tilde{u}\right)\right\rangle 
+ \left\lvert g\left(u\right)-g\left(\tilde{u}\right)\right\rvert^2 \\  
& \qquad 
 + \int_U\left\lvert h\left(u,\xi\right)- h\left(\tilde{u},\xi\right)\right\rvert^2 \nu(\de \xi) 
 - \left(L_s({\omega^\prime})+P\bar{L}_s({\omega^\prime}) 
 + P\right)\left\lvert X^{r}_{s^-}-X^{\tilde{r}}_{s^-}\right\rvert^2\Big] \de s \\  
& \qquad 
 + \sum_{\alpha=1}^P\mathbb{E}\int_0^t \psi_s({\omega^\prime}) 
 \Bigg[ 2\left\langle X^{r}_{s^-} -X^{\tilde{r}}_{s^-}, \tilde{\mathbb{E}} 
 \int_{\Gamma_\alpha} \left( \theta \left( v\right) - \theta\left( \tilde{v}\right)\right) 
  \mathcal{R}(\de r^\prime)\right\rangle \\ 
 & \qquad  
 + \left\lvert \tilde{\mathbb{E}}\int_{\Gamma_\alpha} 
 \left(\beta\left(v \right) -\beta\left(\tilde{v}\right) \right) 
 \mathcal{R}(\de r^\prime)\right\rvert^2 
 + \int_U\left\lvert\tilde{\mathbb{E}}\int_{\Gamma_\alpha} 
 \left(\eta \left(v, \xi\right)-\eta\left(\tilde{v},\xi\right)\right) 
  \mathcal{R} (\de r^\prime)\right\rvert^2\nu (\de \xi)\Bigg]\de s \\ 
& \le \mathbb{E}\int_0^t \psi_s({\omega^\prime}) \varepsilon  
  \Big[L_s({\omega^\prime}) \left(1+\left\lvert X^{r}_{s^-}\right\rvert^2\right)
  + \bar{L}_s({\omega^\prime}) \Big( P+P\left\lvert X^{r}_{s^-}\right\rvert^2  \\ 
& \qquad\qquad   
 + \tilde{\mathbb{E}}\int_{\Gamma}\int_{-\tau}^0\left[\left\lvert
    \tilde{X}^{r^\prime}_{(s+u)^-}\right\rvert^2  
    +\mathbf{1}_{\left\lbrace u<0\right\rbrace} 
   \left\lvert \tilde{X}^{r^\prime}_{s+u}\right\rvert^2\right]\lambda(\de u) 
     \mathcal{R}(\de r^\prime)\Big)\Big] \de s \\ 
& \le \varepsilon \int_0^t \psi_s({\omega^\prime}) 
\left(P \bar{L}_s({\omega^\prime})+L_s({\omega^\prime})\right)\left(1+3C_1(s,{\omega^\prime}) 
\right)\de s \\
& \le \varepsilon (1+3C_1(t,{\omega^\prime})) \, . 
\end{align*}
Dividing by $\psi_t({\omega^\prime})$ we get the desired result.
\end{proof}

The following theorem now is our second main result: 

\begin{thm}
\label{PropagationChaos} 
Under Hypothesis \ref{hyp2} and the chaotic initial condition assumption (i.e.,
the initial conditions $z^r$ for $r\in \Gamma_\alpha\cap  
\left(\bigcup_{N\in\mathbb{N}} \mathcal{A}_N\right)$
are independent and identically distributed copies of $\hat{z}^\alpha$ with $\hat{z}^\alpha\in 
L^2(\Omega,\mathbb{P};\text{C\`adl\`ag}\,([-\tau,0];\mathbb{R}^d))$), the solution   
$\left(X^{r,\mathcal{A}_N}_t, -\tau\le t \le T\right)$ of the network equation \eqref{equ1} 
converges in the space $L^2 \left( \Omega^\prime, L^\infty \left( [0,T], L^2 \left( \Omega, 
\mathbb{E}\right) \right)\right)$ towards the process $\left(\bar{X}^r_t, -\tau\le t \le T\right)$ which 
is the solution of the mean-field equation \eqref{equ3}, i.e.
\[
\lim_{N\to\infty}\mathcal{E}\sup_{\substack{t\in [-\tau ,T] \\ r\in \mathcal{A}_N}} 
\mathbb{E}\left\lvert X^{r,\mathcal{A}_N}_t-\bar{X}^r_t\right\rvert^2=0\, . 
\]
\end{thm}

The proof of the theorem is given in Section \ref{ProofThmPropagationChaos}.

\section{Proof of Theorem \ref{existunique}}
\label{ProofOfThmExistunique}

\begin{proof}
\textbf{Existence:} Fix $n\in\mathbb{N}$ and define 
$\kappa(n,t):=\frac{k\tau}{n}$ for $t\in \left]\frac{k\tau}{n},\frac{(k+1)\tau}
{n}\right]$. We then define the process $X^{n,r}$ inductively as 
follows: Let $X^{n,r}_t := \hat{z}^\zeta_t$ for $t\in [-\tau,0]$ 
and $r\in \Gamma_\zeta$. Given that $X^{n,r}_t$ is defined 
for $t\le \frac{k\tau}{n}$ and for all $r\in\Gamma$ we extend 
$X^{n,r}_t$ for $t\in \left]\frac{k\tau}{n},\frac{(k+1)\tau}{n}\right]$ 
as the unique strong solution of
\begin{equation}
\label{df Xn}
\begin{split}
X^{n,r}_t &=X^{n,r}_{\frac{k\tau}{n}}+\int_{\frac{k\tau}{n}}^t f
\left(s,r, X^{n,r}_{s^-},{\omega^\prime}\right)\de s 
+ \int_{\frac{k\tau}{n}}^t g\left(s,r, X^{n,r}_{s^-},{\omega^\prime}
\right)\de W_s \\  
& \quad + \int_{\frac{k\tau}{n}}^t \int_U h\left(s,r, X^{n,r}_{s^-},{\omega^\prime}
,\xi\right)\tilde{N}(\de s,\de \xi) \\ 
& \quad + \sum_{\alpha=1}^P\int_{\frac{k\tau}{n}}^t
\tilde{\mathbb{E}}\int_{\Gamma_\alpha} \theta\left(s,r,r^\prime,X^{n,r}_{s^-} ,  Y^{n,r^\prime}_{\kappa(n,(s-\tau):s)},{\omega^\prime}
\right)\mathcal{R}(\de r^\prime)\de s \\ 
& \quad + \sum_{\alpha=1}^P\int_{\frac{k\tau}{n}}^t
\tilde{\mathbb{E}}\int_{\Gamma_\alpha} \beta
\left(s,r,r^\prime,X^{n,r}_{s^-},Y^{n,r^\prime}_{\kappa(n,(s-
\tau):s)},{\omega^\prime}\right)\mathcal{R}(\de r^\prime)\de B^\alpha_s \\ 
& \quad +\sum_{\alpha=1}^P\int_{\frac{k\tau}{n}}^t\int_U
  \tilde{\mathbb{E}}\int_{\Gamma_\alpha}  
  \eta \left(s,r,r^\prime,X^{n,r}_{s^-},Y^{n,r^\prime}_{\kappa(n,(s- \tau):s)},  
   {\omega^\prime},\xi\right)\mathcal{R}(\de r^\prime)\tilde{N}^\alpha(\de s,\de \xi)\, , 
\end{split}
\end{equation}
which exists and is measurable w.r.t. $(t,r,\omega, \omega^\prime )\in \left[\frac{k\tau}{n},\frac{(k+1)\tau}{n}\right]\times \Gamma\times \Omega\times \Omega^\prime$ and satisfies 
\[
1+\sup_{u\leq t}\mathbb{E}\left\lvert X^{n,r}_u\right\rvert^2\leq \left(1+\sup_{u\leq \frac{k\tau}{n}}\mathbb{E}\left\lvert X^{n,r}_u\right\rvert^2\right)\exp\left( \int_{k\tau/n}^t (K_s(\omega^\prime)+3P\bar{K}_s(\omega^\prime))\de s\right)
\]
according to Theorem \ref{thmA2}. Then for all $n\in \mathbb{N}$ and all $r\in \Gamma$,
\begin{equation}
\label{bddsquare}
\sup_{t\in[-\tau,T]}\mathbb{E}\left\lvert X^{n,r}_t\right\rvert^2\le C_1(T,{\omega^\prime}).
\end{equation}
Here $Y^{n,r^\prime}$ is an independent copy of $X^{n,r^\prime}$ on the probability space $(\tilde{\Omega},\tilde{\mathcal{F}},\tilde{\mathbb{P}})$ and the 
expectation $\tilde{\mathbb{E}}$ is taken with respect to $\tilde{\mathbb{P}}$ and 
\[Y^{n,r^\prime}_{\kappa(n,(s-\tau):s)}(u):=Y^{n,r^\prime}_{\kappa(n,s+u)},\quad u\in[-\tau,0].\]
We also use in the following the notation
\[Y^{n,r}_{\kappa(n,t^+)}:=\lim_{s\searrow t}Y^{n,r}_{\kappa(n,s)}.\]
For convenience we assume that $Y^{n,r}$ is obtained similar to 
$X^{n,r}$ using independent copies \[\left(\left\lbrace
\tilde{z}^\alpha\right\rbrace_{1\le \alpha\le P},\tilde{W},
\left\lbrace\tilde{B}^\alpha\right\rbrace_{1\le \alpha\le 
P},{\bar{N}},
\left\lbrace\bar{N}^\alpha\right\rbrace_{1\le \alpha\le 
P}\right)\] of $\left(\left\lbrace{z}^\alpha\right
\rbrace_{1\le \alpha\le P},{W},\left\lbrace{B}^\alpha\right
\rbrace_{1\le \alpha\le P},{N},
\left\lbrace {N}^\alpha\right\rbrace_{1\le \alpha\le 
P}\right)$. 
Note that $X^{n,r}$ is c\`adl\`ag, whereas the process $X^{n,r}_{\kappa(n,t)}$, $t\ge -\tau$, 
is c\`agl\`ad. It is easy to see, using induction w.r.t. to $k$, that 
$X^{(n)}_t$, $t\in \left] \frac{k\tau}n , \frac{(k+1)}n\right]$ is a.s. locally bounded 
and that the stochastic integrals are well-defined and local martingales up to time $+\infty$.  

\medskip 
\noindent 
Let us next define the remainder 
\[
p^{n,r}_t=X^{n,r}_{\kappa(n,t)}-X^{n,r}_{t^-}, \quad q^{n,r}_t=Y^{n,r}_{\kappa(n,t)}-Y^{n,r}_{t^-} , \quad t\in [-\tau,T]\, . 
\]
We can then write 
\begin{equation}  
\label{Xn=phi(Xn+pn)}
\begin{split}
X^{n,r}_t & =\hat{z}^\zeta(0)+\int_0^t f\left(s,r, X^{n,r}_{s^-} , 
{\omega^\prime}\right) \de s 
 +\int_0^t g\left(s,r, X^{n,r}_{s^-},{\omega^\prime}\right) 
\de W_s \\  
& \quad+\int_0^t \int_U h\left(s,r, X^{n,r}_{s^-} , \omega^\prime , \xi \right) 
\tilde{N}(\de s,\de \xi) \\ 
& \quad + \sum_{\alpha=1}^P \int_0^t\tilde{\mathbb{E}}\int_{\Gamma_\alpha}
\theta\left(s,r,r^\prime,X^{n,r}_{s^-} , Y^{n,r^\prime}_{(s-\tau)^-:s^-}
+q^{n,r^\prime}_{(s-\tau):s},{\omega^\prime}\right)\mathcal{R}(\de r^\prime)\de s \\  
& \quad + \sum_{\alpha=1}^P\int_0^t\tilde{\mathbb{E}}\int_{\Gamma_\alpha}
\beta\left(s,r,r^\prime,X^{n,r}_{s^-} , Y^{n,r^\prime}_{(s-
\tau)^-:s^-}+q^{n,r^\prime}_{(s-\tau):s},{\omega^\prime}\right)\mathcal{R}(\de r^\prime)\de B^
\alpha_s \\ 
& \quad + \sum_{\alpha=1}^P\int_0^t\int_U\tilde{\mathbb{E}}\int_{\Gamma_\alpha}
\eta\left(s,r,r^\prime,X^{n,r}_{s^-} , Y^{n,r^\prime}_{(s-
\tau)^-:s^-}+q^{n,r^\prime}_{(s-\tau):s},{\omega^\prime},\xi\right)\mathcal{R} 
(\de r^\prime)\tilde{N}^\alpha(\de s, \de \xi).
\end{split}
\end{equation}
In the next step let us define the stopping times 
\[
\tau^{n,r}_R:=\inf\left\lbrace t\ge 0: \left\lvert 
X^{n,r}_t\right\rvert>\frac{R}{3}\right\rbrace 
\]
for given $R > 0$. Then  
\[
\left\lvert p^{n,r}_{t}\right\rvert \le \frac{2R}{3}, \quad \left
\lvert X^{n,r}_{t^-}\right\rvert,\left
\lvert X^{n,r}_{\kappa(n,t)}\right\rvert \le \frac{R}{3},\qquad t\in(0,
\tau^{n,r}_R]. 
\]
We now prove the following properties which complete the 
existence proof. 
\begin{enumerate}[label=(\roman{*})]
\item \label{i} For all $t> -\tau$,   
$\mathbf{1}_{(-\tau,\tau^{n,r}_R]}(t)\left\lvert p^{n,r}_t\right\rvert \to 0$ in probability as   
$n\to\infty$. 
\item \label{ii} For any stopping time $\tau^*\le T\wedge\tau^{n,r}_R$ we have
$\mathbb{E}\left\lvert X^{n,r}_{\tau^*}\right\rvert^2\le C(T,{\omega^\prime})$. 
\item \label{iii} $\lim_{R\to \infty}\limsup_{n\to\infty}\mathbb{P}\left\lbrace \tau^{n,r}_R < T\right\rbrace=0$.
\item \label{iv} $\forall \varepsilon>0, \lim_{n,m\to\infty}\mathbb{P}\left\lbrace \sup_{t\in[0,T]}\left\lvert X^{n,r}_t-X^{m,r}_t\right\rvert >\varepsilon\right\rbrace=0$.
\item \label{v}  $\exists X: \forall \varepsilon>0, \lim_{n\to\infty}\mathbb{P}\left\lbrace \sup_{t\in[0,T]}\left\lvert X^{n,r}_t-X^{r}_t\right\rvert >\varepsilon\right\rbrace=0$ and $X$ is a strong solution of equation \eqref{equ2}.
\end{enumerate}

\paragraph*{Proof of \ref{i}:} Since $\hat{z}^\zeta$ is c\`adl\`ag w.r.t. time, it follows that $\mathbf{1}_{(-\tau,0]}(t)\left\lvert p^{n,r}_t\right\rvert \to 0$ almost surely. Using \eqref{df Xn} and Hypothesis \ref{hyp1}, we have
\begin{align*}
\mathbb{P}\big\lbrace \left\lvert p^{n,r}_t\right\rvert 
&  \ge \varepsilon, 0<t \le 
\tau^{n,r}_R\big\rbrace 
\\ 
& \le \mathbb{P}\left\lbrace \int_{\kappa(n,t)}^t \left(\sup_{\left\lvert x\right\rvert 
\le R}\left\lvert f(s,r,x,{\omega^\prime})\right\rvert  
+ P\sqrt{\bar{K}_s({\omega^\prime})\left(1+R^2+2C_1(t,{\omega^\prime})\right)}\right)\de s\ge 
\varepsilon/5\right\rbrace
\\ 
& \quad + \mathbb{P}\left\lbrace \left\lvert\int_{\kappa(n,t)}^t  
g\left(s,r,X^{n,r}_{s^-},
{\omega^\prime}\right)\de W_s\right\rvert\ge \varepsilon/5 , t \le \tau^{n,r}_R \right\rbrace 
\\ 
& \quad + \mathbb{P}\left\lbrace \left\lvert\int_{\kappa(n,t)}^{t^-}\int_U 
h\left(s,r,X^{n,r}_{s^-},\omega^\prime,\xi\right) 
\tilde{N}(\de s,\de \xi)\right\rvert\ge \varepsilon/5 , t \le 
\tau^{n,r}_R \right\rbrace 
\\ 
& \quad +\sum_{\alpha=1}^P\mathbb{P}\left\lbrace \left\lvert 
\int_{\kappa(n,t)}^t\tilde{\mathbb{E}}\int_{\Gamma_\alpha}\beta\left(s,r,r^\prime ,  
X^{n,r}_{s^-},Y^{n,r^\prime}_{\kappa(n,(s-\tau):s)},{\omega^\prime}\right)
\mathcal{R}(\de r^\prime)\de B^\alpha_s\right\rvert\ge \frac{\varepsilon}{5P}, 
t\le\tau^{n,r}_R \right\rbrace 
\\ 
& \quad + \sum_{\alpha=1}^P\mathbb{P}\left\lbrace \left\lvert \int_{\kappa(n,t)}^{t^-}\int_U 
\tilde{\mathbb{E}}\int_{\Gamma_\alpha}\eta\left(s,r,r^\prime,X^{n,r}_{s^-}, 
Y^{n,r^\prime}_{\kappa(n,(s-\tau):s)},{\omega^\prime},\xi\right)\mathcal{R}(\de 
r^\prime)\tilde{N}^\alpha(\de s,\de \xi)\right\rvert\ge \frac{\varepsilon}{5P} , t \le 
\tau^{n,r}_R \right\rbrace 
\end{align*} 
which can be further estimated from above by 
\begin{align*}
\hphantom{klklklkl} 
& \le\mathbb{P}\left\lbrace\int_{\kappa(n,t)}^t \left( \tilde{K}_s(R,{\omega^\prime})  
+ P\sqrt{\bar{K}_s({\omega^\prime})\left(1+R^2+2C_1(t,{\omega^\prime})\right)}\right)\de s 
\ge \varepsilon/5\right\rbrace 
\\ 
& \quad + \frac{25}{\varepsilon^2}\mathbb{E}\left(\left\lvert\int_{\kappa(n,t)}^t 
\mathbf{1}_{\left\lbrace s\le\tau^{n,r}_R\right\rbrace} g\left(s,r,X^{n,r}_{s^-},{\omega^\prime}\right)\de W_s\right\rvert^2 
\right)
\\ 
& \quad + \frac{25}{\varepsilon^2}\mathbb{E}\left( \left\lvert \int_{\kappa(n,t)}^{t^-} 
\int_U\mathbf{1}_{\left\lbrace s\le\tau^{n,r}_R\right\rbrace} h\left(s,r,X^{n,r}_{s^-},{\omega^\prime},\xi\right)
\tilde{N}(\de s,\de \xi)\right\rvert^2 \right)
\\ 
& \quad + \frac{25P^2}{\varepsilon^2}\sum_{\alpha=1}^P\mathbb{E}\left(\left\lvert 
\int_{\kappa(n,t)}^t \tilde{\mathbb{E}} \int_{\Gamma_\alpha}\mathbf{1}_{\left\lbrace s\le\tau^{n,r}_R\right\rbrace}\beta\left(s,r,r^\prime , 
X^{n,r}_{s^-}, Y^{n,r^\prime}_{\kappa(n,(s-\tau):s)},{\omega^\prime}\right)
\mathcal{R}(\de r^\prime)\de B^\alpha_s\right\rvert^2 \right) 
\\ 
& \quad + \frac{25P^2}{\varepsilon^2}\sum_{\alpha=1}^P\mathbb{E} \left( \left\lvert 
\int_{\kappa(n,t)}^{t^-}\int_U \tilde{\mathbb{E}} \int_{\Gamma_\alpha} \mathbf{1}_{\left\lbrace s\le\tau^{n,r}_R\right\rbrace}\eta 
\left(s,r,r^\prime,X^{n,r}_{s^-}, Y^{n,r^\prime}_{\kappa(n,(s-\tau):s)},\omega^\prime,
\xi\right)\mathcal{R}(\de r^\prime)\tilde{N}^\alpha(\de s,\de \xi)\right\rvert^2 
\right) 
\end{align*} 
and finally by 
\begin{align*} 
\hphantom{klklkl}  
& \le \mathbb{P}\left\lbrace\int_{\kappa(n,t)}^t \left(\tilde{K}_s(R,
{\omega^\prime})+P\sqrt{\bar{K}_s({\omega^\prime})\left(1+R^2+2C_1(t,
{\omega^\prime})\right)}\right)\de s\ge \varepsilon/5\right\rbrace\\&\quad +\frac{25}
{\varepsilon^2}\int_{\kappa(n,t)}^t \tilde{K}_s(R,{\omega^\prime})\de s+\frac{25P^3}
{\varepsilon^2}\int_{\kappa(n,t)}^t \bar{K}_s({\omega^\prime})\left(1+R^2+2C_1(t,
{\omega^\prime})\right)\de s\, . 
\end{align*}
So
\[
\limsup_{n\to \infty}\mathbb{P}\left\lbrace \left\lvert p^{n,r}_t\right\rvert\ge \varepsilon, -\tau<t\le \tau^{n,r}_R\right\rbrace = 0
\]
which implies \ref{i}.

\paragraph*{Proof of \ref{ii}:} Let $\tau^*$ be a stopping time such that  
$\tau^*\le T\wedge \tau^{n,r}_R$. Similar to the proof of the corresponding statement 
in Theorem \ref{thmA2} we have that 
\begin{align*}
\mathbb{E}\left\lvert X^{n,r}_{\tau^*}\right\rvert^2
& \le \mathbb{E}\left\lvert \hat{z}^\zeta_0\right\rvert^2+\mathbb{E}\int_0^{\tau^*} 
\bigg[2\left\langle X^{n,r}_{s^-} ,f\left(s,r,X^{n,r}_{s^-}, \omega^\prime 
\right)\right\rangle   
\\ 
& \qquad\qquad\qquad\qquad\qquad 
+ \left\lvert g\left(s,r,X^{n,r}_{s^-}, \omega^\prime \right) \right\rvert^2 
+ \int_U \left\lvert h\left(s,r,X^{n,r}_{s^-}, \omega^\prime ,\xi \right) 
\right \rvert^2\nu(\de \xi)\bigg]\de s \\ 
& \quad + \sum_{\alpha=1}^P\mathbb{E}\int_0^{\tau^*} 
\tilde{\mathbb{E}}\int_{\Gamma_\alpha}\bigg[2\left\langle X^{n,r}_{s^-}, 
\theta\left(s,r,r^\prime,X^{n,r}_{s^-}, Y^{n,r^\prime}_{\kappa(n,(s -\tau):s)} ,
\omega^\prime \right)\right\rangle 
\\ 
& \qquad\qquad\qquad\qquad\qquad  
+ \left\lvert \beta\left(s,r,r^\prime,X^{n,r}_{s^-},  
Y^{n,r^\prime}_{\kappa(n,(s-\tau):s)}, \omega^\prime \right)\right\rvert^2
\\
& \qquad\qquad\qquad\qquad\qquad 
+ \int_U \left\lvert \eta \left(s,r,r^\prime,X^{n,r}_{s^-},  
Y^{n,r^\prime}_{\kappa(n,(s-\tau):s)},  
\omega^\prime ,\xi\right)\right\rvert^2\nu (\de \xi)\bigg]\mathcal{R}(\de r^\prime)\de s 
\\ 
& \le \mathbb{E}\left\lvert \hat{z}^\zeta_0\right\rvert^2 
+ \mathbb{E}\int_0^T\bigg[ \left(K_s({\omega^\prime}) + P\bar{K}_s({\omega^\prime}) + P\right) 
\left(1+\left\lvert X^{n,r}_{s^-}\right\rvert^2 \right) 
\\ 
& \qquad\qquad\qquad\qquad\qquad 
+ \bar{K}_s({\omega^\prime})
\tilde{\mathbb{E}}\int_\Gamma\int_{-\tau}^0\left[\left\lvert 
Y^{n,r^\prime}_{\kappa(n,s+u)}\right\rvert^2 
+ \mathbf{1}_{\left\lbrace u<0\right\rbrace}\left\lvert Y^{n,r^\prime}_{\kappa(n,(s+u)^+)} 
\right \rvert^2\right] \lambda(\de u)\mathcal{R}(\de r^\prime)\bigg]\de s
\\
&  \le  \mathbb{E}\left\lvert \hat{z}^\zeta_0\right\rvert^2 + \left(1+C_1(T,{\omega^\prime}) 
\right) \int_0^T \left(K_s({\omega^\prime})+3P\bar{K}_s({\omega^\prime})+P\right)\de s
\\ 
& \le C(T,{\omega^\prime})\, . 
\end{align*}
Here we have used that for $t\in (0,T]$  
$$
\mathbb{E}\lvert X_{t-}^{n,r}\rvert^2 
\leq \liminf_{\delta\downarrow 0} \mathbb{E}\lvert X_{t-\delta}^{n,r}\rvert^2 
\le \sup_{s\in [-\tau ,T]} \mathbb{E}\lvert X_s^{n,r}\rvert^2  \le C_1 (T, \omega '). 
$$ 
%and 
%$$
%\mathbb{E}\lvert p_{t-}^{n,r}\rvert^2 \le 2 \mathbb{E}\lvert X_{\kappa (n,t)}^{n,r}\rvert^2  
%+ 2 \mathbb{E}\lvert X_{t-}^{n,r}\rvert^2  \le 2C_1 (T, \omega ')\, .  
%$$ 
\paragraph*{Proof of \ref{iii}:}
Let 
\[
\tau^* = T\wedge\tau^{n,r}_R\wedge \inf\left\lbrace t\ge 0:\left\lvert X^{n,r}_t\right\rvert 
\ge a\right\rbrace 
\]
in \ref{ii}, we get
\[ 
\mathbb{P}\left\lbrace \sup_{t\in[0,T\wedge\tau^{n,r}_R]}\left\lvert X^{n,r}_t\right\rvert 
\ge a\right\rbrace\le \frac{1}{a^2}\mathbb{E}\left\lvert X^{n,r}_{\tau^*} \right \rvert^2 \le 
\frac{C(T,{\omega^\prime})}{a^2}. 
\]
So
\begin{align*}
\limsup_{R\to\infty} \limsup_{n\to \infty}\, & \mathbb{P}\left\lbrace \sup_{t\in[0,
\tau^{n,r}_R]}\left\lvert X^{n,r}_t\right\rvert\ge \frac{R}{4};\tau^{n,r}_R\le T \right\rbrace  
\le \limsup_{R\to\infty}\limsup_{n\to \infty}\mathbb{P}\left\lbrace 
\sup_{t\in[0,T\wedge\tau^{n,r}_R]}\left\lvert X^{n,r}_t\right\rvert\ge \frac{R}{4}\right\rbrace 
\\  
& = \limsup_{R\to\infty}\limsup_{n\to\infty}\left(\frac{16C(T,{\omega^\prime})}{R^2}\right)=0.
\end{align*}
Hence we have 
\begin{align*}
\limsup_{R\to\infty}\limsup_{n\to \infty}\mathbb{P}\left\lbrace \tau^{n,r}_R\le T\right\rbrace 
%\\ 
& \le\limsup_{R\to\infty}\limsup_{n\to \infty}\mathbb{P}\left\lbrace\sup_{t\in[0,\tau^{n,r}_R]}  
\left\lvert X^{n,r}_t\right\rvert\ge \frac{R}{4};\tau^{n,r}_R\le T \right\rbrace=0
\end{align*}
which completes the proof of \ref{iii}. 

\paragraph*{Proof of \ref{iv}:} Let $\tau^{n,m,r}_R:=T\wedge 
\tau^{n,r}_R \wedge \tau^{m,r}_R$ and let  
\[
\tilde{\tau}^{n,r}_R:=\inf\left\lbrace t\ge 0:\left\lvert Y^{n,r}_t\right\rvert 
 > \frac{R}{3}\right\rbrace , \quad \tilde{\tau}^{n,m,r}_R:=T\wedge \tilde{\tau}^{n,r}_R 
\wedge \tilde{\tau}^{m,r}_R. 
\]
To shorten notations let $u^n_s := (s,r,X_{s^-}^{n,r}, \omega')$ and 
$v^n_s := (s,r,r', X_{s^-}^{n,r})$. Using It\^o's formula, we have for any stopping 
time $\bar{\tau}\le t\wedge\tau^{n,m,r}_R$, $t\in[0,T]$, that
\begin{align*} 
\mathbb{E}\big\lvert X^{n,r}_{\bar{\tau}} 
& - X^{m,r}_{\bar{\tau}}\big\rvert^2 
\le \mathbb{E}\int_0^{\bar{\tau}} 2\left\langle X^{n,r}_{s^-}-X^{m,r}_{s^-} , 
f\left(u_s^n\right)-f\left(u_s^m\right) \right\rangle\de s 
\\ 
&  + \mathbb{E}\int_0^{\bar{\tau}} \left\lvert g\left(u_s^n\right)-g\left(u_s^m\right) 
\right\rvert^2\de s 
+ \mathbb{E}\int_0^{\bar{\tau}}\int_U \left\lvert h\left(u_s^n, \xi\right) 
 -h\left(u_s^m,\xi\right)\right\rvert^2\nu(\de \xi)\de s 
\\ 
&  + \sum_{\alpha=1}^P\mathbb{E}\int_0^{\bar{\tau}} 2\bigg\langle  
X^{n,r}_{s^-}- X^{m,r}_{s^-},
\tilde{\mathbb{E}}\int_{\Gamma_\alpha}\bigg[\theta\left(v_s^n, 
Y^{n, r^\prime}_{\kappa (n,(s-\tau):s)}, \omega^\prime \right) 
%\\ 
%& \qquad\qquad\qquad\qquad\qquad\qquad\qquad\qquad\quad 
-\theta\left(v_s^m, Y^{m,r^\prime}_{\kappa (m,(s-\tau):s)}, \omega^\prime \right) \bigg] 
\mathcal{R}(\de r^\prime)\bigg\rangle\de s 
\\
& + \sum_{\alpha=1}^P\mathbb{E}\int_0^{\bar{\tau}}\bigg\lvert 
\tilde{\mathbb{E}}\int_{\Gamma_\alpha}\bigg[\beta\left(v_s^n, 
Y^{n,r^\prime}_{\kappa (n,(s-\tau):s)},{\omega^\prime}\right) 
%\\ 
%& \quad\qquad\qquad\qquad\qquad\qquad\qquad\qquad\qquad 
-\beta\left(v_s^m , Y^{m,r^\prime}_{\kappa (m,(s-\tau):s)},
\omega^\prime\right)\bigg]\mathcal{R}(\de r^\prime)\bigg\rvert^2\de s 
\\ 
&  + \sum_{\alpha=1}^P\mathbb{E}\int_0^{\bar{\tau}}\int_U\bigg\lvert 
\tilde{\mathbb{E}}\int_{\Gamma_\alpha}\bigg[\eta\left(v_s^n, 
Y^{n,r^\prime}_{\kappa (n, (s-\tau):s)}, \omega^\prime ,\xi\right) 
%\\ 
%& \qquad\qquad\qquad\qquad\qquad\qquad\qquad\qquad\quad 
-\eta\left(v_s^m, Y^{m,r^\prime}_{\kappa (m,(s-\tau):s)}, \omega^\prime,\xi\right) \bigg] 
\mathcal{R}(\de r^\prime)\bigg\rvert^2\nu (\de \xi)\de s\, . 
\end{align*}
Hypothesis \ref{hyp1} implies
\begin{align*}
\mathbb{E}\big\lvert X^{n,r}_{\bar{\tau}} 
& - X^{m,r}_{\bar{\tau}}\big\rvert^2  
\\  
& \le \int_0^t \left(L_s({\omega^\prime}) + P + P 
\bar{L}_s({\omega^\prime}) \right)\sup_{u\in [0,s]} \mathbb{E}  
\left\lvert X^{n,r}_{u\wedge\bar{\tau}}-X^{m,r}_{u\wedge\bar{\tau}}  \right\rvert^2  
\de s \\
& \quad 
+ \sum_{\alpha=1}^P\mathbb{E}\int_0^{\bar{\tau}}  2\bar{L}_s({\omega^\prime}) 
 \tilde{\mathbb{E}}\int_{\Gamma_\alpha} 
\Bigg(\mathbf{1}_{\left\lbrace s\le{\tilde{\tau}}^{n,m,r^\prime}_R 
\right\rbrace}\int_{-\tau}^0\Big[\left\lvert Y^{n,r^\prime}_{\kappa (n,s+u)} 
-Y^{m,r^\prime}_{\kappa (m,s+u)}\right\rvert^2 
\\ 
& \qquad\qquad\qquad\qquad\qquad 
+ \mathbf{1}_{\left\lbrace u<0\right\rbrace}\left\lvert Y^{n,r^\prime}_{\kappa (n,(s+u)^+)} 
- Y^{m,r^\prime}_{\kappa (m,(s+u)^+)} 
\right\rvert^2\Big]\lambda(\de u)\Bigg)\mathcal{R}(\de r^\prime)\de s 
\\ 
& \quad 
+ \sum_{\alpha=1}^P\mathbb{E}\int_0^{\bar{\tau}}\Bigg[4\bar{K}_s ({\omega^\prime}) 
\tilde{\mathbb{E}}\left(\int_{\Gamma_\alpha}  \mathbf{1}_{\left\lbrace s >  
{\tilde{\tau}}^{n,m,r^\prime}_R \right \rbrace}\mathcal{R}(\de r^\prime)\right)\cdot\tilde{\mathbb{E}} \int_{\Gamma_\alpha} 
\Big( 2 + \left\lvert X^{n,r}_{s^-}\right\rvert^2 
+ \left\lvert X^{m,r}_{s^-} \right\rvert^2 +  
\int_{-\tau}^0 \Big[\left \lvert Y^{n,r^\prime}_{\kappa(n,s+u)} \right\rvert^2
\\ 
& \qquad\qquad\qquad  
+\left\lvert Y^{m,r^\prime}_{\kappa(m,s+u)}\right\rvert^2 
+ \mathbf{1}_{\left\lbrace u<0\right\rbrace}\left(\left\lvert Y^{n,r^\prime}_{\kappa(n,
(s+u)^+)}\right\rvert^2 + \left\lvert Y^{m,r^\prime}_{\kappa(m,(s+u)^+)}\right\rvert^2 
\right)\Big] \lambda(\de u) \Big) \mathcal{R}(\de r^\prime ) \Bigg]\de s 
\end{align*}
where we separate the two cases ${\left\lbrace s>{\tilde{\tau}}^{n,m,r^\prime}_R\right\rbrace}$  
and ${\left\lbrace s\le {\tilde{\tau}}^{n,m,r^\prime}_R\right\rbrace}$ in order to apply 
Gronwall's inequality to the difference $\mathbb{E}\big\lvert X^{n,r}_{s\wedge\bar{\tau}} 
- X^{m,r}_{s\wedge\bar{\tau}}\big\rvert^2$. Using \ref{ii} we obtain that
\begin{equation} 
\label{beforegronwall}
\begin{aligned}
%\max\left\lbrace\sup_{t\in[0,T]}
\mathbb{E} & 
%\left\lvert X^{n,r}_{t\wedge\tau^{n,m,r}_R}- 
%X^{m,r}_{t\wedge\tau^{n,m,r}_R}\right\rvert^2,\mathbb{E} 
\left\lvert X^{n,r}_{\bar{\tau}} - X^{m,r}_{\bar{\tau}}\right\rvert^2 
%\right\rbrace 
\\
& \le \Bigg[\int_0^t 6 \bar{L}_s({\omega^\prime}) 
\mathbb{E}\int_{\Gamma}\int_{-\tau}^0\mathbf{1}_{[0,\tau^{n,m,r^\prime}_R]}(s)\left[ 
\left\lvert p^{n,r^\prime}_{s+u}\right\rvert^2 
+ \left\lvert p^{m,r^\prime}_{s+u}\right\rvert^2+\mathbf{1}_{\left\lbrace u<0 \right \rbrace} 
\left(\left\lvert p^{n,r^\prime}_{(s+u)^+}\right\rvert^2 
+ \left\lvert p^{m,r^\prime}_{(s+u)^+}\right\rvert^2\right)\right]\lambda(\de u)  
\mathcal{R}(\de r^\prime)\de s
\\
& \qquad + 8\left(1+3C_1(T,{\omega^\prime})\right) 
\int_0^t\bar{K}_s({\omega^\prime}) 
\int_\Gamma\mathbb{P}\left\lbrace s>\tau^{n,m,r^\prime}_R\right\rbrace  
\mathcal{R}(\de r^\prime)\de s\Bigg]
\\  
& \quad + \int_0^t\left[ \left(L_s({\omega^\prime}) 
+ 13 P \bar{L}_s({\omega^\prime})+P\right)\sup_{r^\prime\in \Gamma} 
\sup_{u\in [0,s]}\mathbb{E}\left\lvert X^{n,r^\prime}_{u\wedge \tau^{n,m,r^\prime}_R} 
-X^{m,r^\prime}_{u\wedge \tau^{n,m,r^\prime}_R}\right\rvert^2\right]\de s \\ 
& = I^{n,m}_{R,T}({\omega^\prime})  
+ \int_0^t\left[ \left(L_s({\omega^\prime}) 
+ 13 P\bar{L}_s({\omega^\prime}) + P\right)\sup_{r^\prime\in \Gamma} 
\sup_{u\in [0,s]}\mathbb{E}\left\lvert X^{n,r^\prime}_{u\wedge \tau^{n,m,r^\prime}_R}-
X^{m,r^\prime}_{u\wedge \tau^{n,m,r^\prime}_R}\right\rvert^2\right]\de s
\end{aligned}
\end{equation} 
Choosing $\bar{\tau} = t\wedge\tau^{n,m,r}_R$ for $t\in [0,T]$ we obtain by Gronwall's 
inequality 
\begin{equation}\label{intgronwall}
\sup_{r\in \Gamma}\sup_{t\in[0,T]}\mathbb{E}\left\lvert X^{n,r}_{t\wedge\tau^{n,m,r}_R}- X^{m,r}_{t\wedge\tau^{n,m,r}_R}\right\rvert^2\le C(T,{\omega^\prime}) I^{n,m}_{R,T}({\omega^\prime}).
\end{equation}
Inserting this bound in the left hand side of \eqref{beforegronwall} implies 
\[
\mathbb{E}\left\lvert X^{n,r}_{\bar{\tau}}- X^{m,r}_{\bar{\tau}}\right\rvert^2 
\le  C(T,{\omega^\prime}) I^{n,m}_{R,T}({\omega^\prime})\, . 
\]
Note that $C(T,{\omega^\prime})$ may differ from a line to another line but always 
$T\mapsto C(T,{\omega^\prime})$ is an increasing function. By setting 
\[
\bar{\tau} := \tau^{n,m,r}_R\wedge \inf\left\lbrace t\ge 0:\left\lvert 
X^{n,r}_t-X^{m,r}_t\right\rvert\ge \varepsilon\right\rbrace, 
\]
we have
\begin{align*}
\mathbb{P}\left\lbrace \sup_{t\in[0,T]}\left\lvert X^{n,r}_t-X^{m,r}_t\right\rvert 
\ge \varepsilon\right\rbrace 
& 
\le \mathbb{P}\left\lbrace T>\tau^{n,r}_R\right\rbrace 
+ \mathbb{P}\left\lbrace T>\tau^{m,r}_R\right\rbrace 
+ \mathbb{P}\left\lbrace \sup_{t\in[0,\tau^{n,m,r}_R]}\left\lvert X^{n,r}_t-X^{m,r}_t \right 
\rvert \ge \varepsilon\right\rbrace 
\\ 
& \le \mathbb{P}\left\lbrace T>\tau^{n,r}_R\right\rbrace  
+ \mathbb{P}\left\lbrace T>\tau^{m,r}_R\right \rbrace + \frac{1}{\varepsilon^2} \mathbb{E} 
\left\lvert X^{n,r}_{\bar{\tau}}- X^{m,r}_{\bar{\tau}}\right\rvert^2 
\\ 
& \le \mathbb{P} \left\lbrace T>\tau^{n,r}_R\right\rbrace 
+ \mathbb{P} \left \lbrace T>\tau^{m,r}_R\right\rbrace 
+ \frac{C(T,{\omega^\prime}) I^{n,m}_{R,T}}{\varepsilon^2}\, .
\end{align*}
From \ref{i} and \ref{iii}, one can obtain that 
\[
\lim_{R\to\infty}\limsup_{n,m\to \infty}I^{n,m}_{R,T}=0 
\]
and so
\begin{align*}
\limsup_{n,m\to\infty}\mathbb{P}\left\lbrace \sup_{t\in[0,T]} 
\left\lvert X^{n,r}_t-X^{m,r}_t\right\rvert\ge \varepsilon\right\rbrace 
\le \lim_{R\to\infty}\limsup_{n,m\to\infty}\left[\mathbb{P} 
\left\lbrace T>\tau^{n,r}_R\right\rbrace 
+ \mathbb{P}\left\lbrace T>\tau^{m,r}_R\right\rbrace 
+ \frac{C(T,{\omega^\prime})I^{n,m}_{R,T}}{\varepsilon^2} \right]=0\, . 
\end{align*}
So \ref{iv} is obtained.

\paragraph*{Proof of \ref{v}:} Since the space 
$L^2\left(\Omega, \text{C\`adl\`ag}([-\tau,T],E)\right)$ is 
complete with respect to the topology of convergence in probability, \ref{iv} yields that there exist  
$X^{r}, Y^r\in L^2\left(\Omega, \text{C\`adl\`ag}([-\tau,T],E)\right)$ such that
\[
\lim_{n\to\infty}\mathbb{P}\left\lbrace \sup_{t\in[0,T]} 
\left\lvert X^{n,r}_t-X^{r}_t\right\rvert\ge \varepsilon\right\rbrace=0,\quad \lim_{n\to\infty}\tilde{\mathbb{P}}\left\lbrace \sup_{t\in[0,T]} 
\left\lvert Y^{n,r}_t-Y^{r}_t\right\rvert\ge \varepsilon\right\rbrace=0. 
\]
We next have to show that all terms of equation \eqref{Xn=phi(Xn+pn)} for a subsequence of 
$n\in\mathbb{N}$ converge almost surely to the terms of equation \eqref{equ2}. We have
\begin{align*}
\lim_{n\to\infty}\mathbb{P} & \left\lbrace \sup_{t\in[0,T]} 
\left\lvert Y^{n,r}_{\kappa(n,t)}-Y^{r}_{t^-}\right\rvert\ge \varepsilon\right\rbrace 
\\ 
& \le \lim_{n\to\infty}\mathbb{P}\left\lbrace \sup_{t\in[0,T]} 
\left\lvert Y^{n,r}_{\kappa(n,t)}-Y^{r}_{\kappa(n,t)}\right\rvert \ge \varepsilon/2\right\rbrace 
 +\lim_{n\to\infty}\mathbb{P}\left\lbrace \sup_{t\in[0,T]} 
\left\lvert Y^{r}_{\kappa(n,t)}-Y^{r}_{t^-}\right\rvert\ge \varepsilon/2\right\rbrace=0 \, . 
\end{align*} 
So there exists a subsequence, say $\left\lbrace n_l\right\rbrace_{l\in\mathbb{N}}$, such that,  
as $l\to\infty$, 
\begin{equation}\label{kappalimit}
\sup_{t\in[0,T]}\left[\left\lvert X^{n_l,r}_{t^-}-X^{r}_{t^-}\right\rvert+\left\lvert Y^{n_l,r}_{\kappa(n_l,t)}-Y^{r}_{t^-}\right\rvert\right]\to 0, 
\quad \mathbb{P}\otimes \tilde{\mathbb{P}}-a.s. 
\end{equation}
for all  $(r,\omega^\prime)$ in a subset $D_0$  of $\Gamma\times \Omega^\prime$ of full $\mathcal{R}\otimes \mathbb{P}^\prime$-measure. Now let us define
\[ 
S_r(t) := \sup_{l\in\mathbb{N}}\left\lvert X^{n_l,r}_{t^-}\right\rvert, \quad  \tilde{S}_r(t) := \sup_{l\in\mathbb{N}}\left\lvert Y^{n_l,r}_{t^-}\right\rvert\, . 
\]
Then 
\[ 
\sup_{t\in[0,T]}S_r(t)<\infty\quad \mathbb{P}-a.s., 
\quad \sup_{t\in[0,T]}\tilde{S}_r(t)<\infty\quad \tilde{\mathbb{P}}-a.s. 
\]
for all $(r,\omega^\prime)\in D_0$. So by  \ref{H4''} and inequality \eqref{bddsquare}, for  $(r,\omega^\prime)\in D_0$, 
\begin{align*}
\MoveEqLeft[5]\int_0^t \tilde{\mathbb{E}}\int_{\Gamma_\alpha} \left\lvert\theta\left( s,r,r^\prime, X^{n_l,r}_{s^-},Y^{n_l,r^\prime}_{\kappa(n_l,(s-\tau):s)},\omega^\prime\right)\right\rvert^{2}\mathcal{R}(\de r^\prime)\de s<\infty \quad \mathbb{P}-a.s.
\end{align*}
Using continuity of $\theta$ and $L^1([0,T]\times\tilde{\Omega}\times\Gamma,\de t \otimes \tilde{\mathbb{P}}\otimes\mathcal{R})$-uniform integrability of \[(s, \tilde{\omega}, r^\prime)\mapsto \theta\left( s,r,r^\prime, X^{n_l,r}_{s^-},Y^{n_l,r^\prime}_{\kappa(n_l,(s-\tau):s)},\omega^\prime\right)\] for all $(r,\omega^\prime)\in D_0$, we 
obtain that
\[
\lim_{l\to\infty}\int_0^t \tilde{\mathbb{E}}\int_{\Gamma_\alpha} \theta\left( s,r,r^\prime, X^{n_l,r}_{s^-},Y^{n_l,r^\prime}_{\kappa(n_l,(s-\tau):s)},\omega^\prime\right)\mathcal{R}(\de r^\prime)\de s 
= \int_0^t \tilde{\mathbb{E}}\int_{\Gamma_\alpha} \theta\left(s,r,r^\prime, X^{r}_{s^-},Y^{r^\prime}_{(s-\tau)^-:s^-},{\omega^\prime}\right)\mathcal{R}(\de r^\prime)\de s\quad 
\]
 $\mathbb{P}$-almost surely for $\mathcal{R}\otimes \mathbb{P}^\prime$-almost all $(r,\omega^\prime)$. Let 
 \[\tau_{r,R}:=\inf\left\lbrace t\ge 0:S_r(t)>R\right\rbrace 
\wedge T, \quad \tilde{\tau}_{r,R}:=\inf\left\lbrace t\ge 0:\tilde{S}_r(t)>R\right\rbrace 
\wedge T.\]
 For all $t\in [0,T]$ and for all $(r,\omega^\prime)\in D_0$, we have
\begin{align*}
\MoveEqLeft[0]\mathbb{E}\left\lvert  
\int_0^{t\wedge \tau_{r,R}}\tilde{\mathbb{E}}\int_{\Gamma_\alpha}\left[\beta\left( s,r,r^\prime, X^{n_l,r}_{s^-},Y^{n_l,r^\prime}_{\kappa(n_l,(s-\tau):s)},\omega^\prime\right)
-\beta\left( s,r,r^\prime, X^{r}_{s^-},Y^{r^\prime}_{(s-\tau)^-:s^-},\omega^\prime\right)\right]\mathcal{R}(\de r^\prime)\de 
B^\alpha_s\right\rvert^{2} 
\\
&\leq 2\,\mathbb{E}\int_0^t\tilde{\mathbb{E}}\int_{\Gamma_\alpha}\mathbf{1}_{\left\lbrace s\leq \tilde{\tau}_{r,R}\wedge \tau_{r,R}\right\rbrace}\left\lvert\beta\left( s,r,r^\prime, X^{n_l,r}_{s^-},Y^{n_l,r^\prime}_{\kappa(n_l,(s-\tau):s)},\omega^\prime\right)
-\beta\left( s,r,r^\prime, X^{r}_{s^-},Y^{r^\prime}_{(s-\tau)^-:s^-},\omega^\prime\right)\right\rvert^2\mathcal{R}(\de r^\prime)\de  s\\&\quad + 4\,\mathbb{E}\int_0^t\mathbf{1}_{\left\lbrace s\leq  \tau_{r,R}\right\rbrace}\bar{K}_s(\omega^\prime)(2+\left\lvert X^{n_l,r}_{s^-}\right\rvert^2+\left\lvert X^{r}_{s^-}\right\rvert^2 +4C_1(s,\omega^\prime))\tilde{\mathbb{E}}\int_{\Gamma_\alpha}\mathbf{1}_{\left\lbrace s> \tilde{\tau}_{r,R}\right\rbrace}\mathcal{R}(\de r^\prime) \de s\, . 
\end{align*}
So
\begin{align*}
\MoveEqLeft[1]\mathbb{P}\left\lbrace \left\lvert 
\int_0^t\tilde{\mathbb{E}}\int_{\Gamma_\alpha}\left[\beta\left( s,r,r^\prime, X^{n_l,r}_{s^-},Y^{n_l,r^\prime}_{\kappa(n_l,(s-\tau):s)},\omega^\prime\right)
-\beta\left( s,r,r^\prime, X^{r}_{s^-},Y^{r^\prime}_{(s-\tau)^-:s^-},\omega^\prime\right)\right]\mathcal{R}(\de r^\prime)\de 
B^\alpha_s\right\rvert > \varepsilon\right\rbrace 
\\ 
& \le \frac{1}{\varepsilon^2}\mathbb{E}\left\lvert 
\int_0^{t\wedge\tau_{r,R}}\tilde{\mathbb{E}}\int_{\Gamma_\alpha}\left[\beta\left( s,r,r^\prime, X^{n_l,r}_{s^-},Y^{n_l,r^\prime}_{\kappa(n_l,(s-\tau):s)},\omega^\prime\right)
-\beta\left( s,r,r^\prime, X^{r}_{s^-},Y^{r^\prime}_{(s-\tau)^-:s^-},\omega^\prime\right)\right]\mathcal{R}(\de r^\prime)\de 
B^\alpha_s\right \rvert^2 
\\&\quad+\mathbb{P}\left\lbrace t > \tau_{r,R} \right \rbrace.
\\& \leq \frac{2}{\varepsilon^2}\mathbb{E}\int_0^t\tilde{\mathbb{E}}\int_{\Gamma_\alpha}\mathbf{1}_{\left\lbrace s\leq \tilde{\tau}_{r,R}\wedge \tau_{r,R}\right\rbrace}\left\lvert\beta\left( s,r,r^\prime, X^{n_l,r}_{s^-},Y^{n_l,r^\prime}_{\kappa(n_l,(s-\tau):s)},\omega^\prime\right)
-\beta\left( s,r,r^\prime, X^{r}_{s^-},Y^{r^\prime}_{(s-\tau)^-:s^-},\omega^\prime\right)\right\rvert^2\mathcal{R}(\de r^\prime)\de  s\\&\quad + \left[\frac{8}{\varepsilon^2}\int_0^t \bar{K}_s(\omega^\prime) (1+3C_1(s,\omega^\prime))\int_{\Gamma_\alpha} \mathbb{P}\left\lbrace s> \tilde{\tau}_{r,R}\right\rbrace\mathcal{R}(\de r^\prime)\de s+\mathbb{P}\left\lbrace t > \tau_{r,R} \right \rbrace\right]
\end{align*}
For given $\delta > 0$ we can now find $R$ sufficiently large, such that the second term on 
the right hand side is less than $\delta$. Taking the limit $l\to\infty$ now implies that 
\[ 
\lim_{l\to \infty}\mathbb{P}\left\lbrace \left\lvert 
\int_0^t\tilde{\mathbb{E}}\int_{\Gamma_\alpha}\left[\beta\left( s,r,r^\prime, X^{n_l,r}_{s^-},Y^{n_l,r^\prime}_{\kappa(n_l,(s-\tau):s)},\omega^\prime\right)
-\beta\left( s,r,r^\prime, X^{r}_{s^-},Y^{r^\prime}_{(s-\tau)^-:s^-},\omega^\prime\right)\right]\mathcal{R}(\de r^\prime)\de 
B^\alpha_s\right\rvert>\varepsilon\right\rbrace\le \delta \, .
\]
Therefore 
\[
\int_0^t\tilde{\mathbb{E}}\int_{\Gamma_\alpha}\beta\left( s,r,r^\prime, X^{n_l,r}_{s^-},Y^{n_l,r^\prime}_{\kappa(n_l,(s-\tau):s)},\omega^\prime\right)
\mathcal{R}(\de r^\prime)\de 
B^\alpha_s\to 
\int_0^t\tilde{\mathbb{E}}\int_{\Gamma_\alpha}\beta\left( s,r,r^\prime, X^{r}_{s^-},Y^{r^\prime}_{(s-\tau)^-:s^-},\omega^\prime\right)\mathcal{R}(\de r^\prime)\de 
B^\alpha_s
\]
in probability. The same argument implies
\begin{align*}
\MoveEqLeft[1]\int_0^t\int_U \tilde{\mathbb{E}}\int_{\Gamma_\alpha}\eta\left( s,r,r^\prime, X^{n_l,r}_{s^-},Y^{n_l,r^\prime}_{\kappa(n_l,(s-\tau):s)}, \omega^\prime,\xi\right)
\mathcal{R}(\de r^\prime)\tilde{N}^\alpha(\de s,
\de \xi)\\&\to\int_0^t\int_U \tilde{\mathbb{E}}\int_{\Gamma_\alpha}\eta\left( s,r,r^\prime, X^{r}_{s^-},Y^{r^\prime}_{(s-\tau)^-:s^-}, \omega^\prime,\xi\right)
\mathcal{R}(\de r^\prime)\tilde{N}^\alpha(\de s,
\de \xi)\quad \text{in probability} 
\end{align*}
and for some further subsequence $n_{l_k}$ the above convergences are $\mathbb{P}-a.s.$ 
The convergence of the terms concerning the local dynamics in \eqref{Xn=phi(Xn+pn)} to the respective terms 
of  \eqref{equ2} follow from dominated convergence for the stopped solution (using $\tau_{r,R}$) and \ref{H6}. 
Therefore $X$ is a solution of equation \eqref{equ2} for $\mathcal{R}\times \mathbb{P}^\prime$-almost all $(r,\omega^\prime)$. 

\paragraph*{Uniqueness:} Let $X$ and $Y$ be two strong solutions of equation \eqref{equ2}. 
To shorten the notation again, let $u^r_s = (s,r,X_{s^-}^r )$, $u^{r,r'}_s  
= (s,r,r',X_{s^-}^r )$, $v_s^r = (s,r,Y_{s^-}^r)$ and $v_s^{r,r'} = (s,r,r',Y_{s^-}^r)$. We 
then have
\begin{align*}
\big\lvert X^{r}_t & - Y^{r}_t\big\rvert^2 
\\ 
& = M_t + \int_0^t \big[ 2\left\langle X^{r} _{s^-}-Y^{r} _{s^-},f\left( u_s^r, 
\omega^\prime\right) - f\left( v_s^r , \omega^\prime\right) \right \rangle 
+ \left\lvert g\left( u_s^r , \omega^\prime \right) - g\left(v_s^r ,\omega^\prime \right) 
  \right \rvert^2\big]\de s  
\\ 
& \quad 
+ \int_0^t\int_U \left\lvert h\left( u_s^r ,{\omega^\prime},\xi\right) - h \left( v_s^r ,
\omega^\prime , \xi \right) \right\rvert^2 N (\de s,\de \xi) 
\\ 
& \quad 
+ \sum_{\alpha=1}^P\int_0^t\bigg[2\left\langle X^{r} _{s^-}-Y^{r} _{s^-},
\int_{\Gamma_\alpha}\tilde{\mathbb{E}}\left[\theta\left(u_s^{r,r'},
\tilde{X}^{r^\prime}_{(s-\tau)^-:s^-},{\omega^\prime}\right)  
- \theta \left( v_s^{r,r'}, \tilde{Y}^{r^\prime}_{(s-\tau)^-:s^-},\omega^\prime 
\right) \right] \mathcal{R}(\de r^\prime)\right\rangle 
\\ 
& \qquad\qquad\qquad 
+ \left\lvert \int_{\Gamma_\alpha}\tilde{\mathbb{E}}\left[\beta\left(u_s^{r,r'},
\tilde{X}^{r^\prime}_{(s-\tau)^-:s^-},{\omega^\prime}\right)-\beta\left(v_s^{r,r'},
\tilde{Y}^{r^\prime}_{(s-\tau)^-:s^-},{\omega^\prime}\right)\right]\mathcal{R}(\de 
r^\prime)\right\rvert^2\bigg]\de s 
\\ 
& \quad 
+ \sum_{\alpha=1}^P\int_0^t\int_U\left\lvert \int_{\Gamma_\alpha} \tilde{\mathbb{E}} 
\left[\eta\left(u_s^{r,r'},\tilde{X}^{r^\prime}_{(s-\tau)^-:s^-},{\omega^\prime},\xi \right) 
- \eta\left(v_s^{r,r'},\tilde{Y}^{r^\prime}_{(s-\tau)^-:s^-},{\omega^\prime}, \xi \right) 
\right] \mathcal{R}(\de r^\prime)\right\rvert^2 N^\alpha(\de s,\de \xi)
\end{align*}
where $\left(\tilde{X},\tilde{Y}\right)$ are independent copies of $\left(X,Y\right)$ and $M_t$ 
is a local martingale w.r.t. some localizing sequence $\sigma_n$, $n\ge 1$, of stopping times. 
Using Fatou's lemma and Hypothesis \ref{hyp1} we then have
\begin{align*}
\mathbb{E}\left\lvert X^{r} _t-Y^{r} _t\right\rvert^2 
& \le  \liminf_{l\to\infty}\mathbb{E}\left\lvert X^{r} _{t\wedge\sigma_l} 
-Y^{r}_{t\wedge\sigma_l}\right\rvert^2 
\\ 
& \le \mathbb{E} \int_0^t \left( L_s({\omega^\prime}) + P\bar{L}_s ({\omega^\prime}) + P\right) 
\left\lvert X^{r} _{s^-}-Y^{r} _{s^-}\right\rvert^2\de s
\\ 
& \quad  
+ \int_0^t\bar{L}_s({\omega^\prime})\tilde{\mathbb{E}}\int_\Gamma\int_{-\tau}^0 
\left[\left \lvert \tilde{X}^{r^\prime}_{(s+u)^-}-\tilde{Y}^{r^\prime}_{(s+u)^-}\right\rvert^2 
+ \mathbf{1}_{\left\lbrace u<0\right\rbrace} \left \lvert \tilde{X}^{r^\prime}_{s+u} 
-\tilde{Y}^{r^\prime}_{s+u} \right\rvert^2\right] \lambda(\de u)\mathcal{R}(\de r^\prime)\de s
\end{align*}
so
\begin{align*}
\sup_{s\le t}\mathbb{E}\int_\Gamma\left\lvert X^{r}_s-Y^{r}_s\right\rvert^2\mathcal{R}(\de r) 
\le  
& 
\int_0^t\left(L_s({\omega^\prime})+3P\bar{L}_s({\omega^\prime})+P\right) 
\sup_{u\le s}\mathbb{E}\int_\Gamma\left\lvert X^{r}_u-Y^{r}_u\right\rvert^2\mathcal{R}(\de r) 
\de s\, . 
\end{align*}
By Gronwall's lemma and Lemma \ref{boundedness} we have
\[
\sup_{s\le T}\mathbb{E}\int_\Gamma\left\lvert X^{r}_s-Y^{r}_s\right\rvert^2 
\mathcal{R}(\de r)=0  \, , 
\] 
which proves uniqueness.
\end{proof}

\section{Proof of Theorem \ref{PropagationChaos}}
\label{ProofThmPropagationChaos}

\begin{proof}
Let us first introduce the notation  
\[
\avint{A} \psi(r^\prime)\mathcal{R}(\de r^\prime) 
:= \begin{dcases}\frac{1}{\mathcal{R}(A)}\int_A \psi(r^\prime)\mathcal{R}(\de r^\prime), 
& \mathcal{R}(A)\neq 0\, ,  
\\ 0, & \mathcal{R}(A)= 0\, . 
\end{dcases}
\]
To shorten the notation again, let $u^r_s = (s,r,X_{s^-}^{r,\mathcal{A}_N}, \omega')$, 
$u^{r,\tilde{r}}_s = (s,r,\tilde{r}, X_{s^-}^{r,\mathcal{A}_N}$, $X_{(s-\tau)^-:s^-}^{\tilde{r},
\mathcal{A}_N}, \omega')$, $v^r_s = (s,r,\bar{X}_{s^-}^{r}, \omega')$ and 
$v^{r,r'}_s = (s,r,r', \bar{X}_{s^-}^{r}, \tilde{X}_{(s-\tau)^-:s^-}^{r'}, \omega')$.  
Then 
\begin{align*}
\left\lvert X^{r,\mathcal{A}_N}_t-\bar{X}^r_t\right\rvert^2 
& =M_t+\int_0^t\bigg[2\left\langle X^{r,\mathcal{A}_N}_{s^-}-\bar{X}^r_{s^-}, 
f\left( u^r_s\right)-f\left(v^r_s \right)\right\rangle 
+ \left\lvert g\left(u^r_s\right) - g\left(v^r_s\right)\right\rvert^2\bigg]\de s 
\\ 
& \quad + \int_0^t\int_U\left\lvert h\left(u^r_s,\xi\right)- h\left(v^r_s,\xi \right) 
\right\rvert^2{N}^r(\de s,\de \xi)
\\ 
& \quad 
+ \sum_{\alpha=1}^P\int_0^t2\Bigg\langle  X^{r,\mathcal{A}_N}_{s^-}-\bar{X}^r_{s^-}, 
\frac{1}{\mathcal{S}_{\mathcal{A}_N,\alpha}}\sum_{\tilde{r}\in \mathcal{A}_N\cap  
\Gamma_\alpha }\theta\left(u_s^{r,\tilde{r}}\right) 
-\tilde{\mathbb{E}} \int_{\Gamma_\alpha}\theta\left(v_s^{r,r'} \right) 
\mathcal{R}(\de r^\prime)\Bigg\rangle\de s 
\\ 
& \quad 
+ \sum_{\alpha=1}^P\int_0^t\Bigg\lvert \frac{1}{\mathcal{S}_{\mathcal{A}_N,
\alpha}}\sum_{\tilde{r}\in \mathcal{A}_N\cap \Gamma_\alpha }\beta\left( 
u_s^{r,\tilde{r}} \right) 
- \tilde{\mathbb{E}} \int_{\Gamma_\alpha}\beta \left(v_s^{r,r'}\right) 
\mathcal{R}(\de r^\prime)\Bigg\rvert^2\de s \\
& \quad 
 + \sum_{\alpha=1}^P\int_0^t\int_U\Bigg\lvert \frac{1}{\mathcal{S}_{\mathcal{A}_N,
\alpha}}\sum_{\tilde{r}\in \mathcal{A}_N\cap \Gamma_\alpha } 
\eta\left(u_s^{r,\tilde{r}},\xi\right) 
- \tilde{\mathbb{E}}\int_{\Gamma_\alpha}\eta \left(v_s^{r,r'},\xi\right) 
\mathcal{R} (\de r^\prime)\Bigg\rvert^2N^{r,\alpha}(\de s,\de \xi)
\end{align*}
where $M_t$ is a local martingale up to time $T$ starting from zero with localizing sequence 
$\sigma_n$, $n\ge 1$, of stopping times. Taking expectation, using Fatou's lemma and Hypothesis 
\ref{hyp1} we obtain that 
\begin{equation}
\label{thm1.6:eq1}
\begin{aligned} 
\mathbb{E} \big\lvert X^{r,\mathcal{A}_N}_t & -\bar{X}^r_t\big\rvert^2 
 \le \liminf_{n\to\infty}  \mathbb{E} \left\lvert X^{r,\mathcal{A}_N}_{t\wedge\sigma_n} 
- \bar{X}^r_{t\wedge\sigma_n} \right\rvert^2 
\\
& \le \int_0^t L_s (\omega') \mathbb{E} \left\lvert X^{r,\mathcal{A}_N}_{s^-}-\bar{X}^r_{s^-} 
\right\rvert^2 \de s 
\\
& \qquad + P \int_0^t \mathbb{E}  \left\lvert X^{r,\mathcal{A}_N}_{s^-}-\bar{X}^r_{s^-} 
\right\rvert^2 \de s + \sum_{\alpha=1}^P \int_0^t \mathbb{E} \left\lvert
\frac{1}{\mathcal{S}_{\mathcal{A}_N,\alpha}}\sum_{\tilde{r}\in \mathcal{A}_N\cap  
\Gamma_\alpha }\theta\left(u_s^{r,\tilde{r}}\right) 
-\tilde{\mathbb{E}} \int_{\Gamma_\alpha}\theta\left(v_s^{r,r'} \right) 
\mathcal{R}(\de r^\prime)\right\rvert^2 \de s  
\\
& \qquad + \sum_{\alpha=1}^P\int_0^t\mathbb{E} \Bigg\lvert \frac{1}{\mathcal{S}_{\mathcal{A}_N,
\alpha}}\sum_{\tilde{r}\in \mathcal{A}_N\cap \Gamma_\alpha }\beta\left( 
u_s^{r,\tilde{r}} \right) 
- \tilde{\mathbb{E}} \int_{\Gamma_\alpha}\beta \left(v_s^{r,r'}\right) 
\mathcal{R}(\de r^\prime)\Bigg\rvert^2\de s 
\\
& \qquad + \sum_{\alpha=1}^P\int_0^t\int_U\mathbb{E} \Bigg\lvert \frac{1}{\mathcal{S}_{\mathcal{A}_N,
\alpha}}\sum_{\tilde{r}\in \mathcal{A}_N\cap \Gamma_\alpha } 
\eta\left(u_s^{r,\tilde{r}},\xi\right) 
- \tilde{\mathbb{E}}\int_{\Gamma_\alpha}\eta \left(v_s^{r,r'},\xi\right) 
\mathcal{R} (\de r^\prime)\Bigg\rvert^2\nu (\de \xi) \de s\, . 
\end{aligned}
\end{equation} 
Note that 
\begin{align*}
\sum_{\alpha=1}^P  \mathbb{E}\int_0^t & \left\lvert
\frac{1}{\mathcal{S}_{\mathcal{A}_N,\alpha}}\sum_{\tilde{r}\in \mathcal{A}_N\cap  
\Gamma_\alpha }\theta\left(u_s^{r,\tilde{r}}\right) 
-\tilde{\mathbb{E}} \int_{\Gamma_\alpha}\theta\left(v_s^{r,r'} \right) 
\mathcal{R}(\de r^\prime)\right\rvert^2 \de s 
\\
& \le 2\sum_{\alpha=1}^P  \mathbb{E}\int_0^t  \left\lvert
\frac{1}{\mathcal{S}_{\mathcal{A}_N,\alpha}}\sum_{\tilde{r}\in \mathcal{A}_N\cap  
\Gamma_\alpha } \left( \theta\left(u_s^{r,\tilde{r}}\right) 
- \theta\left(s,r,\tilde{r},\bar{X}^r_{s^-},\bar{X}^{\tilde{r}}_{(s-\tau)^-:s^-},{\omega^\prime} 
\right)\right) \right\rvert^2 \de s 
\\
& \qquad\qquad 
+ 2\sum_{\alpha=1}^P  \mathbb{E}\int_0^t \left\lvert
\frac{1}{\mathcal{S}_{\mathcal{A}_N,\alpha}}\sum_{\tilde{r}\in \mathcal{A}_N\cap  
\Gamma_\alpha }\theta\left(s,r,\tilde{r},\bar{X}^r_{s^-},\bar{X}^{\tilde{r}}_{(s-\tau)^-:s^-},
{\omega^\prime} \right) - \tilde{\mathbb{E}} \int_{\Gamma_\alpha}\theta\left(v_s^{r,r'} \right) 
\mathcal{R}(\de r^\prime)\right\rvert^2 \de s 
\\
& \le 2\, \mathbb{E} \int_0^t \bar{L}_s (\omega ') \frac{\#\mathcal{A}_N\cap\Gamma_\alpha}
{\mathcal{S}_{\mathcal{A}_N,\alpha}^2}\sum_{\tilde{r}\in \mathcal{A}_N\cap \Gamma_\alpha }
\Big[  |X_{s^-}^{r, \mathcal{A}_N}- \bar{X}^r_{s^-}|^2 
\\
& \qquad\qquad 
+ \int_{-\tau}^0\left(\left\lvert X^{\tilde{r},\mathcal{A}_N}_{(s+u)^-} 
- \bar{X}^{\tilde{r}}_{(s+u)^-} \right\rvert^2+\mathbf{1}_{\left\lbrace u<0 \right \rbrace} 
\left\lvert X^{\tilde{r},\mathcal{A}_N}_{s+u} - \bar{X}^{\tilde{r}}_{s+u} \right\rvert^2 
\right)\lambda(\de u)\Big] \de s 
 + 2 \sum_{\alpha=1}^P I_\alpha^\theta 
\end{align*}
with 
$$ 
I_\alpha^\theta :=  \mathbb{E}\int_0^t \left\lvert
\frac{1}{\mathcal{S}_{\mathcal{A}_N,\alpha}}\sum_{\tilde{r}\in \mathcal{A}_N\cap  
\Gamma_\alpha }\theta\left(s,r,\tilde{r},\bar{X}^r_{s^-},\bar{X}^{\tilde{r}}_{(s-\tau)^-:s^-},
{\omega^\prime} \right) - \tilde{\mathbb{E}} \int_{\Gamma_\alpha}\theta\left(v_s^{r,r'} \right) 
\mathcal{R}(\de r^\prime)\right\rvert^2 \de s \, . 
$$ 
The remaining terms on the right hand side of \eqref{thm1.6:eq1} can be estimated from above 
similarly so that \eqref{thm1.6:eq1} now yields the following estimate 
\begin{equation}
\label{thm1.6:eq2}
\begin{aligned} 
\mathbb{E} \big\lvert X^{r,\mathcal{A}_N}_t  -\bar{X}^r_t\big\rvert^2 
& \le \int_0^t (L_s (\omega') + P) \mathbb{E} \left\lvert X^{r,\mathcal{A}_N}_{s^-} 
- \bar{X}^r_{s^-} \right\rvert^2 \de s  
+ 2\sum_{\Theta\in\left\lbrace \theta,\beta\right\rbrace}\sum_{\alpha=1}^P I^{\Theta}_\alpha 
+ 2\sum_{\alpha=1}^P I^\eta_\alpha
\\ 
& \qquad  
+ 2 \sum_{\alpha=1}^P \int_0^t \bar{L}_s (\omega^\prime )  
\frac{\# \mathcal{A}_N\cap\Gamma_\alpha}{\mathcal{S}_{\mathcal{A}_N,\alpha}^2} 
\sum_{\tilde{r}\in \mathcal{A}_N\cap \Gamma_\alpha }\mathbb{E} \Bigg[ \left\lvert X^{r,\mathcal{A}_N}_{s^-}-\bar{X}^r_{s^-}\right\rvert^2
\\ 
& \qquad\qquad\qquad 
+ \int_{-\tau}^0\left(\left\lvert X^{\tilde{r},\mathcal{A}_N}_{(s+u)^-} 
- \bar{X}^{\tilde{r}}_{(s+u)^-}\right\rvert^2 + \mathbf{1}_{\left\lbrace u<0\right\rbrace}\left\lvert X^{\tilde{r},\mathcal{A}_N}_{s+u}-\bar{X}^{\tilde{r}}_{s+u}\right\rvert^2\right)\lambda(\de u)\Bigg]\de s\, . 
\end{aligned}
\end{equation} 
where
\begin{align*}
I_{\alpha}^\Theta 
& = \mathbb{E}\int_0^t\left\lvert \frac{1}{\mathcal{S}_{\mathcal{A}_N,\alpha}}  
\sum_{\tilde{r}\in \mathcal{A}_N\cap \Gamma_\alpha }\Theta\left(s,r,\tilde{r}, 
\bar{X}^r_{s^-},\bar{X}^{\tilde{r}}_{(s-\tau)^-:s^-},\omega^\prime \right) - \tilde{\mathbb{E}} 
\int_{\Gamma_\alpha}\Theta\left(s,r,r^\prime,\bar{X}^r_{s^-},
\tilde{X}^{r^\prime}_{(s-\tau)^-:s^-}, \omega^\prime\right)\mathcal{R} (\de r^\prime) 
\right\rvert^2\de s
\\ 
& \le 3\mathbb{E}\int_0^t \left\lvert \frac{1}{\mathcal{S}_{\mathcal{A}_N,\alpha}} 
\sum_{\tilde{r} \in \mathcal{A}_N\cap \Gamma_\alpha }\left[\Theta\left(s,r,\tilde{r},
\bar{X}^r_{s^-},\bar{X}^{\tilde{r}}_{(s-\tau)^-:s^-},\omega^\prime \right) - \tilde{\mathbb{E}} \Theta 
\left(s,r,\tilde{r},\bar{X}^r_{s^-},\tilde{X}^{\tilde{r}}_{(s-\tau)^-:s^-},\omega^\prime \right) 
\right] \right\rvert^2\de s 
\\ 
& \quad
 + 3 \mathbb{E}\int_0^t \Bigg\lvert \sum_{m=1}^{M_\alpha^{(\varepsilon)}}\frac{1}{\mathcal{S}_{\mathcal{A}_N,
\alpha}} \sum_{\tilde{r}\in \mathcal{A}_N\cap\Gamma^{m,\varepsilon}_\alpha} \tilde{\mathbb{E}} 
\Bigg[\Theta\left(s,r,\tilde{r},\bar{X}^r_{s^-},\tilde{X}^{\tilde{r}}_{(s-\tau)^-:s^-},
\omega^\prime\right)
\\ 
& \qquad\qquad\qquad\qquad\qquad\qquad  
-\avint{\Gamma_\alpha^{m,\varepsilon}}\Theta\left(s,r,r^\prime,\bar{X}^r_{s^-},
\tilde{X}^{r^\prime}_{(s-\tau)^-:s^-}, \omega^\prime\right)\mathcal{R}(\de r^\prime)\Bigg] 
\Bigg\rvert^2 \de s 
\\ 
& \quad
+ 3 \mathbb{E}\int_0^t \left\lvert \sum_{m=1}^{M_\alpha^{(\varepsilon)}} 
\left(\frac{\# \mathcal{A}_N \cap \Gamma^{m,\varepsilon}_\alpha } 
{\mathcal{S}_{\mathcal{A}_N,\alpha}} - \mathcal{R} \left(\Gamma^{m,\varepsilon}_\alpha 
\right)\right)\tilde{\mathbb{E}}\avint{\Gamma_\alpha^{m, \varepsilon}} 
\Theta\left(s,r,r^\prime,\bar{X}^r_{s^-},\tilde{X}^{r^\prime}_{(s-\tau)^-:s^-},
\omega^\prime\right)\mathcal{R}(\de r^\prime)\right\rvert^2 \de s \\ 
& = I + II + III\, , \text{ say.} 
\end{align*}
We can now further estimate the integrals $I-III$ from above as follows:  
\begin{equation} 
\label{thmPC:eq1}
\begin{aligned}
I & = \frac 3{\mathcal{S}_{\mathcal{A}_N,\alpha}^2} 
\sum_{\tilde{r}_1 , \tilde{r}_2\in\mathcal A_N\cap\Gamma_\alpha} \int_0^t 
\mathbb{E} \text{ tr } \Big[\Big( \Theta \left( s,r, \tilde{r}_1, \bar{X}^r_{s^-}, 
\bar{X}^{\tilde{r}_1}_{(s-\tau)^-:s^-},{\omega^\prime}\right)  - 
\tilde{\mathbb{E}} \Theta \left( s,r, \tilde{r}_1, \bar{X}^r_{s^-},  
\tilde{X}^{\tilde{r}_1}_{(s-\tau)^-:s^-}, {\omega^\prime}\right) \Big)^T \\ 
& \qquad\qquad\qquad\qquad\qquad 
\Big( \Theta \left( s,r, \tilde{r}_2, \bar{X}^r_{s^-}, 
\bar{X}^{\tilde{r}_2}_{(s-\tau)^-:s^-},{\omega^\prime}\right)  - 
\tilde{\mathbb{E}} \Theta \left( s,r, \tilde{r}_2, \bar{X}^r_{s^-}, 
\tilde{X}^{\tilde{r}_2}_{(s-\tau)^-:s^-} ,
{\omega^\prime}\right) \Big)\Big] \de s \\
& =\frac 3{\mathcal{S}_{\mathcal{A}_N,\alpha}^2} 
\sum_{\tilde{r}\in \mathcal{A}_N\cap\Gamma_\alpha} 
\int_0^t 
\mathbb{E} \Big[\left\lvert \Theta \left( s,r, \tilde{r}, \bar{X}^r_{s^-}, 
\bar{X}^{\tilde{r}}_{(s-\tau)^-:s^-},{\omega^\prime}\right)  - 
\tilde{\mathbb{E}} \Theta \left( s,r, \tilde{r}, \bar{X}^r_{s^-},  
\tilde{X}^{\tilde{r}}_{(s-\tau)^-:s^-} ,{\omega^\prime}\right) \right\rvert^2 \Big]  
\de s 
\end{aligned}
\end{equation} 
since for distinct $\tilde{r}_1$ and $\tilde{r}_2$, if $\tilde{r}_1\neq r$ or $\tilde{r}_2 
\neq r$ then $\bar{X}^{\tilde{r}_1}$ or $\bar{X}^{\tilde{r}_2}$ are independent of each other 
and also independent of $\bar{X}^r$, and therefore for arbitrary $i$-th component of $\Theta$, 
\begin{align*} 
\mathbb{E} \Big[\Big( & \Theta_i \left( s,r, \tilde{r}_1, \bar{X}^r_{s^-}, 
\bar{X}^{\tilde{r}_1}_{(s-\tau)^-:s^-},{\omega^\prime}\right)  - 
\tilde{\mathbb{E}} \Theta_i \left( s,r, \tilde{r}_1, \bar{X}^r_{s^-},  
\tilde{X}^{\tilde{r}_1}_{(s-\tau)^-:s^-} ,{\omega^\prime}\right) \Big) \\
& \qquad\qquad\qquad 
\times\Big( \Theta_i \left( s,r, \tilde{r}_2, \bar{X}^r_{s^-}, 
\bar{X}^{\tilde{r}_2}_{(s-\tau)^-:s^-},{\omega^\prime}\right)  - 
\tilde{\mathbb{E}} \Theta_i \left( s,r, \tilde{r}_2, \bar{X}^r_{s^-}, 
\tilde{X}^{\tilde{r}_2}_{(s-\tau)^-:s^-} ,
{\omega^\prime}\right) \Big) \mid \bar{X}^r \Big] 
\\ 
& = \mathbb{E} \Big[ \Big( \Theta_i \left( s,r, \tilde{r}_1, \bar{X}^r_{s^-}, 
\bar{X}^{\tilde{r}_1}_{(s-\tau)^-:s^-},{\omega^\prime}\right)  - 
\tilde{\mathbb{E}} \Theta_i \left( s,r, \tilde{r}_1, \bar{X}^r_{s^-},  
\tilde{X}^{\tilde{r}_1}_{(s-\tau)^-:s^-} ,{\omega^\prime}\right) \Big)\mid \bar{X}^r \Big] \\ 
& \qquad\qquad\qquad 
\times \mathbb{E} \Big[
\Big( \Theta_i \left( s,r, \tilde{r}_2, \bar{X}^r_{s^-}, 
\bar{X}^{\tilde{r}_2}_{(s-\tau)^-:s^-},{\omega^\prime}\right)  - 
\tilde{\mathbb{E}} \Theta_i \left( s,r, \tilde{r}_2, \bar{X}^r_{s^-}, 
\tilde{X}^{\tilde{r}_2}_{(s-\tau)^-:s^-} ,
{\omega^\prime}\right) \Big) \mid \bar{X}^r \Big]
= 0\, . 
\end{align*} 
Using Lemma \ref{boundedness} we can then further estimate the right hand side of 
\eqref{thmPC:eq1} from above by  
\begin{align*} 
\hphantom{kl} & \frac 3{\mathcal{S}_{\mathcal{A}_N,\alpha}^2} 
\sum_{\tilde{r}\in \mathcal{A}_N\cap\Gamma_\alpha} 
\int_0^t 
\mathbb{E} \Big[\left\lvert \Theta \left( s,r, \tilde{r}, \bar{X}^r_{s^-}, 
\bar{X}^{\tilde{r}}_{(s-\tau)^-:s^-},{\omega^\prime}\right)  - 
\tilde{\mathbb{E}} \Theta \left( s,r, \tilde{r}, \bar{X}^r_{s^-},  
\tilde{X}^{\tilde{r}}_{(s-\tau)^-:s^-}, {\omega^\prime}\right) \right\rvert^2 \Big]  
\de s 
\\ 
& \qquad\le 
\frac{12}{\mathcal{S}_{\mathcal{A}_N,\alpha}^2}
\sum_{\tilde{r}\in {\mathcal A}_N\cap \Gamma_\alpha} 
\int_0^t \bar{K}_s (\omega' ) 
\left[1+ \mathbb{E} \lvert \bar{X}^r_{s^-}\rvert^2  
+\int_{-\tau}^0 \mathbb{E} \lvert\bar{X}^r_{(s+u)^-}\rvert^2 
+ \mathbf{1}_{\{ u < 0\}} \mathbb{E} \lvert\bar{X}^r_{s+u}\rvert^2 
\lambda (du)\right]\de s 
\\ 
& \qquad\le 
\frac{\#{\mathcal A}_N\cap \Gamma_\alpha}{\mathcal{S}_{\mathcal{A}_N,\alpha}^2}  
\int_0^t 12 \bar{K}_s (\omega' ) ( 1 + 3C_1 (s, \omega '))\de s  
\end{align*}
The next term can be estimated from above as follows: 
\begin{align*} 
II 
& \le 3 \int_0^t \sum_{m=1}^{M_\alpha^{(\varepsilon)}} \sum_{\tilde{r}\in \mathcal{A}_N  
\cap \Gamma_\alpha^{m,\varepsilon}} \bar{L}_s({\omega^\prime}) \frac{\# \mathcal{A}_N 
\cap \Gamma_\alpha}{\mathcal{S}_{\mathcal{A}_N,\alpha}^2}
\\ 
& \qquad\qquad 
\times\bigg[ \tilde{\mathbb{E}}\avint{\Gamma_\alpha^{m,\varepsilon}} \int_{-\tau}^0  
\bigg[\left\lvert \tilde{X}^{\tilde{r}}_{(s+u)^-}- \tilde{X}^{r^\prime}_{(s+u)^-}\right\rvert^2 
+ \mathbf{1}_{\left\lbrace u<0\right\rbrace}\left\lvert \tilde{X}^{\tilde{r}}_{s+u} 
- \tilde{X}^{r^\prime}_{s+u}\right\rvert^2\bigg]\lambda(\de u)\mathcal{R}(\de r^\prime)
\\ 
& \qquad\qquad\qquad\qquad\qquad 
+ \varepsilon\left( 1 + \mathbb{E}\left\lvert \bar{X}^r_{s^-}\right\rvert^2 
+ \tilde{\mathbb{E}}\int_{-\tau}^0\bigg[\left\lvert\tilde{X}^{\tilde{r}}_{(s+u)^-}\right\rvert^2  
+ \mathbf{1}_{\left\lbrace u<0\right\rbrace}\left\lvert \tilde{X}^{\tilde{r}}_{s+u} 
 \right\rvert^2\bigg] \lambda(\de u)\right)\bigg] \de s \\
& \le 3\varepsilon 
\frac{(\# \mathcal{A}_N\cap\Gamma_\alpha)^2}{\mathcal{S}_{\mathcal{A}_N,\alpha}^2} \int_0^t 
\bar{L}_s ({\omega^\prime}) \left( 1 + 3 C_1(s,{\omega^\prime}) + 
C_2(s,{\omega^\prime})\right)\de s 
\end{align*}
using Lemma \ref{boundedness} and Lemma \ref{continuity w.r.t r}. Finally, 
\begin{align*} 
III 
& \le 3 M_\alpha^{(\varepsilon )} 
\sum_{m=1}^{M_\alpha^{(\varepsilon)}} 
\left(\frac{\# \mathcal{A}_N \cap \Gamma^{m,\varepsilon}_\alpha } 
{\mathcal{S}_{\mathcal{A}_N,\alpha}} - \mathcal{R} \left(\Gamma^{m,\varepsilon}_\alpha 
\right) \right)^2 
\mathbb{E}\int_0^t\left\lvert \tilde{\mathbb{E}}\avint{\Gamma_\alpha^{m, \varepsilon}} 
\Theta\left( s,r,r^\prime,\bar{X}^r_{s^-},\tilde{X}^{r^\prime}_{(s-\tau)^-:s^-},
\omega^\prime\right) \mathcal{R}(\de r^\prime)\right\rvert^2\de s 
\\ 
& \le  3 M_\alpha^{(\varepsilon )} 
\sum_{m=1}^{M_\alpha^{(\varepsilon )}} 
\left(\frac{\# \mathcal{A}_N \cap \Gamma^{m,\varepsilon}_\alpha } 
{\mathcal{S}_{\mathcal{A}_N,\alpha}} - \mathcal{R} \left(\Gamma^{m,\varepsilon}_\alpha 
\right) \right)^2 
\mathbb{E}\int_0^t \bar{K}_s (\omega') \Big( 1
+ \left\lvert \bar{X}^r_{s^-}\right\rvert^2   
\\  
& \qquad\qquad\qquad\qquad 
+ \tilde{\mathbb{E}} \avint{\Gamma^{m,\varepsilon}_\alpha} \int_{-\tau}^0 \left[ \left\lvert 
\tilde{X}^{r^\prime}_{(s+u)^-} \right\rvert^2  + \mathbf{1}_{\left\lbrace u<0\right\rbrace} 
\left\lvert \tilde{X}^{r^\prime}_{s+u}\right\rvert^2\right] \lambda(\de u) 
\mathcal{R}(\de r^\prime )\Big) \de s 
\\
& \le 3 M_\alpha^{(\varepsilon )} \sum_{m=1}^{M_\alpha^{(\varepsilon )}} 
\left(\frac{\# \mathcal{A}_N \cap \Gamma^{m,\varepsilon}_\alpha } 
{\mathcal{S}_{\mathcal{A}_N,\alpha}} - \mathcal{R} \left(\Gamma^{m,\varepsilon}_\alpha 
\right) \right)^2 \int_0^t \bar{K}_s (\omega ') \left( 1 +  3C_1 (s, \omega ')\right)\, ds    
\end{align*} 
using Lemma \ref{boundedness}. Summing up the above estimates we now obtain that 
\begin{align*}
I_\alpha^\Theta 
& \le 6\left(\frac{\#\mathcal{A}_N\cap\Gamma_\alpha}{\mathcal{S}_{\mathcal{A}_N,
\alpha}^2} + \varepsilon \frac{(\#\mathcal{A}_N\cap\Gamma_\alpha )^2}
{\mathcal{S}_{\mathcal{A}_N,\alpha}^2} 
+ M_\alpha^{(\varepsilon)}\sum_{m=1}^{M_\alpha^{(\varepsilon)}} 
\left( 
\frac{\#\left(\mathcal{A}_N\cap\Gamma^{m,\varepsilon}_\alpha\right)}{\mathcal{S}_{\mathcal{A}_N,
\alpha}}-\mathcal{R}\left(\Gamma^{m,\varepsilon}_\alpha\right)\right)^2\right) 
C_2 (t,{\omega^\prime}) \, . 
\end{align*}
Similar arguments imply that
\begin{align*}
I_\alpha^\eta  
& =\int_0^t\int_U\Bigg\lvert \frac{1}{\mathcal{S}_{\mathcal{A}_N,\alpha}} 
\sum_{\tilde{r}\in \mathcal{A}_N\cap \Gamma_\alpha }\eta\left(s,r,\tilde{r},\bar{X}^r_{s^-},
\bar{X}^{\tilde{r}}_{(s-\tau)^-:s^-},{\omega^\prime},\xi\right) 
\\ 
& \qquad\qquad\quad\qquad 
-\tilde{\mathbb{E}}\int_{\Gamma_\alpha}\eta\left(s,r,r^\prime,\bar{X}^r_{s^-},\tilde{X}^{r^\prime}_{(s-\tau)^-:s^-},{\omega^\prime},\xi\right)\mathcal{R} 
(\de r^\prime)\Bigg\rvert^2\nu(\de \xi)\de s
\\ 
& \le 6\left(\frac{\#\mathcal{A}_N\cap\Gamma_\alpha} 
{\mathcal{S}_{\mathcal{A}_N,\alpha}^2}  
+ \varepsilon \frac{\left(\#\mathcal{A}_N\cap\Gamma_\alpha \right)^2 } 
{\mathcal{S}_{\mathcal{A}_N,\alpha}^2} 
+ M_\alpha^{(\varepsilon)} \sum_{m=1}^{M_\alpha^{(\varepsilon)}} 
\left(\frac{\#\mathcal{A}_N\cap\Gamma^{m,\varepsilon}_\alpha}
{\mathcal{S}_{\mathcal{A}_N,\alpha}} -\mathcal{R}\left(\Gamma^{m,\varepsilon}_\alpha \right) 
\right)^2\right) C_2(t,{\omega^\prime})\, . 
\end{align*}
So
\begin{align*}
\mathbb{E} & \left\lvert X^{r,\mathcal{A}_N}_t 
-\bar{X}^r_t\right\rvert^2
\\ 
&\le\int_0^t\left(L_s({\omega^\prime})
+6\bar{L}_s({\omega^\prime})\sum_{\alpha=1}^P 
\frac{\left(\# \mathcal{A}_N\cap\Gamma_\alpha\right)^2} 
{\mathcal{S}_{\mathcal{A}_N,\alpha}^2}+P\right)\max_{r\in \mathcal{A}_N} 
\sup_{u\in [0,s]}\mathbb{E}\left\lvert X^{r,\mathcal{A}_N}_u-\bar{X}^r_u\right\rvert^2\de s 
\\ 
& \quad 
+ 36\sum_{\alpha=1}^P\left(\frac{\#\mathcal{A}_N\cap\Gamma_\alpha} 
{\mathcal{S}_{\mathcal{A}_N,\alpha}^2} + \varepsilon \frac{\left(\#\mathcal{A}_N \cap 
\Gamma_\alpha\right)^2}{\mathcal{S}_{\mathcal{A}_N,\alpha}^2} 
+ M_\alpha^{(\varepsilon)}\sum_{m=1}^{M_\alpha^{(\varepsilon)}} 
\left( \frac{\#\mathcal{A}_N\cap\Gamma^{m,\varepsilon}_\alpha}  
{\mathcal{S}_{\mathcal{A}_N,\alpha}}-\mathcal{R}\left(\Gamma^{m,\varepsilon}_\alpha
\right)\right)^2\right) 
C_2 (t,{\omega^\prime}) \, . 
\end{align*}
Hence Gronwall's lemma implies
\begin{align*}
\sup_{\substack{s\in[0,T]\\r\in \mathcal{A}_N}} 
\mathbb{E} & \left\lvert X^{r,\mathcal{A}_N}_s-\bar{X}^r_s\right\rvert^2 
\\ 
& \le 36\sum_{\alpha=1}^P 
\left(\frac{\#\mathcal{A}_N\cap\Gamma_\alpha}
{\mathcal{S}_{\mathcal{A}_N,\alpha}^2} 
+ \varepsilon \frac{\left(\#\mathcal{A}_N\cap\Gamma_\alpha\right)^2 } 
{\mathcal{S}_{\mathcal{A}_N,\alpha}^2} + M_\alpha^{(\varepsilon)} 
\sum_{m=1}^{M_\alpha^{(\varepsilon)}} 
\left( \frac{\#\mathcal{A}_N\cap\Gamma^{m,\varepsilon}_\alpha} 
{\mathcal{S}_{\mathcal{A}_N,\alpha}}-\mathcal{R}\left(\Gamma^{m,\varepsilon}_\alpha
\right)\right)^2\right) C_2 (T,{\omega^\prime})
\\ 
& \quad 
\times\exp\left[\int_0^T\left(L_s({\omega^\prime})
 + 6\bar{L}_s({\omega^\prime}) 
\sum_{\alpha=1}^P \frac{\left(\#\mathcal{A}_N\cap\Gamma_\alpha\right)^2} 
{\mathcal{S}_{\mathcal{A}_N,\alpha}^2} 
+ P\right)\de s \right] 
\\ 
& \le  144\sum_{\alpha=1}^P\left(\frac{\#\mathcal{A}_N\cap\Gamma_\alpha}
{\mathcal{S}_{\mathcal{A}_N,\alpha}^2} 
+ \varepsilon \frac{\left(\#\mathcal{A}_N\cap \Gamma_\alpha\right)^2}
{\mathcal{S}_{\mathcal{A}_N,\alpha}^2} + M_\alpha^{(\varepsilon)} 
\sum_{m=1}^{M_\alpha^{(\varepsilon)}} 
\left( \frac{\#\mathcal{A}_N\cap\Gamma^{m,\varepsilon}_\alpha} 
{\mathcal{S}_{\mathcal{A}_N,\alpha}}-\mathcal{R}\left(\Gamma^{m,\varepsilon}_\alpha
\right)\right)^2 \right) 
\\ 
& \quad 
\times\exp\left[\int_0^T\left(2L_s({\omega^\prime})
+\bar{L}_s({\omega^\prime}) \left( 6\sum_{\alpha=1}^P\frac{\left(\# 
\mathcal{A}_N\cap\Gamma_\alpha\right)^2} {\mathcal{S}_{\mathcal{A}_N,\alpha}^2} + P 
\right)+K_s({\omega^\prime}) + 3P\bar{K}_s({\omega^\prime}) +3 P\right)\de s\right] \, . 
\end{align*}
Now by integrating with respect to ${\omega^\prime}$, we get for some finite constant $C(T)$ 
that 
\begin{align*}
\mathcal{E} & \left[\sup_{\substack{s\in [0,T]\\r\in \mathcal{A}_N}} 
\mathbb{E}\left\lvert X^{r,\mathcal{A}_N}_s-\bar{X}^r_s\right\rvert^2\right] 
\\ 
& \qquad \le C(T)\sum_{\alpha=1}^P\left(\frac{\#\mathcal{A}_N\cap\Gamma_\alpha}
{\mathcal{S}_{\mathcal{A}_N,\alpha}^2} 
+ \varepsilon \frac{\left(\#\mathcal{A}_N\cap\Gamma_\alpha\right)^2} 
{\mathcal{S}_{\mathcal{A}_N,\alpha}^2} 
+ M_\alpha^{(\varepsilon)}\sum_{m=1}^{M_\alpha^{(\varepsilon)}}
\left(\frac{\#\mathcal{A}_N\cap\Gamma^{m,\varepsilon}_\alpha}{\mathcal{S}_{\mathcal{A}_N,\alpha}}-\mathcal{R}\left(\Gamma^{m,\varepsilon}_\alpha\right)\right)^2\right) .
\end{align*}
Hence
\[
\lim_{N\to\infty}\mathcal{E}\left[\sup_{\substack{s\in [0,T]\\r\in \mathcal{A}_N}}\mathbb{E}\left\lvert X^{r,\mathcal{A}_N}_s-\bar{X}^r_s\right\rvert^2\right]\le PC(T)\varepsilon \, , 
\]
where $\varepsilon$ is arbitrary and therefore
\[ 
\lim_{N\to\infty}\mathcal{E}\left[\sup_{\substack{s\in [0,T] \\ r\in \mathcal{A}_N}} 
\mathbb{E} \left\lvert X^{r,\mathcal{A}_N}_s -\bar{X}^r_s\right\rvert^2\right]=0\, . 
\]
\end{proof}

\begin{appendix} 
\section{Well-posedness for SDEs with path-dependent delay driven by jump diffusions} 
\label{Appendix}  

The purpose of this Appendix is to provide a general existence and uniqueness result on strong 
solutions of stochastic delay differential equations with monotone coefficients driven by jump 
diffusions, that in particular covers the assumptions on the network equations \eqref{equ1}. 
For further reference we formulate our results under more general assumptions on the coefficients. 

\medskip 
Let $(\Omega,\mathcal{F},\mathbb{P})$ be a complete probability space with filtration  
$(\mathcal{F}_t)_{t\ge 0}$ satisfying the usual conditions. Let $(\Omega^\prime, \mathcal{F}^\prime,\mathbb{P}^\prime)$ be a probability space of parameters $\omega^\prime$.  Consider the following stochastic 
delay differential equation:
\begin{equation}
\label{equ1-}
\begin{split}
\de X_t^{\omega^\prime}
& = f(t,\omega,X_{(t-\tau)^-:t^-}^{\omega^\prime},\omega^\prime)\de t+g(t,\omega, X_{(t-\tau)^-:t^-}^{\omega^\prime},\omega^\prime)\de W_t 
+ \int_U h(t,\omega, X_{(t-\tau)^-:t^-}^{\omega^\prime},\omega^\prime, \xi)\tilde{N}(\de t,\de \xi)
\\ 
X_t^{\omega^\prime} & = z_t^{\omega^\prime}, \quad t\in [-\tau,0] \, . 
\end{split}
\end{equation}  
Here, $W$ is a standard Brownian motion in $\mathbb{R}^m$ adapted to the filtration 
$(\mathcal{F}_t)_{t\ge 0}$ such that $(W_t-W_s)_{t\ge s}$ is independent of $\mathcal{F}_s$, 
$s\ge 0$. $N$ is a time homogeneous Poisson measure on $[0,\infty)\times U$ with 
intensity measure 
$\de t \otimes \nu$, where $(U,\mathcal{U}, \nu)$ is an arbitrary $\sigma$-finite measure space. 
$N$ is adapted to the filtration $(\mathcal{F}_t)_{t\ge 0}$, and $N(A)$ is independent of 
$\mathcal{F}_s$, $s\ge 0$, for every measurable subset $A\subseteq (s,\infty)\times U$. Finally, 
denote with $\tilde{N}:=N-\de t \otimes \nu$ the compensated Poisson measure associated with $N$. 
The initial condition $z_{-\tau : 0}^{\omega^\prime}$ belongs to $L^2(\Omega,\mathbb{P};
\text{C\`adl\`ag}\,([-\tau,0];\mathbb{R}^d))$ and is measurable w.r.t $(t,\omega,\omega^\prime)\in [-\tau,0]\times \Omega\times \Omega^\prime$. Recall that we consider the space 
$\text{C\`adl\`ag}\,([-\tau,0];\mathbb{R}^d)$ as well as $\text{C\`agl\`ad} \,([-\tau,0];
\mathbb{R}^d)$, to be endowed with the supremum norm. Finally we assume that $W$, $N$ and $z_{-\tau:0}^{\omega^\prime}$ 
are independent. The coefficients
\[ 
f ,g :\left( [0, \infty )\times \Omega\times \text{C\`agl\`ad}\,([-\tau,0];\mathbb{R}^d)\times \Omega^\prime ,\mathcal{BF}\otimes\mathcal{B}\left(\text{C\`agl\`ad}\,([-\tau,0];\mathbb{R}^d)\right)\otimes \mathcal{F}^\prime\right)\]\[\to \left(\mathbb{R}^d,\mathcal{B}\left(\mathbb{R}^d\right)\right), \left(\mathbb{R}^{d\times m},\mathcal{B}\left(\mathbb{R}^{d\times m}\right)\right)\]
are progressively measurable and
\[h  : \left([0, \infty )\times \Omega\times \text{C\`agl\`ad}\,([-\tau,0];\mathbb{R}^d)\times \Omega^\prime\times U,\mathcal{P}\otimes\mathcal{B}\left(\text{C\`agl\`ad}\,([-\tau,0];\mathbb{R}^d)\right)\otimes \mathcal{F}^\prime\otimes \mathcal{U}\right)\to\left(\mathbb{R}^d,\mathcal{B}\left(\mathbb{R}^d\right)\right) \]
is predictable. Here $\mathcal{BF}$ and $\mathcal{P}$ are the $\sigma$-field of progressively measurable sets and predictable sets on $[0,\infty)\times \Omega$ respectively. 

The following monotonicity and growth conditions are assumed: 

\begin{hyp}
\label{hyp1-} 
There exist a probability measure $\lambda$ on $[-\tau , 0]$ and nonnegative measurable functions $K_t(\omega^\prime) $, $L_t(R,\omega^\prime)$ and $\tilde{K}_t (R,\omega^\prime)$ in 
$L^1_{loc} ([0,\infty[,dt)$, for all $R > 0$ and all $\omega^\prime\in \Omega^\prime$, such that 
the following conditions hold: 
\begin{enumerate}[label=(C\theenumi)]
\item \label{C1} for $\left\lvert x\right\rvert_{L^\infty}$, $\left\lvert y\right\rvert_{L^\infty}\le R$,
\begin{align*}
&  2\left\langle x_0-y_0,f(t,\omega,x_{-\tau:0},\omega^\prime)-f(t,\omega,y_{-\tau:0},\omega^\prime)\right\rangle 
 +\left\lvert g(t,\omega,x_{-\tau:0},\omega^\prime)-g(t,\omega,y_{-\tau:0},\omega^\prime)\right\rvert^2 \\ 
& \qquad\qquad\qquad\qquad 
+\int_{U}\left\lvert h(t,\omega,x_{-\tau:0},\omega^\prime, \xi)-h(t,\omega,y_{-\tau:0},\omega^\prime,
  \xi)\right\rvert^2 \nu ( \de \xi) 
 \\ 
  & \qquad \le L_t(R,\omega^\prime) \int_{-\tau}^0 \left[\left\lvert x_s-y_s\right\rvert^2 
+ \mathbf{1}_{\left\lbrace s<0\right\rbrace}\left\lvert x_{s^+}-y_{s^+}\right\rvert^2\right] 
\lambda(\de s),
\end{align*}
\item \label{C2}
$ 
2\left\langle  x_0,f(t,\omega,x_{-\tau:0},\omega^\prime)\right\rangle  
+ \left\lvert g(t,\omega,x_{-\tau:0},\omega^\prime)\right\rvert^2 
+ \int_{U}\left\lvert h(t,\omega,x_{-\tau:0},\omega^\prime, \xi)\right\rvert^2\nu(\de \xi) 
$
\begin{flalign*} 
& \le K_t(\omega^\prime)(1+\int_{-\tau}^0\left[\left\lvert x_s\right\rvert^2 
+ \mathbf{1}_{\left\lbrace s<0\right\rbrace}\left\lvert x_{s^+}\right\rvert^2\right] 
\lambda(\de s)) 
\end{flalign*}
\item \label{C3} 
$x_{-\tau:0}\mapsto f(t,\omega,x_{-\tau:0},\omega^\prime)$ as a function from $\text{C\`agl\`ad}\,([-\tau,0];
\mathbb{R}^d)$ to $\mathbb{R}^d$ is continuous.
\item \label{C4}
$\sup_{\left\lvert x\right\rvert_{L^\infty}\le R}\left[\left\lvert  
f(t,\omega,x_{-\tau:0},\omega^\prime)\right\rvert  
+ \left\lvert g(t,\omega,x_{-\tau:0},\omega^\prime)\right\rvert^2+\int_U\left\lvert h(t,\omega, 
 x_{-\tau:0},\omega^\prime, \xi)\right\rvert^2\nu(\de \xi)\right] 
\le \tilde{K}_t(R,\omega^\prime)$
\end{enumerate} 
\end{hyp} 

We are going to prove existence and uniqueness of a strong solution using the Euler method. To 
this end let us introduce for $n\in\mathbb{N}$ the Euler approximation  
\begin{equation}
\label{df Xn-}
\begin{split}
X^{n,\omega^\prime}_t
& = X^{n,\omega^\prime}_{\frac{k\tau}{n}}+\int_{\frac{k\tau}{n}}^t f\left(s,\omega,  
  X^{n,\omega^\prime}_{\kappa(n,(s-\tau):s)},\omega^\prime\right)\de s 
+ \int_{\frac{k\tau}{n}}^t g\left(s,\omega, X^{n,\omega^\prime}_{\kappa(n,(s-\tau):s)},\omega^\prime\right)\de W_s 
\\ 
& \quad +\int_{\frac{k\tau}{n}}^t \int_U h\left(s,\omega, X^{n,\omega^\prime}_{\kappa(n,(s-\tau):s)},\omega^\prime,
\xi\right)\tilde{N}(\de s,\de \xi), \quad t\in \left]\frac{k\tau}{n},\frac{(k+1)\tau}{n}\right]
\end{split}
\end{equation}
to the solution of \eqref{equ1-}. Here, $\kappa(n,t):=\frac{k\tau}{n}$ for 
$t\in \left]\frac{k\tau}{n},\frac{(k+1)\tau}{n}\right]$. The process $X^{n,\omega^\prime}$ can be 
constructed inductively as follows: $X^{n,\omega^\prime}_t := z_t^{\omega^\prime}$ for $t\in [-\tau,0]$, and given 
$X^{n,\omega^\prime}_t$ is defined for $t\le \frac{k\tau}{n}$ we can extend $X^{n,\omega^\prime}_t$ for 
$t\in \left]\frac{k\tau}{n},\frac{(k+1)\tau}{n}\right]$ using \eqref{df Xn-}. Note that 
$X^{n,\omega^\prime}$ is c\`adl\`ag, whereas the process $X^{n,\omega^\prime}_{\kappa(n,t)}$, $t\ge -\tau$, 
is c\`agl\`ad. It is easy to see, using induction w.r.t. to $k$, that 
$X^{n,\omega^\prime}_t$, $t\in \left] \frac{k\tau}n , \frac{(k+1)\tau }n\right]$  and is measurable w.r.t. $(t,\omega,\omega^\prime)$ and is a.s. locally bounded 
and that the stochastic integrals are well-defined .  

\begin{thm}
\label{thmA2}
Under Hypothesis \ref{hyp1-}, equation \eqref{equ1-} has a unique strong solution $\left(X_t^{\omega^\prime}\right)_{t\geq 0}$, and for $\mathbb{P}^\prime$-almost all $\omega^\prime\in \Omega^\prime$,
$\left(X^{n,\omega^\prime}_t\right)_{t\geq 0}$ converges to $\left(X_t^{\omega^\prime}\right)_{t\geq 0}$ locally uniformly in probability, i.e. for all $T>0$ and $\mathbb{P}^\prime$-almost all $\omega^\prime\in \Omega^\prime$, 
\[
\lim_{n\to\infty}\mathbb{P}\left\lbrace \sup_{t\in[0,T]}\left\lvert X^{n,\omega^\prime}_t-X_t^{\omega^\prime}\right\rvert 
  > \varepsilon\right\rbrace=0\qquad\forall\,\varepsilon > 0\, . 
\]
and $X$  is measurable w.r.t. $(t,\omega, \omega^\prime)\in [-\tau,\infty[\times \Omega\times \Omega^\prime$ and satisfies 
\begin{equation}\label{moment estimate}
1+2\mathbb{E}\left\lvert X_t^{\omega^\prime}\right\rvert^2 
\le \left( 1+2\sup_{u\in [-\tau,0]}\mathbb{E}\left\lvert z_u^{\omega^\prime}\right\rvert^2\right) 
\cdot\exp\left(\int_0^t 2 K_s(\omega^\prime)\de s\right), \quad t\geq 0.
\end{equation}
\end{thm}

\begin{remark}\rm
To the best of our knowledge, the above theorem is new in this full generality. Although the 
idea of the proof is based on previous works, in particular  \cite{liu2015stochastic}, 
\cite{von2010existence}. The most far reaching existence and uniqueness results for stochastic 
delay differential equations driven by jump diffusions have been obtained in 
\cite{albeverio2010existence}, \cite{song2012numerical}, and  
\cite{wu2013wiener} for locally Lipschitz continuous coefficients and in \cite{xu2015existence} 
under slightly more general assumptions. 
Existence and uniqueness of stochastic delay differential with merely monotone coefficients 
driven by diffusive noise have been obtained in \cite{von2010existence}. 
\end{remark}

\begin{proof} (of Theorem \ref{thmA2}) 
First we prove that every strong solution $X$ to equation \eqref{equ1-} satisfies the moment estimate \eqref{moment estimate}. To this end consider the stopping time $\bar{\sigma}_{R,\omega^\prime}:=\mathbf{1}_{\left\lbrace R>\left\lvert z^{\omega^\prime}\right\rvert_{\infty}\right\rbrace}\cdot\inf\left\lbrace t\geq 0: \left\lvert X_t^{\omega^\prime}\right\rvert>R\right\rbrace$. By It\^o's formula, \ref{C2} and \ref{C4}, we have
\begin{align*}
\mathbb{E}\left\lvert X_{t\wedge \bar{\sigma}_{R,\omega^\prime}}^{\omega^\prime}\right\rvert^2&=\mathbb{E}\left\lvert z_0^{\omega^\prime}\right\rvert^2+\mathbb{E}\int_0^{t\wedge \bar{\sigma}_{R,\omega^\prime}}\Big[ 2 \left\langle X_{s^-}^{\omega^\prime}, f(s,\omega, X_{(s-\tau)^-:s^-}^{\omega^\prime},\omega^\prime)\right\rangle+\left\lvert g(s,\omega, X_{(s-\tau)^-:s^-}^{\omega^\prime},\omega^\prime)\right\rvert^2\\&\qquad\qquad \qquad +\int_U \left\lvert h(s,\omega, X_{(s-\tau)^-:s^-}^{\omega^\prime},\omega^\prime, \xi )\right\rvert^2 \nu(\de \xi)\Big]\de s\\&\leq \mathbb{E}\left\lvert z_0^{\omega^\prime}\right\rvert^2+\mathbb{E}\int_0^{t\wedge \bar{\sigma}_{R,\omega^\prime}}K_s(\omega^\prime)\left(1+\int_{-\tau}^{0}\left[ \left\lvert X_{(s+u)^-}^{\omega^\prime}\right\rvert^2+\mathbf{1}_{\left\lbrace u<0\right\rbrace} \left\lvert X_{s+u}^{\omega^\prime}\right\rvert^2 \right]\lambda(\de u)\right)\de s\\&\leq\mathbb{E}\left\lvert z_0^{\omega^\prime}\right\rvert^2+ \int_{0}^{t}K_s(\omega^\prime)  
\left(1 + 2 \sup_{u\in[-\tau,s]}\mathbb{E}\left\lvert X_{u\wedge \bar{\sigma}_{R,\omega^\prime}}^{\omega^\prime}\right\rvert^2\right) 
\de s\, .
\end{align*}
Therefore by Gronwall's lemma and subsequently Fatou's lemma,  
\[1+2
\mathbb{E}\left\lvert X_{t}^{\omega^\prime}\right\rvert^2 
\leq1+2\liminf_{R\to\infty}\mathbb{E}\left\lvert X_{t\wedge \bar{\sigma}_{R,\omega^\prime}}^{\omega^\prime}\right\rvert^2 
\leq \left( 1+2\sup_{u\in [-\tau,0]}\mathbb{E}\left\lvert z_u^{\omega^\prime}\right\rvert^2\right) 
\cdot\exp\left(\int_0^t 2 K_s(\omega^\prime)\de s\right). 
\]
\paragraph*{Existence:} Let us define the remainder 
\[
p^{n,\omega^\prime}_t=X^{n,\omega^\prime}_{\kappa(n,t)}-X^{n,\omega^\prime}_{t^-}, \quad t\in (-\tau,\infty)\, . 
\]
We can then write
\begin{equation}  
\label{Xn=phi(Xn+pn)-}
\begin{split}
X^{n,\omega^\prime}_t
& = z_0^{\omega^\prime}+\int_0^t f\left(s,\omega, X^{n,\omega^\prime}_{(s-\tau)^-:s^-}
 + p^{n,\omega^\prime}_{(s-\tau):s},\omega^\prime\right)\de s\\&\quad+\int_0^t g\left(s,\omega, X^{n,\omega^\prime}_{(s-\tau)^-:s^-}
 + p^{n,\omega^\prime}_{(s-\tau):s},\omega^\prime\right)\de W_s\\&\quad+\int_0^t \int_U h\left(s,\omega, 
    X^{n,\omega^\prime}_{(s-\tau)^-:s^-}
 + p^{n,\omega^\prime}_{(s-\tau):s},\omega^\prime, \xi\right)\tilde{N}(\de s,\de \xi)\, .
\end{split}
\end{equation}   
In the next step define the stopping times 
\[
\tau^{n,\omega^\prime}_R:=\mathbf{1}_{\left\lbrace 3\left\lvert z^{\omega^\prime}\right\rvert_\infty<R\right\rbrace}\cdot\inf\left\lbrace t\ge 0: \left\lvert X^{n,\omega^\prime}_t\right\rvert > 
\frac{R}{3}\right\rbrace 
\]
for given $R > 0$. Then  
\[
\left\lvert p^{n,\omega^\prime}_{t}\right\rvert \le \frac{2R}{3}, \left
\lvert X^{n,\omega^\prime}_{t^-}\right\rvert \le \frac{R}{3},\quad t\in(0,
\tau^{n,\omega^\prime}_R]. 
\]
For $R > 3|z^{\omega^\prime}|_\infty$ the above inequalities extend to all $t\in (-\tau , \tau_R^{n,\omega^\prime}]$ due to the right continuity of $X_t^{n,\omega^\prime}$. 

For the proof of existence however, we will need a control of $\left\lvert X_{t\wedge \tau_R^{n,\omega^\prime}}^{n,\omega^\prime}\right\rvert$  
in the mean square. To this end note that for $R > 3|z^{\omega^\prime}|_\infty$ the stochastic integrals 
\begin{align*}
M_{t\wedge\tau_R^{n,\omega^\prime}}^{n,\omega^\prime} &:= 2\int_0^{t\wedge\tau_R^{n,\omega^\prime}} \langle X_{s-}^{n,\omega^\prime} , g(s, \omega , 
X_{\kappa (n, (s-\tau):s)}^{n,\omega^\prime},\omega^\prime )dW_s\\&\quad + 2\int_0^{t\wedge\tau_R^{n,\omega^\prime}}  
\int_U \langle X_{s-}^{n,\omega^\prime}, h(s, \omega , X_{\kappa (n, (s-\tau):s)}^{n,\omega^\prime},\omega^\prime, \xi )  
\tilde{N} (ds, d\xi ) \, , t\ge 0 
\end{align*}
are well-defined and square-integrable centered martingales. It follows that  
\[\sup_{R} \tau_R^{n,\omega^\prime} = \lim_{R\to\infty} \tau_R^{n,\omega^\prime} = +\infty,\]
hence the stochastic 
integral $M_t:= M_{t\wedge \tau_R^{n,\omega^\prime}}^{n,\omega^\prime}$, $t \le \tau_R^{n,\omega^\prime}$, is a local martingale up 
to time $+\infty$ with localizing sequence $\sigma_m$, $m\ge 1$, say.  

From now on we will fix $T > 0$ and prove the following properties which complete the 
proof of existence on $[0, T]$, and hence on $t\ge 0$, since $T$ was arbitrary.  
\begin{enumerate}[label=(\roman{*})] 
\item \label{i-} For every $t\ge 0$ and $\omega^\prime\in \Omega^\prime$, $\mathbf{1}_{(-\tau,\tau^{n,\omega^\prime}_R]}(t)p^{n,\omega^\prime}_t\to 0$ in probability as $n\to\infty$.
\item \label{ii-} For any stopping time $\tau^* \le T\wedge\tau^{n,\omega^\prime}_R$, $R$ as in \ref{i-}, we have
$\mathbb{E}\left\lvert X^{n,\omega^\prime}_{\tau^*}\right\rvert^2\le C(T,\omega^\prime)\left(1+I^{n,\omega^\prime}_{T,R}\right)$, for some upper bound $I^{n,\omega^\prime}_{T,R}$ satisfying $\lim_{n\to\infty}I^{n,\omega^\prime}_{T,R}=0$. 
\item \label{iii-} $\lim_{R\to \infty}\limsup_{n\to\infty}\mathbb{P}\left\lbrace \tau^{n,\omega^\prime}_R < T\right\rbrace=0$.
\item \label{iv-} $\forall \varepsilon > 0, \lim_{n,m\to\infty}\mathbb{P}\left\lbrace \sup_{t\in[0,T]}\left\lvert X^{n,\omega^\prime}_t-X^{m,\omega^\prime}_t\right\rvert >\varepsilon\right\rbrace=0$.
\item \label{v-}  $\exists X: \forall \varepsilon>0, \lim_{n\to\infty}\mathbb{P}\left\lbrace \sup_{t\in[0,T]}\left\lvert X^{n,\omega^\prime}_t-X_t^{\omega^\prime}\right\rvert >\varepsilon\right\rbrace=0$ and $X$ is a strong solution of equation \eqref{equ1-} on $[0, T]$.
\end{enumerate}

\paragraph*{Proof of \ref{i-}:} Since $z^{\omega^\prime}$ is c\`adl\`ag w.r.t. time, $\mathbf{1}_{(-\tau,0]}(t)p^{n,\omega^\prime}_t\to 0$ almost surely. Using \eqref{df Xn-} and Hypothesis \ref{hyp1-}, we have for every $t>0$
\begin{align*}
\MoveEqLeft[1]\mathbb{P}\left\lbrace \left\lvert p^{n,\omega^\prime}_t\right\rvert\ge \varepsilon,0< t\le \tau^{n,\omega^\prime}_R\right\rbrace\\ 
& \le \mathbb{P}\left\lbrace \int_{\kappa(n,t)}^t \sup_{\left\lvert x\right\rvert_{L^\infty}\le R}\left\lvert f(s,\omega, x_{-\tau:0},\omega^\prime)\right\rvert \de s\ge \varepsilon/3\right\rbrace \\ 
& \quad +\mathbb{P}\left\lbrace \left\lvert\int_{\kappa(n,t)}^t g\left(s,\omega, X^{n,\omega^\prime}_{\kappa(n,(s-\tau):s)},\omega^\prime\right)\de W_s\right\rvert\ge \varepsilon/3 , t\le\tau^{n,\omega^\prime}_R
\right\rbrace \\ 
& \quad +\mathbb{P}\left\lbrace \left\lvert\int_{\kappa(n,t)}^{t^-}\int_U h\left(s,\omega, X^{n,\omega^\prime}_{\kappa(n,(s-\tau):s)},\omega^\prime, \xi\right)\tilde{N}(\de s,\de \xi)\right\rvert\ge \varepsilon/3, t\le\tau^{n,\omega^\prime}_R\right\rbrace \\ 
& \le \mathbb{P}\left\lbrace\int_{\kappa(n,t)}^t \tilde{K}_s(R,\omega^\prime)\de s 
\ge \varepsilon/3\right\rbrace \\ 
& \quad +\frac{9}{\varepsilon^2} \left( \mathbb{E}\left\lvert\int_{\kappa(n,t)}^t 
\mathbf{1}_{\left\lbrace s\le\tau^{n,\omega^\prime}_R\right\rbrace} g \left(s,\omega, X^{n,\omega^\prime}_{\kappa(n,(s-\tau):s)},\omega^\prime\right)\de W_s\right\rvert^2  
\right) \\ 
& \quad +\frac{9}{\varepsilon^2}\mathbb{E} \left( \left\lvert \int_{\kappa(n,t)}^{t^-} 
\int_U\mathbf{1}_{\left\lbrace s\le\tau^{n,\omega^\prime}_R 
\right\rbrace} h\left(s,\omega, X^{n,\omega^\prime}_{\kappa(n,(s-\tau):s)},\omega^\prime, \xi\right) 
\tilde{N}(\de s,\de \xi)\right\rvert^2  \right) \\ 
& \le \frac{3}{\varepsilon}\int_{\kappa(n,t)}^t \tilde{K}_s(R,\omega^\prime)\de s 
+\frac{9}{\varepsilon^2}\mathbb{E} \left( \int_{\kappa(n,t)\wedge  
\tau^{n,\omega^\prime}_R}^{t\wedge \tau^{n,\omega^\prime}_R}  \tilde{K}_s(R,\omega^\prime)\de s \right) 
\le \left(\frac{3}{\varepsilon} + \frac{9}{\varepsilon^2} \right) 
\int_{\kappa(n,t)}^t \tilde{K}_s (R,\omega^\prime)\, \de s\, , 
\end{align*}
so
\[
\limsup_{n\to \infty} 
\mathbb{P}\left\lbrace \left\lvert p^{n,\omega^\prime}_t\right\rvert\ge \varepsilon, -\tau<t\le 
\tau^{n,\omega^\prime}_R\right\rbrace = 0
\]
which implies \ref{i-}.
\paragraph*{Proof of \ref{ii-}:} Using It\^o's formula equation \eqref{Xn=phi(Xn+pn)-} implies that 
\begin{equation*} 
\begin{aligned} 
\left\lvert X_t^{n,\omega^\prime}\right\rvert^2 & = \left\lvert z_0^{\omega^\prime}\right\rvert^2 
+ \int_0^t 
\bigg[ 2\left\langle X^{n,\omega^\prime}_{s^-},f\left(s, \omega, 
X^{n,\omega^\prime}_{\kappa(n, (s-\tau):s)},\omega^\prime\right) \right\rangle 
+ \left\lvert g\left(s,\omega,X^{n,\omega^\prime}_{\kappa (n,(s-\tau):s)},\omega^\prime\right)\right\rvert^2\bigg]\de s \\ 
& \quad
+ \int_0^t\int_U \left\lvert h\left(s,\omega,X^{n,\omega^\prime}_{\kappa (n,(s-\tau):s)},\omega^\prime, \xi 
  \right)\right\rvert^2 N(\de s,\de \xi)
+ M_t \, . 
\end{aligned}
\end{equation*}
For any stopping time $\tau^*\le t\wedge\tau_R^{n,\omega^\prime}$ we then have 
\begin{equation*} 
\begin{aligned}
\MoveEqLeft[1]\mathbb{E}\left\lvert X^{n,\omega^\prime}_{\tau^*\wedge\sigma_m}\right\rvert^2\\
& \le \mathbb{E}\left\lvert z_0^{\omega^\prime} \right\rvert^2+\mathbb{E}\int_0^{\tau^*\wedge\sigma_m} 
\bigg[2\left\langle X^{n,\omega^\prime}_{s-},f\left(s,\omega,X^{n,\omega^\prime}_{\kappa (n,(s-\tau):s)},\omega^\prime
\right)\right\rangle 
+ \left\lvert g\left(s,\omega,X^{n,\omega^\prime}_{\kappa (n,(s-\tau):s)},\omega^\prime\right)\right\rvert^2\\ 
& \qquad\qquad\qquad\qquad\qquad 
+\int_U \left\lvert h\left(s,\omega,X^{n,\omega^\prime}_{\kappa (n,(s-\tau):s)},\omega^\prime,
\xi\right)\right\rvert^2 \nu(\de \xi)\bigg]\de s \\ 
& \le \mathbb{E}\left\lvert z_0^{\omega^\prime} \right\rvert^2 
+ \mathbb{E}\int_0^{\tau^*\wedge\sigma_m} 2\left\langle p^{n,\omega^\prime}_{s-},f\left(s, \omega,  
X^{n,\omega^\prime}_{\kappa (n,(s-\tau):s)},\omega^\prime \right) \right\rangle\de s 
\\ 
& \quad + \mathbb{E}\int_0^{\tau^*\wedge\sigma_m} K_s(\omega^\prime)\left(1  + 
\int_{-\tau}^0 \left(\left\lvert X^{n,\omega^\prime}_{(s+u)^-}+p^{n,\omega^\prime}_{s+u}\right\rvert^2 
+ \mathbf{1}_{\left\lbrace u<0\right\rbrace}\left\lvert X^{n,\omega^\prime}_{s+u}  
+ p^{n,\omega^\prime}_{(s+u)^+}\right\rvert^2\right) \lambda(\de u)\right)\de s
\\
& \le  C(T,\omega^\prime)+\underbrace{\mathbb{E}\int_0^T \left[2\tilde{K}_s(R,\omega^\prime) 
\left\lvert p^{n,\omega^\prime}_s\right\rvert 
+ 2K_s(\omega^\prime) \int_{-\tau}^0\left(\left\lvert p^{n,\omega^\prime}_{s+u}\right\rvert^2 
+ \mathbf{1}_{\left\lbrace u<0\right\rbrace}\left\lvert 
p^{n,\omega^\prime}_{(s+u)^+}\right\rvert^2\right)\lambda(\de u)\right]\mathbf{1}_{\left\lbrace s\le  
T\wedge \tau^{n,\omega^\prime}_R\right\rbrace}\de s}_{ =: I^{n,\omega^\prime}_{T,R}}  \\ 
& \quad + \int_0^t 4K_s(\omega^\prime)\sup_{u\in[0,s]}\mathbb{E}\left\lvert X^{n,\omega^\prime}_{u\wedge\tau^*}
\right\rvert^2\de s  \, . 
\end{aligned}
\end{equation*} 
Taking the limit $m\to\infty$ yields by Fatou's lemma that 
\begin{equation} 
\label{theA2:eq1} 
\begin{aligned}
\mathbb{E}\left\lvert X^{n,\omega^\prime}_{\tau^*}\right\rvert^2
& \le  C(T,\omega^\prime) + I_{T,R}^{n,\omega^\prime} + \int_0^t 4K_s(\omega^\prime)\sup_{u\in[0,s]}\mathbb{E}\left\lvert  
   X^{n,\omega^\prime}_{u\wedge\tau^*}\right\rvert^2\de s 
\end{aligned}
\end{equation} 
with 
\[
\lim_{n\to\infty}I^{n,\omega^\prime}_{T,R}=0\, . 
\]
Setting $\tau^* := t\wedge\tau_R^{n,\omega^\prime}$, $t\in [0, T]$, and using Gronwall's inequality, we 
obtain that 
\[ 
\sup_{t\in[0,T]}\mathbb{E}\left\lvert X^{n,\omega^\prime}_{t\wedge \tau_R^{n,\omega^\prime}}\right\rvert^2 
\le C(T,\omega^\prime)\left(1+I^{n,\omega^\prime}_{T,R}\right)
\] 
and for any stopping time $\tau^*\le T\wedge\tau_R^{n,\omega^\prime}$ \eqref{theA2:eq1} now implies that  
$$ 
\mathbb{E}  |X^{n,\omega^\prime}_{\tau^*}|^2 \le C(T,\omega^\prime) \left( 1 + I_{T,R}^{n,\omega^\prime}\right)\, . 
$$ 
Note that $C(T,\omega^\prime)$ may differ from line to line, but always can be chosen to be an increasing 
function. 
 
\paragraph*{Proof of \ref{iii-}:} 
Let 
 \[
 \tau^*=T\wedge\tau^{n,\omega^\prime}_R\wedge \inf\left\lbrace t\ge 0:\left\lvert X^{n,\omega^\prime}_t\right\rvert 
 \ge a\right\rbrace\, . 
 \]
\ref{ii-} then implies that 
\[
\mathbb{P}\left\lbrace \sup_{t\in[0,T\wedge\tau^{n,\omega^\prime}_R]}\left\lvert X^{n,\omega^\prime}_t\right\rvert 
\ge a\right\rbrace\le \frac{1}{a^2}\mathbb{E}\left\lvert X^{n,\omega^\prime}_{\tau^*} 
\right\rvert^2\le\frac{C(T,\omega^\prime)}{a^2}\left(1+I^{n,\omega^\prime}_{T,R}\right) 
\]
for any $a > 0$. In particular, 
\begin{align*}
\limsup_{R\to\infty}\limsup_{n\to \infty}\,  
& \mathbb{P}\left\lbrace 
\sup_{t\in[0,\tau^{n,\omega^\prime}_R]}\left\lvert X^{n,\omega^\prime}_t\right\rvert\ge \frac{R}{4};\tau^{n,\omega^\prime}_R 
  < T \right\rbrace 
\le \limsup_{R\to\infty}\limsup_{n\to \infty}\mathbb{P}\left\lbrace 
\sup_{t\in[0,T\wedge\tau^{n,\omega^\prime}_R]}\left\lvert X^{n,\omega^\prime}_t\right\rvert \ge \frac{R}{4}
\right\rbrace \\ 
& \le\limsup_{R\to\infty}\limsup_{n\to\infty} 
\frac{16C(T,\omega^\prime)\left(1+ I^{n,\omega^\prime}_{T,R}\right)}{R^2} = 0.
\end{align*} 
It follows that  
 \begin{align*} 
\limsup_{R\to\infty}\limsup_{n\to \infty}\mathbb{P}\left\lbrace \tau^{n,\omega^\prime}_R 
< T\right\rbrace 
\le\limsup_{R\to\infty}\limsup_{n\to \infty}\left[\mathbb{P}\left\lbrace  
\sup_{t\in[0,\tau^{n,\omega^\prime}_R]}\left\lvert X^{n,\omega^\prime}_t\right\rvert\ge \frac{R}{4};\tau^{n,\omega^\prime}_R 
< T \right\rbrace+\mathbb{P}\left\lbrace 3\left\lvert z\right\rvert_\infty>R\right\rbrace\right]=0
 \end{align*}
which completes the proof of \ref{iii-}.  

\paragraph*{Proof of \ref{iv-}:} Let $\tau^{n,m,\omega^\prime}_R:=T\wedge 
\tau^{n,\omega^\prime}_R\wedge \tau^{m,\omega^\prime}_R$. Using It\^o's formula, we have for any stopping time 
$\bar{\tau} \le t\wedge \tau^{n,m,\omega^\prime}_R$ that
\begin{align*}
\MoveEqLeft[1]
\mathbb{E}\left\lvert X^{n,\omega^\prime}_{\bar{\tau}}- X^{m,\omega^\prime}_{\bar{\tau}}\right\rvert^2 \\ 
& \le  
\mathbb{E}\int_0^{\bar{\tau}} 2\left\langle X^{n,\omega^\prime}_{s-}-X^{m,\omega^\prime}_{s-}, 
f\left(s,\omega, X^{n,\omega^\prime}_{\kappa (n,(s-\tau):s)},\omega^\prime\right)-f\left(s,\omega, 
X^{m,\omega^\prime}_{\kappa(n,(s-\tau):s)},\omega^\prime\right)\right\rangle\de s \\ 
& \quad + \mathbb{E}\int_0^{\bar{\tau}} \left\lvert g\left(s,\omega, 
X^{m,\omega^\prime}_{\kappa(n,(s-\tau):s)},\omega^\prime\right) - g\left(s,\omega, 
X^{m,\omega^\prime}_{\kappa(n,(s-\tau):s)},\omega^\prime\right)\right\rvert^2\de s \\ 
& \quad + \mathbb{E}\int_0^{\bar{\tau}}\int_U \left\lvert h\left(s,\omega, 
X^{m,\omega^\prime}_{\kappa(n,(s-\tau):s)},\omega^\prime, \xi\right)-h\left(s,\omega, 
X^{m,\omega^\prime}_{\kappa(n,(s-\tau):s)},\omega^\prime, \xi\right)\right\rvert^2\nu(\de \xi)\de s
\end{align*}
Hypothesis \ref{hyp1-} implies
\begin{equation} 
\label{ineq1-}
\begin{aligned} 
\mathbb{E}\Big\lvert X^{n,\omega^\prime}_{\bar{\tau}} & - X^{m,\omega^\prime}_{\bar{\tau}}\Big\rvert^2 \\ 
& \le \mathbb{E}\int_0^{\bar{\tau}} 2\left\langle p^{n,\omega^\prime}_{s-}-p^{m,\omega^\prime}_{s-}, 
f\left(s,\omega, X^{n,\omega^\prime}_{\kappa(n,(s-\tau):s)},\omega^\prime\right) 
-f \left(s, \omega,X^{m,\omega^\prime}_{\kappa (n,(s-\tau):s)},\omega^\prime\right)\right\rangle\de s \\ 
& \quad + \mathbb{E}\int_0^{\bar{\tau}} L_s (R,\omega^\prime) \int_{-\tau}^0 \left(\left\lvert 
X^{n,\omega^\prime}_{\kappa (n,s+u)} - X^{m,\omega^\prime}_{\kappa (n,s+u)} \right\rvert^2 
+\mathbf{1}_{\left \lbrace u<0\right\rbrace}\left\lvert X^{n,\omega^\prime}_{\kappa (n,(s+u)^+)}  
- X^{m,\omega^\prime}_{\kappa (n,(s+u)^+)}\right\rvert^2\right)\lambda(\de u)\de s \\
& \le \mathbb{E}\Bigg( \int_0^T \mathbf{1}_{[0,\tau^{n,m,\omega^\prime}_R]}(s) 
\Bigg[ 4\tilde{K}_s(R,\omega^\prime)\left(\left\lvert p^{n,\omega^\prime}_s\right\rvert+\left\lvert 
p^{m,\omega^\prime}_s\right\rvert\right) \\ 
& \qquad\qquad\qquad\qquad + 3L_s(R,\omega^\prime)\int_{-\tau}^0\left(\left\lvert 
p^{n,\omega^\prime}_{s+u}\right\rvert^2+\left\lvert p^{m,\omega^\prime}_{s+u}\right\rvert^2+\mathbf{1}_{\left 
\lbrace u<0\right\rbrace}\left(\left\lvert p^{n,\omega^\prime}_{(s+u)^+}\right\rvert^2+\left\lvert 
p^{m,\omega^\prime}_{(s+u)^+}\right\rvert^2\right)\right)\lambda(\de u)\Bigg]\de s \Bigg) \\ 
&
\quad+\int_0^t 6L_s (R,\omega^\prime) \sup_{u\in[0,s]}\mathbb{E}\left\lvert X^{n,\omega^\prime}_{u\wedge\tau^{n,m,\omega^\prime}_R} 
- X^{m,\omega^\prime}_{u\wedge\tau^{n,m,\omega^\prime}_R}\right\rvert^2\de s \\ 
& = I^{n,m,\omega^\prime}_{R,T}+\int_0^t6L_s(R,\omega^\prime) \sup_{u\in[0,s]}\mathbb{E}\left\lvert 
X^{n,\omega^\prime}_{u\wedge\tau^{n,m,\omega^\prime}_R}- X^{m,\omega^\prime}_{u\wedge\tau^{n,m,\omega^\prime}_R}\right\rvert^2\de s
\end{aligned}
\end{equation} 
By setting $\bar{\tau}:= t\wedge\tau^{n,m,\omega^\prime}_R$, $t\in [0,T]$, and using Gronwall's inequality, 
we get
\begin{equation}
\sup_{t\in[0,T]}\mathbb{E}\left\lvert X^{n,\omega^\prime}_{t\wedge\tau^{n,m,\omega^\prime}_R} 
- X^{m,\omega^\prime}_{t\wedge\tau^{n,m,\omega^\prime}_R}\right\rvert^2\le \exp\left(\int_0^T 6L_s(R,\omega^\prime)\de s 
\right)I^{n,m,\omega^\prime}_{R,T}.
\end{equation}
 Substituting this bound into the right hand side of \eqref{ineq1-} implies 
\[\mathbb{E}\left\lvert X^{n,\omega^\prime}_{\bar{\tau}}- X^{m,\omega^\prime}_{\bar{\tau}}\right\rvert^2 
\le  C(T,R,\omega^\prime)I^{n,m,\omega^\prime}_{R,T}\]
By setting 
\[ 
\bar{\tau}:=\tau^{n,m,\omega^\prime}_R\wedge \inf\left\lbrace t\ge 0:\left\lvert X^{n,\omega^\prime}_t 
-X^{m,\omega^\prime}_t\right\rvert\ge \varepsilon\right\rbrace, 
\]
we have
\begin{align*}
\mathbb{P} & \left\lbrace \sup_{t\in[0,T]}\left\lvert X^{n,\omega^\prime}_t-X^{m,\omega^\prime}_t\right\rvert\ge 
\varepsilon\right\rbrace \\ 
& \qquad \le \mathbb{P}\left\lbrace T>\tau^{n,\omega^\prime}_R\right\rbrace 
+\mathbb{P}\left\lbrace T>\tau^{m,\omega^\prime}_R\right\rbrace +
\mathbb{P}\left\lbrace \sup_{t\in[0,\tau^{n,m,\omega^\prime}_R]}\left\lvert X^{n,\omega^\prime}_t-X^{m,\omega^\prime}_t\right\rvert\ge 
\varepsilon\right\rbrace \\ 
& \qquad \le\mathbb{P}\left\lbrace T>\tau^{n,\omega^\prime}_R\right\rbrace 
+ \mathbb{P}\left\lbrace T>\tau^{m,\omega^\prime}_R\right\rbrace 
+ \frac{1}{\varepsilon^2}\mathbb{E}\left\lvert X^{n,\omega^\prime}_{\bar{\tau}}- 
X^{m,\omega^\prime}_{\bar{\tau}}\right\rvert^2 \\ 
& \qquad \le \mathbb{P}\left\lbrace T>\tau^{n,\omega^\prime}_R\right\rbrace 
+ \mathbb{P}\left\lbrace T>\tau^{m,\omega^\prime}_R\right\rbrace 
+\frac{C(T,R,\omega^\prime)}{\varepsilon^2} I^{n,m,\omega^\prime}_{R,T}\, . 
\end{align*}
\ref{i-} and dominated convergence now implies that 
\[
\limsup_{n,m\to \infty}I^{n,m,\omega^\prime}_{R,T}=0 
\]
and using \ref{iii-}, we get 
\begin{align*}
\limsup_{n,m\to\infty}\mathbb{P} 
& \left\lbrace \sup_{t\in[0,T]}\left\lvert X^{n,\omega^\prime}_t-X^{m,\omega^\prime}_t\right\rvert\ge 
\varepsilon\right\rbrace \\ 
& 
\le  \lim_{R\to\infty}\limsup_{n,m\to\infty} 
\left[\mathbb{P}\left\lbrace T>\tau^{n,\omega^\prime}_R\right\rbrace 
+ \mathbb{P}\left\lbrace T > \tau^{m,\omega^\prime}_R\right\rbrace+\frac{C(T,R, \omega^\prime)}{\varepsilon^2} 
I^{n,m }_{R,T}\right]=0\, .
\end{align*}
So \ref{iv-} is obtained. 

\paragraph*{Proof of \ref{v-}:} Since the space $L^2\left(\Omega, 
\text{C\`adl\`ag}\,([-\tau,T],\mathbb{R}^d)\right)$ is complete w.r.t. convergence in 
probability, \ref{iv-} yields that there exists $X^{\omega^\prime}\in L^2\left(\Omega, \text{C\`adl\`ag}\,
([-\tau,T],\mathbb{R}^d)\right)$ such that
\[ 
\lim_{n\to\infty}\mathbb{P}\left\lbrace \sup_{t\in[0,T]}\left\lvert X^{n,\omega^\prime}_t 
- X_t^{\omega^\prime}\right\rvert\ge \varepsilon\right\rbrace=0. 
\]
We have to show that all terms of equation \eqref{Xn=phi(Xn+pn)-} for a subsequence of 
$n\in\mathbb{N}$ converge almost surely to the terms of equation \eqref{equ1-}. We have
\begin{align*}
\lim_{n\to\infty}\mathbb{P} & \left\lbrace \sup_{t\in[0,T]}\left\lvert X^{n,\omega^\prime}_{\kappa(n,t)} 
  -X_{t^-}^{\omega^\prime}\right\rvert\ge \varepsilon\right\rbrace \\
& \qquad \le \lim_{n\to\infty}\mathbb{P}\left\lbrace \sup_{t\in[0,T]}\left\lvert 
X^{n,\omega^\prime}_{\kappa(n,t)}-X_{\kappa(n,t)}^{\omega^\prime}\right\rvert \ge \varepsilon/2\right\rbrace  
 +\lim_{n\to\infty}\mathbb{P}\left\lbrace \sup_{t\in[0,T]}\left\lvert 
X_{\kappa(n,t)}^{\omega^\prime}-X_{t^-}^{\omega^\prime}\right\rvert\ge \varepsilon/2\right\rbrace=0\, . 
\end{align*} 
We can find a subsequence, say 
$\left\lbrace n_l\right\rbrace_{l\in\mathbb{N}}$, such that as $l\to\infty$,
\[
\sup_{t\in [0,T]}\left\lvert X^{n_l,\omega^\prime}_{\kappa(n_l,t)}-X_{t^-}^{\omega^\prime}\right\rvert\to 0 
\quad \mathbb{P}-a.s.
\]
for all $\omega^\prime$ in a subset $\Omega^\prime_0$ of full $\mathbb{P}^\prime$-measure. Now let us define
\[
S(t,\omega^\prime):=\sup_{l\in\mathbb{N}} \left\lvert X^{n_l,\omega^\prime}_{\kappa(n_l,t)}\right\rvert \, , 
\]
then
\[
\sup_{t\in [0,T]}S(t,\omega^\prime)<\infty\quad \mathbb{P}-a.s. 
\]
for all $\omega^\prime$ in $\Omega^\prime_0$. So by using \ref{C3}, \ref{C4} and dominated convergence, we obtain for all $\omega^\prime$ in $\Omega^\prime_0$ that
\[
\lim_{l\to\infty}\int_0^t f\left( s,\omega,  X^{n_l,\omega^\prime}_{\kappa(n_l,(s-\tau):s)},\omega^\prime\right)\de s 
= \int_0^t f\left(s,\omega,X_{(s-\tau)^-:s^-}^{\omega^\prime},\omega^\prime\right)\de s\quad \mathbb{P}-a.s. 
\]
Let $\tau (R):=\inf\left\lbrace t\ge 0:S(t)>R\right\rbrace\wedge T$. For all $t\in [0,T]$ and all $\omega^\prime\in\Omega^\prime_0$,  
we have by dominated convergence that
\begin{align*}
\MoveEqLeft[3]\lim_{l\to\infty}\mathbb{E}\left\lvert \int_0^{t\wedge\tau (R)}\left[g\left(s,
\omega,X^{n_l,\omega^\prime}_{\kappa(n_l,(s-\tau):s)},\omega^\prime\right)-g\left( s,\omega, 
X_{(s-\tau)^-:s^-}^{\omega^\prime},\omega^\prime\right)\right]\de W_s\right\rvert^2 \\ 
& =\lim_{l\to\infty}\mathbb{E}\int_0^t \mathbf{1}_{\left\lbrace s\le \tau (R) 
\right\rbrace}\left\lvert g\left(s,\omega,X^{n_l,\omega^\prime}_{\kappa(n_l,(s-\tau):s)},\omega^\prime\right) 
-g\left( s,\omega,X_{(s-\tau)^-:s^-}^{\omega^\prime},\omega^\prime\right)\right\rvert^2\de s=0
\end{align*}
So
\begin{align*}
\MoveEqLeft[3]\mathbb{P}\left\lbrace \left\lvert \int_0^t\left[g\left(s,\omega, 
X^{n_l,\omega^\prime}_{\kappa(n_l,(s-\tau):s)},\omega^\prime\right)-g\left( s,\omega,X_{(s-\tau)^-:s^-}^{\omega^\prime},\omega^\prime\right)\right] 
\de W_s\right\rvert>\varepsilon\right\rbrace \\ 
& \le \mathbb{P}\left\lbrace \left\lvert \int_0^{t\wedge\tau (R)}\left[g\left(s,\omega, 
X^{n_l,\omega^\prime}_{\kappa(n_l,(s-\tau):s)},\omega^\prime\right)-g\left( s,\omega,X_{(s-\tau)^-:s^-}^{\omega^\prime},\omega^\prime\right)\right] 
\de W_s\right\rvert>\varepsilon\right\rbrace\\&\quad+\mathbb{P}\left\lbrace t>\tau (R) 
\right\rbrace.
\end{align*}
Fix sufficiently large $R$ such that the second term on the right hand side is less than 
$\delta>0$, then taking the limit $l\to\infty$ implies
\[ 
\lim_{l\to \infty}\mathbb{P} 
\left\lbrace \left\lvert \int_0^t\left[g\left(s,\omega,X^{n_l,\omega^\prime}_{\kappa(n_l,
(s-\tau):s)},\omega^\prime\right)-g\left( s,\omega,X_{(s-\tau)^-:s^-}^{\omega^\prime},\omega^\prime\right)\right] 
\de W_s\right\rvert > 
\varepsilon\right\rbrace\le \delta\]
where  $\delta>0$ is arbitrary. Therefore 
\[ 
\int_0^tg\left(s,\omega,X^{n_l,\omega^\prime}_{\kappa(n_l,(s-\tau):s)},\omega^\prime\right)\de W_s\to \int_0^t 
g \left(s,\omega,X_{(s-\tau)^-:s^-}^{\omega^\prime},\omega^\prime\right)\de W_s\quad \text {in probability} 
\]
The same argument implies
\[ 
\int_0^t\int_U h\left(s,\omega,X^{n_l,\omega^\prime}_{\kappa(n_l,(s-\tau):s)},\omega^\prime, \xi\right) 
\tilde{N}(\de s,\de \xi)\to \int_0^t\int_U h\left(s,\omega,X_{(s-\tau)^-:s^-}^{\omega^\prime},\omega^\prime, \xi\right) 
\tilde{N}(\de s,\de \xi)\quad \text {in probability} 
\]
and for some subsequence $n_{l_k}$ the above convergences are $\mathbb{P}-a.s$. 
Therefore $X$ is a solution of equation \eqref{equ1-} on $[0, T]$. Since $X^{n,\omega^\prime}_t$ is measurable w.r.t. $(t,\omega,\omega^\prime)\in [-\tau,\infty[\times \Omega\times \Omega^\prime$, $X_t^{\omega^\prime}$ is also measurable.

\paragraph*{Uniqueness:} Let $X$ and $Y$ be two solutions of equation \eqref{equ1-} and define 
\[
\tau(R,\omega^\prime):=\inf\left\lbrace t\ge 0; \left\lvert X_t^{\omega^\prime}\right\rvert>R \text{  or  } \left\lvert 
Y_t^{\omega^\prime}\right\rvert>R\right\rbrace \, . 
\]
We have
\begin{align*}
\mathbb{E}\left\lvert X_{t\wedge \tau(R,\omega^\prime)}^{\omega^\prime}-Y_{t\wedge \tau(R,\omega^\prime)}^{\omega^\prime}\right\rvert^2 
& = \mathbb{E}\int_0^{t\wedge \tau(R,\omega^\prime)} \big[2\left\langle X_{s^-}^{\omega^\prime}-Y_{s^-}^{\omega^\prime}, 
f\left( s, \omega,X_{(s-\tau)^-:s^-}^{\omega^\prime},\omega^\prime\right)-f\left(s, \omega , Y_{(s-\tau)^-:s^-}^{\omega^\prime},\omega^\prime
   \right)\right\rangle  \\ 
& \qquad\qquad\qquad+\left\lvert g\left( s,\omega,X_{(s-\tau)^-:s^-}^{\omega^\prime},\omega^\prime\right)  
   - g\left(s,\omega,Y_{(s-\tau)^-:s^-}^{\omega^\prime},\omega^\prime\right)\right\rvert^2\big]\de s \\ 
& \quad +\mathbb{E}\int_0^{t\wedge \tau(R,\omega^\prime)}\int_U \left\lvert h\left( s,\omega , 
    X_{(s-\tau)^-:s^-}^{\omega^\prime},\omega^\prime, \xi\right)-h\left(s,\omega,Y_{(s-\tau)^-:s^-}^{\omega^\prime},\omega^\prime, \xi\right)\right   
 \rvert^2\nu(\de \xi)\de s \\ 
& \le \mathbb{E}\int_0^{t\wedge \tau(R,\omega^\prime)}L_s(R,\omega^\prime)\int_{-\tau}^0\left(  \left  
 \lvert X_{(s+u)^-}^{\omega^\prime}-Y_{(s+u)^-}^{\omega^\prime}\right\rvert^2+\mathbf{1}_{\left\lbrace u<0\right\rbrace}  
 \left\lvert X_{s+u}^{\omega^\prime}-Y_{s+u}^{\omega^\prime}\right\rvert^2\right)\lambda(\de u)\de s
\end{align*} 
so that 
\begin{align*}
\sup_{s\le t}\mathbb{E}\left\lvert X_{s\wedge \tau(R)}^{\omega^\prime}-Y_{s\wedge \tau(R)}^{\omega^\prime}\right\rvert^2 
& \le \int_0^t2L_s(R,\omega^\prime)\sup_{u\le s}\mathbb{E}\left\lvert X_{u\wedge \tau(R)}^{\omega^\prime} 
 -Y_{u\wedge \tau(R)}^{\omega^\prime}\right\rvert^2\de s\, . 
\end{align*}
Gronwall's lemma now implies that 
\[
\sup_{s\le T}\mathbb{E}\left\lvert X_{s\wedge \tau(R)}^{\omega^\prime}-Y_{s\wedge \tau(R)}^{\omega^\prime}\right\rvert^2=0, 
\] 
so
\[
\mathbb{P}\left\lbrace X_s^{\omega^\prime}=Y_s^{\omega^\prime}\right\rbrace = \lim_{R\to\infty} 
\left[\mathbb{P}\left\lbrace X_s^{\omega^\prime} =Y_s^{\omega^\prime},  
s\le \tau(R)\right\rbrace+\mathbb{P}\left\lbrace s>\tau(R)\right\rbrace\right]=1
\]
and the uniqueness is proved.
\end{proof}
\end{appendix}

\subsection*{Acknowledgements}

We would like to thank the referee as well as the associate editor for several constructive comments and 
suggestions. We are in particular thankful for pointing out an unjustified application of the moment 
estimate (Lemma \ref{boundedness}) in the first version of the proof of Theorem \ref{existunique}, that 
could be corrected by changing the construction of a solution of \eqref{equ2} to a semi-implicit Euler 
scheme.  
 
\bibliographystyle{plain}
%\nocite{prevot2007concise}
%\nocite{gyongy2013note}
%\nocite{kumar2014strong}

\bibliography{biblio}

\begin{thebibliography}{10}

\bibitem{albeverio2010existence}
{Albeverio, Sergio and Brze{\'z}niak, Zdzis{\l}aw and Wu, Jiang-Lun}.
\newblock {Existence of global solutions and invariant measures for stochastic
  differential equations driven by Poisson type noise with non-Lipschitz
  coefficients}.
\newblock {\em {Journal of Mathematical Analysis and Applications}},
  {371}({1}):{309--322}, {2010}.

\bibitem{andreis2018mckean}
{Andreis, Luisa and Dai Pra, Paolo and Fischer, Markus}.
\newblock {McKean--Vlasov limit for interacting systems with simultaneous
  jumps}.
\newblock {\em {Stochastic Analysis and Applications}}, pages {1--36}, {2018}.

\bibitem{baladron2012mean}
{Baladron, Javier and Fasoli, Diego and Faugeras, Olivier and Touboul,
  Jonathan}.
\newblock {Mean-field description and propagation of chaos in networks of
  {{H}}odgkin-{{H}}uxley and {{F}}itz{{H}}ugh-Nagumo neurons}.
\newblock {\em {J. Math. Neurosci.}}, {2}:{10}, {2012}.

\bibitem{bossy2015clarification}
{Bossy, Mireille and Faugeras, Olivier and Talay, Denis}.
\newblock {Clarification and Complement to ``Mean-Field Description and
  Propagation of Chaos in Networks of {{H}}odgkin--{{H}}uxley and
  {{F}}itz{{H}}ugh--{{N}}agumo Neurons''}.
\newblock {\em {J. Math. Neurosci.}}, {5}:{19}, {2015}.

\bibitem{liu2015stochastic}
{Liu, Wei and R\"ockner, Michael}.
\newblock {\em Stochastic partial differential equations: an introduction}.
\newblock Universitext. Springer, Cham, 2015.

\bibitem{luccon2014mean}
{Lu{\c c}on, Eric and Stannat, Wilhelm}.
\newblock {Mean field limit for disordered diffusions with singular
  interactions}.
\newblock {\em {The Annals of Applied Probability}}, {24}({5}):{1946--1993},
  {2014}.

\bibitem{quininao2015limits}
{Qui{\~n}inao, Cristobal and Touboul, Jonathan}.
\newblock {Limits and dynamics of randomly connected neuronal networks}.
\newblock {\em {Acta Applicandae Mathematicae}}, {136}({1}):{167--192}, {2015}.

\bibitem{song2012numerical}
{Song, Minghui and Yu, Hui}.
\newblock {Numerical solutions of stochastic differential delay equations with
  Poisson random measure under the generalized Khasminskii-type conditions}.
\newblock In {\em {Abstr. Appl. Anal.}} {Hindawi Publishing Corporation}, {Art.
  ID 127397, 2012}.

\bibitem{touboul2012limits}
{Touboul, Jonathan}.
\newblock {Limits and dynamics of stochastic neuronal networks with random
  heterogeneous delays}.
\newblock {\em {Journal of Statistical Physics}}, {149}({4}):{569--597},
  {2012}.

\bibitem{touboul2014propagation}
{Touboul, Jonathan}.
\newblock {Propagation of chaos in neural fields}.
\newblock {\em {The Annals of Applied Probability}}, {24}({3}):{1298--1328},
  {2014}.
\newblock {Update on arXiv: arXiv:1108.2414v6}.

\bibitem{touboul2014spatially}
{Touboul, Jonathan}.
\newblock {Spatially extended networks with singular multi-scale connectivity
  patterns}.
\newblock {\em {Journal of Statistical Physics}}, {156}({3}):{546--573},
  {2014}.

\bibitem{von2010existence}
{von Renesse, Max and Scheutzow, Michael}.
\newblock {Existence and uniqueness of solutions of stochastic functional
  differential equations}.
\newblock {\em {Random Operators and Stochastic Equations}},
  {18}({3}):{267--284}, {2010}.

\bibitem{wu2013wiener}
{Wu, Jing}.
\newblock {On Wiener-Poisson type multivalued stochastic differential equations
  with non-Lipschitz coefficients}.
\newblock {\em {Acta Mathematica Sinica. English Series}},
  {29}({4}):{675--690}, {2013}.

\bibitem{xu2015existence}
{Xu, Yong and Pei, Bin and Guo, Guobin}.
\newblock {Existence and stability of solutions to non-Lipschitz stochastic
  differential equations driven by L{\'e}vy noise}.
\newblock {\em {Applied Mathematics and Computation}}, {263}:{398--409},
  {2015}.

\end{thebibliography}

\end{document}